\documentclass{amsart}
\usepackage{graphicx} 
\usepackage[T1]{fontenc}

\usepackage[utf8]{inputenc}

\usepackage{amsmath}
\usepackage{mathtools}
\usepackage{amsfonts}
\usepackage{amssymb}
\usepackage{tikz-cd}

\DeclareMathOperator{\vol}{vol}

\DeclareMathOperator{\SL}{SL}

\DeclareMathOperator{\ASL}{ASL}
\DeclareMathOperator{\Aut}{Aut}

\DeclareMathOperator{\Hom}{Hom}

\DeclareMathOperator{\id}{id}

\usepackage{xcolor}

\newtheorem{prop}{Proposition}[section]
\newtheorem{thm}[prop]{Theorem}

\newtheorem{lem}[prop]{Lemma}
\newtheorem{cor}[prop]{Corollary}
\newtheorem{conj}[prop]{Conjecture}
\newtheorem{defn}[prop]{Definition}

\title{${\ASL_n}(\mathbb Z)$ invariant random subsets of $\mathbb Z^n$}

\author{Miko{\l}aj Fr\k{a}czyk}
\address{\parbox{\linewidth}{Faculty of Mathematics and Computer Science, Jagiellonian University, ul. Łojasiewicza 6, 30-348 Krak{\'o}w, Poland
}}
\email{mikolaj.fraczyk@uj.edu.pl}
\author{Simon Machado}
\address{\parbox{\linewidth}{
ETH Zurich, Ramistrasse 101, Zurich, Switzerland}}
\email{smachado@ethz.ch}
\date{\today}

\begin{document}

\begin{abstract}
We classify measures on $\{0,1\}^{\mathbb{Z}^d}$, $d \geq 3$, the space of subsets of $\mathbb{Z}^d$, which are invariant under all affine special linear transformations. In other words, we classify simple point processes on $\mathbb{Z}^d$ whose law is invariant under affine special linear transformations.

We show that every such process is built from a random equivariant polynomial together with independent random sampling, a higher-order generalisation of the cut-and-project method: a random polynomial map is drawn from a distribution invariant under a natural action of $\SL_d(\mathbb{Z})$, each site is then retained independently with a probability determined by a measurable function of the polynomial's value, and the classical cut-and-project construction is recovered in the degree-one case. As a corollary, when the underlying $\mathbb{Z}^d$-action is weakly mixing the measure must be a convex combination of Bernoulli shifts, in the spirit of de Finetti's theorem on exchangeable processes. Our theorem also makes precise how the Howe--Moore theorem fails for the pair $(\ASL_d(\mathbb{Z}), \SL_d(\mathbb{Z}))$.

Motivated by this classification, we formulate a conjecture for $\ASL_d(\mathbb{R})$-invariant point processes on $\mathbb{R}^d$, predicting that any such set decomposes into a Poisson part and a quasicrystal part. The proofs rely on the interaction between the Host--Kra theory of characteristic factors, Zimmer's theory of dynamical cocycles of simple Lie groups, and the dynamics of $\SL_d(\mathbb{Z})$-actions on homogeneous spaces.
\end{abstract}

\maketitle
\section{Introduction}
\subsection{Invariant random subsets} An $\ASL_d(\mathbb Z)$-invariant random subset of $\mathbb Z^d$ is a random subset of $\mathbb Z^d$ whose distribution is an $\ASL_d(\mathbb Z)$-invariant probability measure on the space $\{0,1\}^{\mathbb Z^d}$, where $\ASL_d(\mathbb Z)$ acts by translations. Our goal in this work is to understand and possibly classify all such sets. 
For $d=1$, the only condition is invariance under translations, so the problem is tantamount to understanding shift-invariant measures on $\{0,1\}^{\mathbb Z}$, which form a vast, unclassifiable space \cite{foreman2011conjugacy,hjorth2001invariants}. Similarly, it is hopeless to expect any clean description of translation-invariant measures on $\{0,1\}^{\mathbb Z^d}$ for $d\geq 2$. Invariance under a richer group of transformations of $\mathbb{Z}^d$, such as $\ASL_d(\mathbb Z)$, should in principle force some rigidity of invariant measures. In the extreme case, by a classical theorem of de Finetti \cite{de1937prevision}, the only ergodic measures on $\{0,1\}^{\mathbb Z^d}$ invariant under \emph{all} permutations of $\mathbb Z^d$ correspond to Bernoulli random subsets. We refer to \cite{diaconis1977finite,diaconis1980finite,kirsch2019elementary,gavalakis2021information} for a more modern treatment of de Finetti's theorem. 

As the main result of this paper, we construct\footnote{Technically, we construct a dense family in the weak-* topology, so one must take the closure to obtain all of them.} \emph{all} $\ASL_d(\mathbb Z)$-invariant random subsets of $\mathbb Z^d$ for $d\geq 3$. While we believe the case $d=2$ admits a similar description, our current methods use super-rigidity and property (T) for the group $\SL_d(\mathbb Z)$ which hold only for $d\geq 3$. Even though the group $\ASL_d(\mathbb Z)$ is much, much smaller than the group of all permutations of $\mathbb Z^d$, the set of ergodic $\ASL_d(\mathbb Z)$-invariant measures is surprisingly rigid. All of them are attainable by a combination of an explicit algebraic construction and independent random sampling. The algebraic construction is slightly involved and reminiscent of cut-and-project sets, which are used to produce quasi-crystals in $\mathbb R^d$ \cite{ruhr2023classification}. Before we delve into the details, we will motivate our construction with several examples, in order of (subjectively) increasing complexity. 

Perhaps the simplest $\ASL_d(\mathbb Z)$-invariant random subsets can be obtained as random translates of $n\mathbb Z^d$, where $n \in\mathbb N$. These fit into the larger family of periodic $\ASL_d(\mathbb Z)$-invariant subsets of $\mathbb Z^d$ (see Figure \ref{fig-periodic}). Indeed, any periodic subset of $\mathbb Z^d$ has a finite orbit under $\ASL_d(\mathbb Z),$ so we can get an invariant random subset by uniformly sampling the elements of the orbit. 
\begin{figure}[h]
    \centering
    \includegraphics[width=0.7\textwidth]{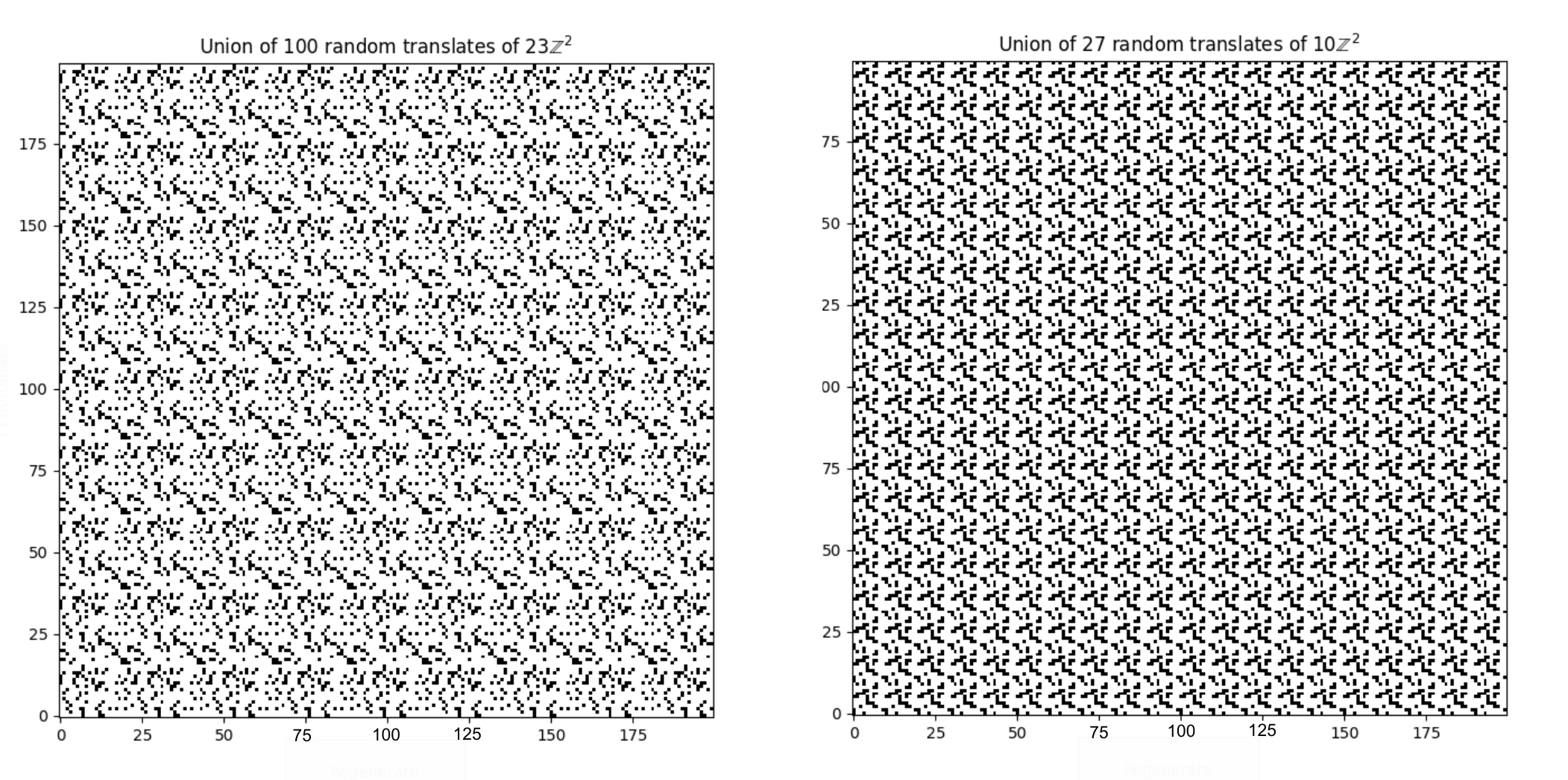}
    \caption{Periodic examples}
    \label{fig-periodic}
\end{figure}

Next in order of complexity are the sets obtained as preimages of random affine characters of $\mathbb Z^d$. For example, choose uniform random $\xi_0,\xi_1,\ldots, \xi_d\in \mathbb R/\mathbb Z$ and set 
$$S_1:=\{t\in \mathbb Z^d\mid \xi_0+\sum_{i=1}^d t_i\xi_i\in [0,0.5]+\mathbb Z\}.$$ 
It is not hard to check that the distribution of $S_1$ is invariant under $\ASL_d(\mathbb Z).$ These sets are no longer periodic, but weakly almost-periodic \cite{guiheneuf2017minkowski}. As we can see in Figure \ref{fig-deg1}, the set $S_1$ consists of layers of points concentrated around hyperplanes, typically of irrational slope. 

\begin{figure}[h]
    \centering
    \includegraphics[width=0.7\textwidth]{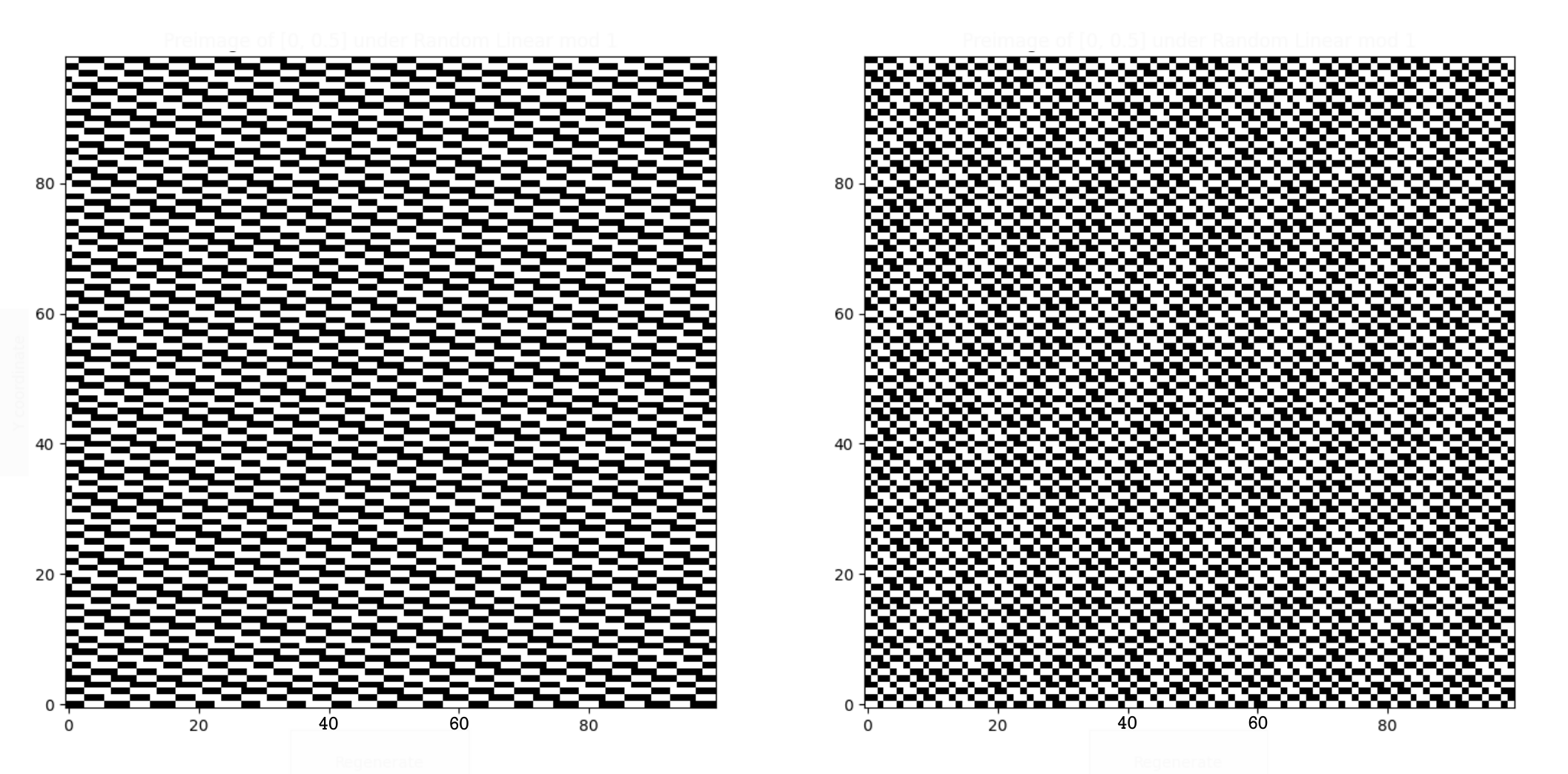}
    \caption{Instances of $S_1$.}
    \label{fig-deg1}
\end{figure}

The above construction can be repeated verbatim using preimages of random polynomial functions $\mathbb Z^d\to \mathbb R/\mathbb Z$. In Figure \ref{fig-deg2}, we have collected several instances of the random set constructed using degree two polynomials:
\begin{align*}
S_2=&\{t\in \mathbb Z^2\mid \xi_1t_1^2+\xi_2t_1t_2+\xi_3t_2^2+\xi_4t_1+\xi_5t_2 +\xi_6\in [0,0.5]+\mathbb Z\},\\ 
& \text{ where } \xi_1,\ldots,\xi_6\sim {\tt Unif}([0,1]) \text{ are independent}.
\end{align*}

\begin{figure}[h]
    \centering
    \includegraphics[width=\textwidth]{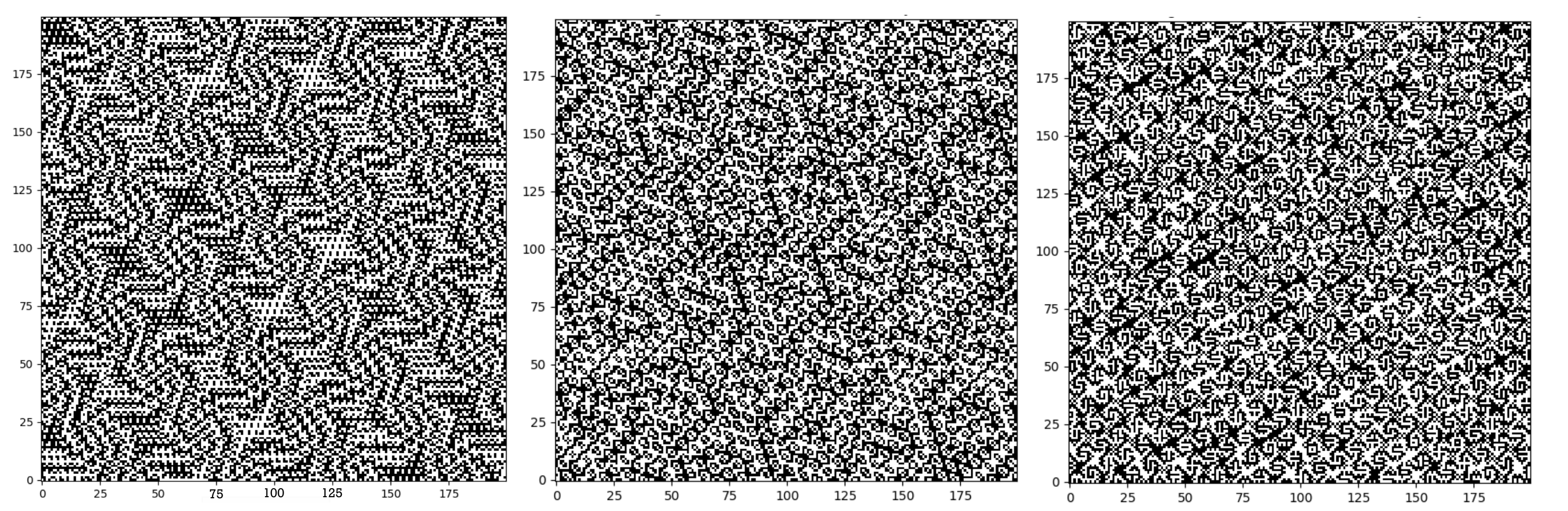}
    \caption{Instances of $S_2$}
    \label{fig-deg2}
\end{figure}
Instances of $S_2$ exhibit curious, somewhat quasi-periodic patterns.
For any degree $k\geq 3$, we can similarly define the set $S_k$ as the preimage of $[0,0.5]+\mathbb Z$ under a Haar-random degree-$k$ polynomial in $d$ variables, with coefficients in $\mathbb R/\mathbb Z$. Contrary to degree $2$, for degrees $k\geq 3$ the set $S_k$ loses its apparent structure and is hard to visually distinguish from a Bernoulli random set, as illustrated in Figure \ref{fig-deg3} \footnote{We invite the reader to guess which is which; the answer is in the source.}. In Bernoulli random sets, each element is contained independently with probability $p$. These are of course the most random examples, where there are no correlations between the structure of our set in disjoint subsets of the space.

\begin{figure}[h]
    \centering
    \includegraphics[width=\textwidth]{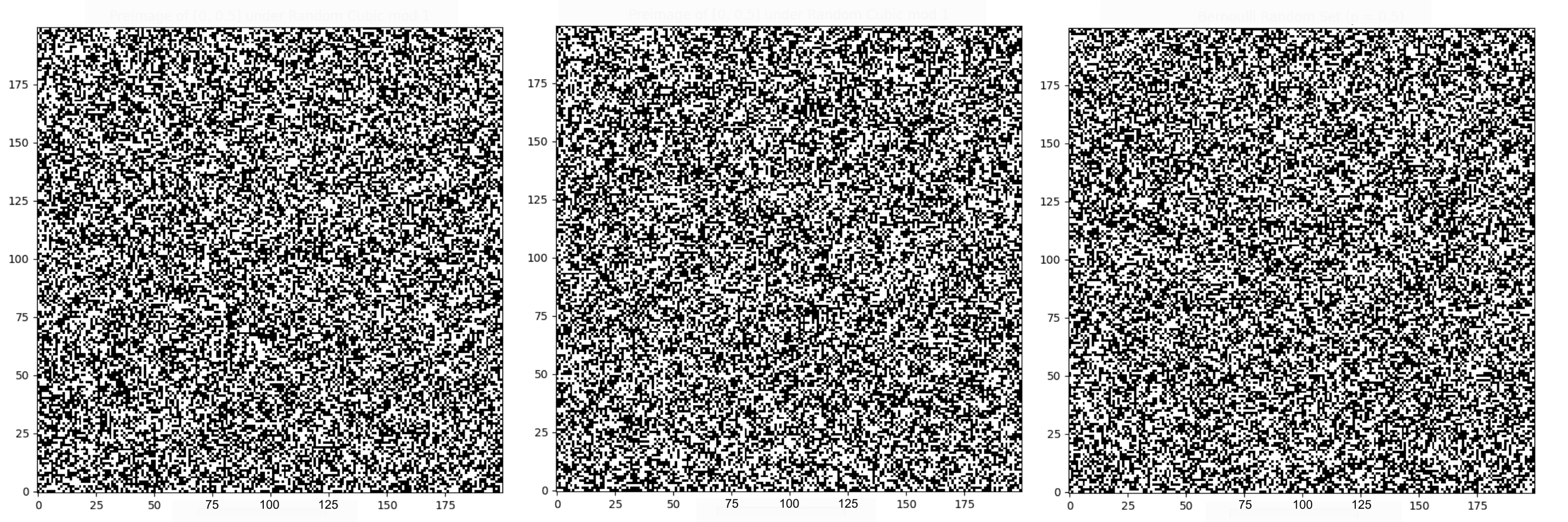}
    \caption{Two instances of $S_3$ and one instance of a Bernoulli random set.}
    \label{fig-deg3}
\end{figure}

Nevertheless, we can distinguish sets $S_k$ for $k\geq 3$, as well as the Bernoulli random examples, by considering their finer arithmetic statistics, captured by higher Gowers uniformity norms. Indeed, in the $S_3$ examples, the intersection of $S_3$ with eight consecutive terms in an arithmetic progression has slightly different distribution than a Bernoulli random subset\footnote{We don't claim here that $8$ is the minimal number one should take to see this, but $5$ is not enough in the particular case of preimages of $[0,0.5].$}. 
The mention of Gowers norms is not accidental, as we found that the study of $\ASL_d(\mathbb Z)$-invariant random subsets of $\mathbb Z^d$ is deeply connected to additive combinatorics, multiple recurrences, and the structure theory of characteristic factors \cite{HostKra, TaoZieglerConcat,ZieglerNonConventional,JamneshanMachado}.

One can interpolate between the constructions explained above, for example, by taking independent unions, intersections, or by using random $\mathbb R/\mathbb Z$-valued polynomials to first construct $[0,1]$-valued invariant random fields $\phi$ on $\mathbb Z^d$ and sample each $t$ independently with probability $\phi(t)$. We can also replace the torus $\mathbb R/\mathbb Z$ with a higher-dimensional torus, possibly endowed with a non-trivial action of $\SL_d(\mathbb Z).$ Our main result states that these constructions are essentially how \emph{all} $\ASL_d(\mathbb Z)$-invariant random subsets are constructed. 

Write $\mathbb T$ for the torus $\mathbb R/\mathbb Z$. Suppose we are given a homomorphism $\pi\colon \SL_d(\mathbb Z)\to \SL_m(\mathbb Z)$. Let $\ASL_d(\mathbb Z)$ act on $\mathbb T^m$ via $\pi$ evaluated on the linear part of the affine map. Consider the space of polynomial maps $\mathbb{Z}^d \to \mathbb{T}^m$ of total degree at most $k$:

$$ \mathbb{T}^m_{\leq k}[X_1, \ldots, X_d] = \bigl\{(t_1, \ldots, t_d) \mapsto \!\!\sum_{i_1 + \cdots + i_d \leq k} t_1^{i_1} \cdots t_d^{i_d} \, \tau_{(i_1, \ldots, i_d)} : \tau_{(i_1, \ldots, i_d)} \in \mathbb{T}^m \bigr\}.$$

The group $\ASL_d(\mathbb Z)$ admits two commuting actions on $\mathbb{T}^m_{\leq k}[X_1, \ldots, X_d]$. One by precomposition $P \mapsto P \circ \gamma$, and the second via a linear action on the coefficient torus $\mathbb T^m$.

We now fix $k, m \geq 0$ and a closed subgroup $T \subseteq \mathbb{T}^m$. Let $T_{\leq k}[X_1, \ldots, X_d]$ be the space of polynomial maps of degree at most $k$ taking values in the subgroup $T$. We choose a Borel section $s: \SL_m(\mathbb{R})/\SL_m(\mathbb{Z}) \to \SL_m(\mathbb{R})$ and define the Mackey cocycle:
\begin{align*}
    \alpha: \SL_d(\mathbb{Z}) \times \SL_m(\mathbb{R})/\SL_m(\mathbb{Z}) &\longrightarrow \SL_m(\mathbb{Z}) \\
    (\gamma, \bar g) &\longmapsto s(\pi(\gamma) \bar g)^{-1} \, \pi(\gamma) \, s(\bar g).
\end{align*}
The cocycle $\alpha$ takes values in $\SL_m(\mathbb{Z})$, as the projection of the right-hand side to $\SL_m(\mathbb{R})/\SL_m(\mathbb{Z})$ is the identity coset $\SL_m(\mathbb Z)$. We suppose in addition that the range of $\alpha$ stabilises $T$, so that $\alpha$ acts on $T_{\leq k}[X_1, \ldots, X_d]$.

\begin{defn}\label{DefIntro: general polynomial process}
    With the data $(k, m, T, \pi, s)$ as above, let $\nu$ be the Haar probability measure on a closed subgroup of $T_{\leq k}[X_1, \ldots, X_d]$ stable under $P \mapsto \alpha(\gamma, x)(P) \circ \gamma^{-1}$ for every $\gamma \in \SL_d(\mathbb{Z})$, and let $f: T \to [0, 1]$ be a measurable function invariant under the range of $\alpha$. The associated \emph{polynomial point process} $\Lambda(k,m,T,\nu,\pi,s,f)$ on $\mathbb{Z}^d$ is obtained as follows:
    \begin{itemize}
        \item draw a random polynomial $\underline{P}$ according to $\nu$;
        \item independently across $t \in \mathbb{Z}^d$, keep $t$ with probability $f(\underline{P}(t))$.
    \end{itemize}
\end{defn}

The process $\Lambda(k,m,T,\nu,\pi,s,f)$ is $\ASL_d(\mathbb{Z})$-invariant because the function $t\mapsto f(\underline{P}(t))$ is invariant under the action of $\ASL_d(\mathbb Z).$ To state our approximation result, we define the intensity of a $\mathbb Z^d$-invariant random subset $\Lambda$ of $\mathbb Z^d$ as the probability $\mathbb P[0\in \Lambda]$. If $\Lambda$ is an ergodic $\ASL_d(\mathbb Z)$-invariant random subset, the value $\mathbb P[0\in \Lambda]$ almost surely equals the upper and lower density of $\Lambda$. Our main result takes the form of an approximation theorem for $\ASL_d(\mathbb Z)$-invariant random subsets.

\begin{thm}\label{mt-approx}
    Let $\Lambda$ be an $\ASL_d(\mathbb{Z})$-invariant random subset of $\mathbb{Z}^d$ for $d\geq 3$. For any $\varepsilon>0$, there exists a tuple $k, m, T,\nu, \pi,s$, and $f$ as above and an $\ASL_d(\mathbb Z)$-invariant coupling $(\Lambda,\Lambda(k,m,T,\nu, \pi,s,f))$, such that the intensity of the symmetric difference $\Lambda\Delta \Lambda(k,m,T,\nu, \pi,s,f)$ is at most $\varepsilon.$
\end{thm}

This means that, up to a sparse $\ASL_d(\mathbb Z)$-invariant "error set", $\Lambda$ is given by our fairly explicit amalgam of a polynomial construction followed by Bernoulli sampling. Allowing for a small error is unavoidable. Consider the sets:
$$S_k^{\delta}:=\{t\in \mathbb Z^d\mid \underline P(t)\in [0,\delta]\}, \text{ where } \underline{P}\in \mathbb T^m_{\leq k}[X_1,\ldots,X_d] \text{ is Haar-random.}$$
The intensity of $S_k^{\delta}$ is $\delta$, so the union $\bigcup_{k=1}^\infty S_{k}^{1/3^k}$ is a proper $\ASL_d(\mathbb Z)$-invariant random subset (because the sum of the intensities of $S_k^{1/3^k}$ is strictly below $1$). At the same time, it cannot be realized as $\Lambda(k,m,T,\nu,\pi,s,f)$ for any finite $k$.

In specific situations, we can show that a finite $k$ is enough to realize an $\ASL_d(\mathbb Z)$-invariant subset as $\Lambda(k,m,T,\nu,\pi,s,f)$. Let $\mu$ be the distribution of $\Lambda$. The restriction of the action of $\ASL_d(\mathbb Z)$ on $(\{0,1\}^{\mathbb Z^d},\mu)$ to $\mathbb Z^d$ might fail to be ergodic, but for each $\mathbb Z^d$-ergodic component $\mu_0$ of $\mu$, we can consider the characteristic factors $\mathcal Z^i(\{0,1\}^{\mathbb Z^d},\mu_0)$ of the underlying p.m.p. action of $\mathbb Z^d$ (see \cite{JamneshanMachado}). If these factors stabilize at $k$ for almost all ergodic components $\mu_0$, that is $\mathcal Z^i(\{0,1\}^{\mathbb Z^d},\mu_0)=\mathcal Z^k(\{0,1\}^{\mathbb Z^d},\mu_0)$ for all $i\geq k$, then we can represent $\Lambda=\Lambda(k,m,T,\nu, \pi,s,f)$ for some choice of $m,T,\nu, \pi,s,f.$

As a corollary, we can classify all $\ASL_d(\mathbb Z)$-invariant random subsets which are weakly mixing under the action of $\mathbb Z^d$. In this case, all characteristic factors considered above are trivial and hence stabilize at $k=0$.
\begin{cor}\label{Cor: Mixing}
Let $\mu$ be an $\ASL_d(\mathbb Z)$-invariant probability measure on $\{0,1\}^{\mathbb Z^d}$, such that the $\mathbb Z^d$ action on $(\{0,1\}^{\mathbb Z^d},\mu)$ is weakly mixing. Then, $\mu$ is the distribution of a Bernoulli random subset, i.e., $\mu=(p\delta_0+(1-p)\delta_1)^{\mathbb Z^d},$ for some $p\in [0,1].$ 
\end{cor}
We believe the assumption $d\geq 3$ in Theorem \ref{mt-approx} could be relaxed to $d\geq 2$. Currently, $d\geq 3$ is used to apply property (T) and super-rigidity in several places in the proof. At least some of these could be avoided with more work.

\subsection{Discrete Howe-Moore theorem and invariant random subsets of coset spaces}
The classical Howe-Moore theorem \cite{howe1979asymptotic} states that, under mild assumptions, the ergodic p.m.p. actions of a real Lie group $G$ remain ergodic after restricting to a closed unbounded subgroup $H$. The case relevant to this work is $G=\ASL_d(\mathbb R), H=\SL_d(\mathbb R),$ where the Howe-Moore theorem states that any ergodic $\ASL_d(\mathbb R)$ p.m.p. action is also $\SL_d(\mathbb R)$-ergodic. 

The analogous statement does not hold for $\ASL_d(\mathbb Z)$ and $\SL_d(\mathbb Z).$ Indeed, the action of $\ASL_d(\mathbb Z)$ on $\mathbb Z^d/2\mathbb Z^d$ with a uniform measure is ergodic, but restriction to $\SL_2(\mathbb Z)$ fixes the positive measure subset $\{0\}$. The knowledge of $\ASL_d(\mathbb Z)$-invariant random subsets can be used to still say something meaningful about $\SL_d(\mathbb Z)$-invariant subsets in an $\ASL_d(\mathbb Z)$-action. 

\begin{thm}\label{mthm-discreteHM}
Let $(X,\mu)$ be an ergodic p.m.p. action of $\ASL_d(\mathbb Z)$ for $d\geq 3$, and let $U \subset X$ be a positive measure $\SL_2(\mathbb Z)$-invariant subset. Suppose that the action by $\mathbb Z^d$ is weakly mixing. Then, for any finite subset $F\subset \mathbb Z^d$, we have 
$$\mu\left(\bigcap_{t \in F}tU\right)=\mu(U)^{|F|}.$$
\end{thm}


We have faint hope that our techniques can be applied to study the failure of Howe-Moore for more general pairs, like, for example, $(\SL_n(\mathbb Z), \SL_m(\mathbb Z))$ for $m<n.$ This would require an understanding of $\SL_n(\mathbb Z)$-invariant measures on $\{0,1\}^{\SL_n(\mathbb Z)/\SL_m(\mathbb Z)}$. We expect that the "easiest" in this family of problems should be an approximation result for $\SL_d(\mathbb Z)$-invariant subsets of $\mathbb Z^d$ for $d\geq 4$. The grid $\mathbb Z^4$ is a union of affine translates of $\mathbb Z^3$ on which suitable subgroups of $\SL_4(\mathbb Z)$ act like $\ASL_3(\mathbb Z)$ acts on $\mathbb{Z}^3$, so Theorem \ref{mt-approx} already provides some information on this problem, although not enough to get an analogous description of the $\SL_4(\mathbb Z)$-invariant random set on the whole space $\mathbb Z^4$. 

\subsection{Dreams about the continuous analogue}

A point process on $\mathbb R^d$ can be viewed either as a random discrete subset of $\mathbb R^d$ or a random locally finite sum of Dirac masses \cite{daley2003introduction}. These points of view are largely equivalent, but working with measures is technically slightly more convenient. Let us write $\mathcal M(\mathbb R^d)$ for the space of locally finite sums of Dirac masses on $\mathbb R^d$ with the topology of weak-* convergence on compact subsets. 

An $\ASL_d(\mathbb R)$-invariant point process on $\mathbb R^d$ is a process whose distribution is an $\ASL_d(\mathbb R)$-invariant probability measure on $\mathcal M(\mathbb R^d).$ To our knowledge, it was Jens Marklof who first asked for the classification of $\ASL_d(\mathbb R)$-invariant point processes on $\mathbb R^d$. This problem was our initial motivation to work on $\ASL_d(\mathbb Z)$-invariant subsets of $\mathbb Z^d$, in hopes that the solution of the discrete analogue will shed some light on the continuous version. 

We present two examples of $\ASL_d(\mathbb R)$-invariant point processes on $\mathbb R^d$. Their constructions can then be combined to produce many more examples. The first is a Poisson point process \cite[Chapter 3]{daley2003introduction}, which is the continuous analogue of a Bernoulli random set. Given $\eta>0$, the Poisson point process $\Pi_\eta$ on $\mathbb R^d$ is the unique random subset $\Pi_\eta$ of $\mathbb R^d$ satisfying the following properties: 
\begin{enumerate}
    \item For any measurable subset $A\subset \mathbb R^d$, the expected number of points in $\Pi_\eta\cap A$ is $\eta {\vol}(A).$
    \item For any disjoint measurable subsets $A,B\subset \mathbb R^d$, the numbers of points in $\Pi_\eta\cap A$ and $\Pi_\eta\cap B$ are independent. 
\end{enumerate}
The Poisson point process $\Pi_\eta$ depends only on the Lebesgue measure and $\eta$; it is therefore invariant under any measure-preserving map. In particular, $\Pi_\eta$ is an example of an $\ASL_d(\mathbb R)$-invariant point process on $\mathbb R^d.$

The second family of examples comes from cut-and-project sets, which were studied and classified in \cite{ruhr2023classification}. Pick $m\geq d$ and fix a decomposition $\mathbb R^{m}=\mathbb R^d \oplus \mathbb R^{m-d}.$ We select a compact window $W\subset \mathbb R^{m-d}$ of positive Lebesgue measure. Let $\mathcal L$ be a translate of a lattice in $\mathbb R^{m}$ and define $\Lambda(\mathcal L,W)$ as the projection of $\mathcal L\cap (\mathbb R^d\times W)$ to $\mathbb R^d$. The sets $\Lambda(\mathcal L,W)$ are always discrete; we call them \emph{cut-and-project} sets. Cut-and-project sets are also known as \emph{quasi-crystals} or \emph{model sets} \cite{baake2013aperiodic, senechal1995quasicrystals}. 

Suppose now that $\ASL_d(\mathbb R)$ acts on $\mathbb R^m$, and preserves the decomposition $\mathbb R^d\oplus \mathbb R^{m-d}$ in such a way that the action on $\mathbb R^d$ is the standard one. Then, taking the intersection with $\mathbb R^d\times W$ and the projection to $\mathbb R^d$ are equivariant with respect to the action of $\ASL_d(\mathbb R)$. In particular, if the lattice $\mathcal L$ is $\ASL_d(\mathbb R)$-invariant random, the cut-and-project set $\Lambda(\mathcal L, W)$ is an $\ASL_d(\mathbb R)$-invariant point process on $\mathbb R^d$. In \cite{ruhr2023classification}, R\"uhr, Smilansky, and Weiss used Ratner's measure classification to describe all $\ASL_d(\mathbb R)$-invariant point processes supported on cut-and-project sets. 

Given these two constructions, we can interpolate them by taking independent unions or by further percolating an existing $\ASL_d(\mathbb R)$-invariant point process in an $\ASL_d(\mathbb R)$-invariant way. At present, we don't know any examples that would be disjoint from these two families. To make it precise, we don't know a counterexample to the following provocative conjecture. 
\begin{conj}\label{conj-hard}
Let $\Lambda$ be an $\ASL_d(\mathbb R)$-invariant point process. Then, either $\Lambda$ is a Poisson point process, or there exists an $\ASL_d(\mathbb R)$-invariant random cut-and-project set $\mathcal{S}$ and an $\ASL_d(\mathbb R)$-invariant coupling $(\Lambda,\mathcal{S})$, such that $\Lambda\cap \mathcal{S}\neq \emptyset$ almost surely.
\end{conj}

In a sense, Theorem \ref{mt-approx} confirms the discrete version of this conjecture. It implies that either $\Lambda\subset \mathbb Z^d$ is Bernoulli random, or it correlates with an "algebraically defined" invariant random subset. The polynomial sets $S_1$ are in fact cut-and-project sets. Indeed, choose uniform random $\tilde\xi_0,\tilde \xi_1,\ldots, \tilde \xi_d\in [0,1]$ and consider a lattice translate 
$\mathcal L=\{(t_1,t_2,\ldots, t_d, z+\tilde\xi_0+\sum_{i=1}^dt_i\tilde\xi_i) \mid t_i,z\in \mathbb Z\}.$ The projection of $\mathcal L\cap (\mathbb R^d\times [0,1/2])$ to $\mathbb R^d$ is exactly the set $S_1$. In the discrete case, we have the "higher" polynomial examples $S_k, k\geq 2$. It would be interesting to determine if there are meaningful continuous analogues of these sets in $\mathbb R^d,$ which would not always be contained in a cut-and-project set. If these "higher cut-and-project sets" existed, they would be promising counter-examples to Conjecture \ref{conj-hard}.

We conclude this section by stating a conjectural analogue of Corollary \ref{Cor: Mixing}, which could be potentially more approachable. 
\begin{conj}
Let $\Lambda$ be an $\ASL_d(\mathbb R)$-invariant point process with distribution $\mu$. Suppose that $(\mathcal M(\mathbb R^d),\mu)$ is weakly mixing as an action of $\mathbb R^d$. Then $\Lambda$ is a Poisson point process. 
\end{conj}
For cut-and-project examples, the $\mathbb R^d$ action on $(\mathcal M(\mathbb R^d),\mu)$ fails to even be ergodic, so the conjecture passes this basic sanity check.


\subsection{Outline of the proof}

We prove our main result (Theorem \ref{mt-approx}) by translating it into a question in ergodic theory. The proof then descends through a sequence of factors, each more structured than the last. It rests on a dichotomy: either $\underline{X}$ mixes, in which case it is Bernoulli, or it does not, in which case it carries algebraic structure we can pin down.

We are given a standard probability space $(X,\nu)$ with an $\ASL_d(\mathbb{Z})$-ergodic probability measure preserving action, and we wish to understand how the dynamics on $X$ controls the law of the associated point process $\underline{X} \subset \mathbb{Z}^d$. Step by step, we will show that this law is controlled by factors of $X$ with more and more structure.

The starting point is the marginals. For pairwise distinct $t_1, \ldots, t_k \in \mathbb{Z}^d$, we wish to understand 
$$ \mathbb{P}(t_1, \ldots, t_k \in \underline{X}).$$
Invariance under $\ASL_d(\mathbb{Z})$ gives that for every $\gamma \in \SL_d(\mathbb{Z})$, 
$$ \mathbb{P}(t_1, \ldots, t_k \in \underline{X}) = \mathbb{P}(\gamma(t_1), \ldots, \gamma(t_k) \in \underline{X}),$$
and the marginal can be rewritten in a more dynamical way as 
$$\mathbb{P}(t_1, \ldots, t_k \in \underline{X}) = \int_X (t_1 \cdot \mathbf{1}_{0 \in \underline{X}}) \cdots (t_k \cdot \mathbf{1}_{0 \in \underline{X}}) = \int_X (\gamma(t_1) \cdot \mathbf{1}_{0 \in \underline{X}}) \cdots (\gamma(t_k) \cdot \mathbf{1}_{0 \in \underline{X}}).$$
Under sufficient mixing of the action of $\mathbb{Z}^d$ by translation, choosing $(\gamma_n)_{n \geq 0}$ with $\gamma_n(t_i) - \gamma_n(t_j) \rightarrow \infty$ for all $i \neq j$, one could reasonably expect 
$$\int_X (\gamma_n(t_1) \cdot \mathbf{1}_{0 \in \underline{X}}) \cdots (\gamma_n(t_k) \cdot \mathbf{1}_{0 \in \underline{X}}) \rightarrow \left(\int_X \mathbf{1}_{0 \in \underline{X}}\right)^k.$$
In other words, under mixing assumptions $\underline{X}$ should be Bernoulli. This is indeed what we prove (Corollary \ref{Cor: Mixing}), by choosing the $\gamma_n$ along a judicious unipotent sequence and averaging, turning the computation above into an ergodic-like average.

But the approach yields much more. Exploiting Host--Kra--Ziegler theory \cite{HostKra,ZieglerNonConventional} and concatenation results of Tao--Ziegler \cite{TaoZieglerConcat}, we show in \S \ref{Subsection: Reduction to Host--Kra} that the marginals $\mathbb{P}(t_1, \ldots, t_k \in \underline{X})$ are controlled by Host--Kra--Gowers norms $\|\cdot\|_{U^{d(k)}(X)}$ of the function $f := \mathbf{1}_{0 \in \underline{X}}$. Dynamically, this means we only need to understand the conditional expectation of $f$ with respect to the $d(k)$-th \emph{characteristic factor} $Z^{d(k)}(X)$ of $X$ — the series of $\ASL_d(\mathbb{Z})$-factors that precisely captures the \emph{lack} of mixing of $X$ under the action of $\mathbb{Z}^d$, see \S \ref{Subsection: Reduction to Host--Kra}.

 We then begin investigating in \S \ref{Section: Bundles of nilsystems} the projections of $f$ to the characteristic factors $Z^k(X)$. Write therefore $f_k := \mathbb{E}[f \mid Z^k(X)]$. This is the function we wish to control, and it has two key features. First, it is invariant under $\SL_d(\mathbb{Z})$. Second, by Host--Kra--Ziegler theory \cite{HostKra,ZieglerNonConventional} and its non-ergodic extension due to Jamneshan and the second author \cite{JamneshanMachado} - recalled in \S \ref{Subsection: Host--Kra--Ziegler theory for non-ergodic actions} - $Z^k(X)$ \emph{as a $\mathbb{Z}^d$-system} is an inverse limit of \emph{nilfactors}, that is, systems built from homogeneous spaces $N/\Gamma$ for nilpotent Lie groups $N$ of step at most $k$ and lattices $\Gamma \subset N$. The theory however produces only $\mathbb{Z}^d$-nilfactors rather than $\ASL_d(\mathbb{Z})$-factors: only the translation action descends to these structured factors. A soft counting argument involving $f_k$ nevertheless asserts that one can find nilfactors of $Z^k(X)$ on which the action of all of $\ASL_d(\mathbb{Z})$ \emph{almost} descends. If $\pi: X \rightarrow Y$ is such a factor, then $\pi \circ \gamma$ and $\pi$ are almost equal for all $\gamma \in \ASL_d(\mathbb{Z})$, which means that the joining $(\pi, \pi \circ \gamma)$ is \emph{almost} a graph joining. Property (T) of $\ASL_d(\mathbb{Z})$ then promotes almost-invariance into genuine invariance - the same mechanism by which property (T) implies finite generation. This is the one place we crucially use $d \geq 3$. Every other appearance of higher rank in the proof is for convenience.

Once we have reduced the study of $f$ to its projection onto $\ASL_d(\mathbb{Z})$-nilfactors, we can appeal to homogeneous dynamics, which we do in \S \ref{Section: Structure of invariant functions}. The problem reduces to understanding a dynamical cocycle
$$ \alpha: \SL_d(\mathbb{Z}) \times X \rightarrow \Aut(N) $$
together with the invariance of $f$ under $\SL_d(\mathbb{Z})$ and $\alpha$. Leveraging directions (i.e. one-parameter subgroups) in $N$ that are contracted by $\alpha$, we show that the stabiliser of the projection of $f$ to the given nilfactor is large. This in turn enables us to consider only $N$'s with \emph{abelian} connected component. Higher rank is used here too, but only for convenience: the essential tool is in reality \emph{relative property (T)} of $\mathbb{Z}^d$ in $\ASL_d(\mathbb{Z})$, which holds for $d = 2$ as well.

Finally, in \S \ref{Section: From nilsystems to polynomials}, having reduced to $N = \mathbb{Z}^d \ltimes \mathbb{T}^m$ for some $m$, we write the action of $\ASL_d(\mathbb{Z})$ explicitly in terms of $1$-cocycles for the (unipotent) action of $\mathbb{Z}^d$ on $\mathbb{T}^m$. Unipotence forces these cocycles to be polynomials of bounded degree. The family of cocycles we encounter is then \emph{precisely} the random polynomial $\underline{P}$ from Theorem \ref{mt-approx}, and the projection of $f$ onto this nilfactor corresponds to the relevant function.
 
\subsection{Acknowledgement}
MF would like to thank Asaf Katz for pointing out counterexamples to a more general version of Conjecture \ref{conj-hard} on $\SL_d(\mathbb R)$ invariant point processes.
MF was supported by the Dioscuri
program, initiated by the Max Planck Society, jointly managed with the National Science Center
in Poland, and mutually funded by Polish the Ministry of Education and Science and the German
Federal Ministry of Education and Research. SM acknowledges support from a Hermann Weyl Instructorship at the Institute for Mathematical Research (FIM) at ETH Zurich.

For the purpose of open access, the authors have
applied a CC BY public copyright licence to any author accepted manuscript arising from this
submission.

\section{Notation, conventions and preliminaries}

Throughout, given a group $\Gamma$, we say that $(X, \mu)$ is a $\Gamma$-system if $X$ is a standard Borel space equipped with a measurable $\Gamma$-action preserving the Borel probability measure $\mu$. A factor $\pi: (X, \mu) \to (Y, \nu)$ is a $\Gamma$-equivariant measurable map between $\Gamma$-systems such that $\nu = \pi_* \mu$. Given a factor $\pi$, we naturally identify $L^\infty(Y)$ with a sub-algebra of $L^\infty(X)$ via $f \mapsto f \circ \pi$. By Mackey's point realisation theorem \cite{MackeyPointRealization}, every $\Gamma$-invariant von Neumann sub-algebra of $L^\infty(X)$ closed under complex conjugation arises this way. Given two factors $(X, \mu) \to (Y_i, \nu_i)$ for $i = 1, 2$, the intersection $L^\infty(Y_1) \cap L^\infty(Y_2)$ is again a $\Gamma$-invariant sub-algebra and therefore defines a factor of $X$, which we denote $Y_1 \wedge Y_2$. By construction, $L^\infty(Y_1 \wedge Y_2) = L^\infty(Y_1) \cap L^\infty(Y_2)$.

\subsection{Dynamical cocycles and compactness criterion}

The notion of \emph{dynamical cocycle} underpins much of this article. Let $\Gamma$ be a group, $(X, \nu)$ an ergodic $\Gamma$-system, and $\Delta$ a locally compact second countable group. A Borel map
$$ \alpha: \Gamma \times X \to \Delta$$
is a \emph{cocycle} if for all $\gamma_1, \gamma_2 \in \Gamma$ and $\nu$-a.e.\ $x \in X$,
$$ \alpha(\gamma_1 \gamma_2, x) = \alpha(\gamma_1, \gamma_2 \cdot x) \, \alpha(\gamma_2, x).$$
Two cocycles $\alpha, \beta: \Gamma \times X \to \Delta$ are \emph{cohomologous (in $\Delta$)} if there is a measurable map $\phi: X \to \Delta$ such that for all $\gamma \in \Gamma$ and $\nu$-a.e.\ $x$,
$$ \alpha(\gamma, x) = \phi(\gamma \cdot x) \, \beta(\gamma, x) \, \phi(x)^{-1}.$$
A cocycle is a \emph{coboundary} if it is cohomologous to the trivial cocycle $\gamma, x \mapsto e$.

Cocycles arise pervasively in dynamics, and we will use a number of properties drawn from Zimmer's superrigidity theory \cite{ErgodicTheoryZimmer} and from the Host--Kra--Ziegler theory of nilsystems \cite{HostKra, ZieglerNonConventional}. We collect general properties for repeated use.

\begin{lem}\label{Lemma: Compactness criterion}
    Let $\Delta$ be a locally compact second countable group, $\Gamma$ a group acting ergodically on a standard probability space $(X, \nu)$, and $\alpha: \Gamma \times X \to \Delta$ a cocycle. For Borel $B \subset X$, set
    $$ E(x, B) := \{\alpha(\gamma, x) : \gamma \cdot x \in B\} \subset \Delta.$$
    If there exists $B$ with $\nu(B) > 0$ such that $E(x, B)$ is relatively compact for $x$ in a set of positive measure, then $\alpha$ is cohomologous to a cocycle taking values in a compact subgroup of $\Delta$.
\end{lem}

\begin{proof}
    For all $\gamma_0 \in \Gamma$ and $\nu$-a.e.\ $x \in X$,
    \begin{align*}
        E(\gamma_0 \cdot x, B)
        &= \{\alpha(\gamma, \gamma_0 \cdot x) : \gamma \gamma_0 \cdot x \in B\} \\
        &= \{\alpha(\gamma' \gamma_0^{-1}, \gamma_0 \cdot x) : \gamma' \cdot x \in B\} && (\gamma' := \gamma \gamma_0) \\
        &= \{\alpha(\gamma', x) \, \alpha(\gamma_0, x)^{-1} : \gamma' \cdot x \in B\} \\
        &= E(x, B) \, \alpha(\gamma_0, x)^{-1},
    \end{align*}
    where the third equality uses the cocycle equation in the form $\alpha(\gamma' \gamma_0^{-1}, \gamma_0 \cdot x) = \alpha(\gamma', x) \alpha(\gamma_0^{-1}, \gamma_0 \cdot x)$ together with $\alpha(\gamma_0^{-1}, \gamma_0 \cdot x) = \alpha(\gamma_0, x)^{-1}$.

    Equip the set $\mathcal{K}(\Delta)$ of compact subsets of $\Delta$ with the Hausdorff topology. The map
    $$ f: X \to \mathcal{K}(\Delta), \quad x \mapsto \overline{E(x, B)}^{-1}$$
    is defined on a set of positive measure by hypothesis and, by the computation above, satisfies
    $$ f(\gamma_0 \cdot x) = \alpha(\gamma_0, x) \cdot f(x)$$
    where $\Delta$ acts on $\mathcal{K}(\Delta)$ by left translation. Thus $f$ is $\alpha$-equivariant in the sense of the cocycle reduction lemma \cite[Lem. 4.2.8]{ErgodicTheoryZimmer}. Since the $\Delta$-action on $\mathcal{K}(\Delta)$ has closed orbits with stabilisers given by stabilisers of compact sets — themselves compact subgroups of $\Delta$ — we conclude that $\alpha$ is cohomologous to a cocycle taking values in some compact subgroup of $\Delta$.
\end{proof}

A useful corollary is the following folklore result.

\begin{cor}\label{Cor: coho to compact indep of ambient}
    Let $\Gamma$ be a countable group, $(X, \nu)$ an ergodic $\Gamma$-system, $G$ a locally compact second countable group, and $\alpha: \Gamma \times X \to G$ a cocycle. Suppose that $\alpha$ is cohomologous (in $G$) to a cocycle taking values in a compact subgroup, and that $\alpha(\Gamma \times X) \subset H$ for some closed subgroup $H \subset G$. Then $\alpha$ is cohomologous in $H$ to a cocycle taking values in a compact subgroup.
\end{cor}

\begin{proof}[Proof sketch]
    Apply Lemma \ref{Lemma: Compactness criterion} to $\alpha$ viewed as an $H$-valued cocycle. The hypothesis on $\alpha$ in $G$ implies that the sets $E(x, B)$ are relatively compact in $G$ for some $B$ of positive measure, hence relatively compact in the closed subgroup $H$.
\end{proof}

By passing to a compact extension, a cocycle cohomologous to one with values in a compact subgroup may always be assumed to be a coboundary.

\begin{lem}\label{Lemma: From compact to trivial}
    Let $K$ be a compact group, $\Gamma$ a group acting by measure-preserving transformations on a probability space $(X, \nu)$, and $\alpha: \Gamma \times X \to K$ a cocycle. Let $Y := X \rtimes_\alpha K$ denote the compact extension: the set $X \times K$ equipped with the product measure $\nu \times \mu_K$ and the $\Gamma$-action
    $$ \gamma \cdot (x, k) := (\gamma \cdot x, \alpha(\gamma, x) \, k).$$
    Then:
    \begin{enumerate}
        \item the natural factor map $Y \to X$ is relatively ergodic if and only if $\alpha$ is not cohomologous to a cocycle taking values in a proper closed subgroup of $K$;
        \item the cocycle $\beta: \Gamma \times Y \to K$ defined by $\beta(\gamma, (x, k)) := \alpha(\gamma, x)$ is a coboundary.
    \end{enumerate}
\end{lem}

\begin{proof}
    Item (1) is \cite[Prop. 4.7]{Austin}. Item (2) is folklore; we recall the short proof. Define
    $$ \phi: Y \to K, \qquad (x, k) \mapsto k.$$
    Then
    $$ \phi(\gamma \cdot (x, k)) = \phi(\gamma \cdot x, \alpha(\gamma, x) \, k) = \alpha(\gamma, x) \, k = \alpha(\gamma, x) \, \phi(x, k),$$
    that is, $\beta(\gamma, y) = \phi(\gamma \cdot y) \phi(y)^{-1}$, so $\beta$ is a coboundary.
\end{proof}

\subsection{Finiteness properties of dynamical cocycles}\label{Subsection: Finiteness properties of dynamical cocycles}

We turn to results which, although general, are tailored to the groups we will encounter. They will be crucial in \S \ref{Section: Bundles of nilsystems}. We begin with a finiteness statement for cocycles of property (T) groups; we refer to \cite{ErgodicTheoryZimmer, PropertyT} for background.

\begin{prop}\label{Proposition: cocycles of prop (T) groups are finitely generated}
    Let $\Gamma$ be a property (T) group acting ergodically on a probability space $(X, \nu)$, $\Delta$ a countable discrete group, and $\alpha: \Gamma \times X \to \Delta$ a cocycle. Then there exists a finitely generated subgroup $H \subset \Delta$ such that $\alpha$ is cohomologous to a cocycle taking values in $H$.
\end{prop}

\begin{proof}
    The strategy parallels the standard proof that property (T) groups are finitely generated. Let $(Q, \kappa)$ be a Kazhdan pair for $\Gamma$: $Q$ is a finite generating set and $\kappa > 0$ is such that any unitary representation of $\Gamma$ admitting a $(Q, \kappa)$-almost invariant unit vector has a non-zero invariant vector.

    Set $\epsilon := \kappa^2 / 4$. By countability of $\Delta$ there is a finite set $F \subset \Delta$ and a Borel set $B \subset X$ with $\nu(B) \geq 1 - \epsilon / (2|Q|)$ such that $\alpha(q, x) \in F$ for every $q \in Q$ and every $x \in B$.

    Let $\pi_F$ denote the quasi-regular representation of $\Delta$ on $\ell^2(\Delta / \langle F \rangle)$, and let $V$ be the unitary representation of $\Gamma$ on
    $$ V := L^2\bigl(X, \ell^2(\Delta / \langle F \rangle)\bigr)$$
    given by $(\gamma \cdot f)(x) := \pi_F(\alpha(\gamma, x)) \, f(\gamma^{-1} x)$ — the $L^2$-induction of $\pi_F$ along $\alpha$. Let $f_0 \in V$ be the constant function $f_0(x) = \mathbf{1}_{\langle F \rangle}$, the indicator of the trivial coset; it is a unit vector.

    For $q \in Q$ and $x \in B$, $\alpha(q, x) \in F \subset \langle F \rangle$, so
    $$(q \cdot f_0)(x) = \pi_F(\alpha(q, x)) \mathbf{1}_{\langle F \rangle} = \mathbf{1}_{\alpha(q, x) \langle F \rangle} = \mathbf{1}_{\langle F \rangle} = f_0(x).$$
    Hence
    $$\|q \cdot f_0 - f_0\|_V^2 = \int_{X \setminus q^{-1}(B) \cap B} \|(q \cdot f_0)(x) - f_0(x)\|_{\ell^2}^2 \, d\nu(x) \leq 2 \nu(X \setminus B) \leq \epsilon / |Q|.$$
    Summing over $q \in Q$ gives $\sum_{q \in Q} \|q \cdot f_0 - f_0\|^2 \leq \epsilon < \kappa^2$, so $f_0$ is $(Q, \kappa)$-almost invariant. By property (T) there is a non-zero $\Gamma$-invariant vector $f \in V$, which we may take to be a unit vector.

    For $\nu$-a.e.\ $x$, $f(x) \in \ell^2(\Delta / \langle F \rangle)$ is a non-zero $\ell^2$-function on a discrete set, hence attains its supremum on a non-empty finite set
    $$ E(x) := \{\delta \langle F \rangle \in \Delta / \langle F \rangle : f(x)(\delta \langle F \rangle) = \max f(x)\}.$$
    The $\Gamma$-invariance of $f$, $\pi_F(\alpha(\gamma, x)) f(\gamma^{-1} x) = f(x)$, gives $f(\gamma^{-1} x)(\delta \langle F \rangle) = f(x)(\alpha(\gamma, x) \delta \langle F \rangle)$, so
    $$ E(\gamma^{-1} x) = \alpha(\gamma, x)^{-1} \cdot E(x).$$
    Thus $E$ is an $\alpha$-equivariant Borel map from $X$ to the countable set of finite subsets of $\Delta / \langle F \rangle$. By the cocycle reduction lemma \cite[Lem. 4.2.8]{ErgodicTheoryZimmer}, $\alpha$ is cohomologous to a cocycle taking values in $\mathrm{Stab}_\Delta(E_0)$ for some finite $E_0 = \{\delta_1 \langle F \rangle, \ldots, \delta_r \langle F \rangle\} \subset \Delta / \langle F \rangle$. This stabiliser is generated by $\delta_1, \ldots, \delta_r$ together with $\langle F \rangle$, hence is finitely generated.
\end{proof}

When working with matrix groups we may combine this with further properties of cocycles of higher rank lattices — in particular for $\SL_d(\mathbb{Z})$ with $d \geq 3$ — to obtain strong information; we refer to \cite{ErgodicTheoryZimmer} for background on cocycles of these groups.

\begin{prop}\label{Prop: SLd cocycles stabilize a lattice}
    Let $d \geq 3$, $m \geq 1$, and let $(X, \nu)$ be an ergodic $\SL_d(\mathbb{Z})$-system. Let $\Delta \subset \SL_m(\mathbb{Q})$ be a subgroup and $\alpha: \SL_d(\mathbb{Z}) \times X \to \Delta$ a cocycle. Then $\alpha$ is cohomologous in $\Delta$ to a cocycle taking values in a subgroup of $\Delta$ commensurable with $\Delta \cap \SL_m(\mathbb{Z})$.
\end{prop}

\begin{proof}
    Since $d \geq 3$, $\SL_d(\mathbb{Z})$ has property (T). By Proposition \ref{Proposition: cocycles of prop (T) groups are finitely generated}, after replacing $\alpha$ by a cohomologous cocycle we may assume $\Delta$ is finitely generated; in particular $\Delta \subset \SL_m(\mathbb{Z}[\tfrac{1}{M}])$ for some integer $M \geq 1$. Let $p_1, \ldots, p_r$ be the prime divisors of $M$, and consider the diagonal embedding
    $$ \pi: \SL_m(\mathbb{Z}[\tfrac{1}{M}]) \hookrightarrow \prod_{i = 1}^{r} \SL_m(\mathbb{Q}_{p_i}).$$
    Let $L$ be the closure of $\pi(\Delta)$ in $\prod_i \SL_m(\mathbb{Q}_{p_i})$.

    By Zimmer's cocycle superrigidity theorem \cite[Thm. 1.5]{FisherMargulis} and property (T) \cite[Thm. 9.1.1]{ErgodicTheoryZimmer}, $\pi \circ \alpha$ is cohomologous in $\prod_i \SL_m(\mathbb{Q}_{p_i})$ to a cocycle taking values in a compact subgroup. By Corollary \ref{Cor: coho to compact indep of ambient}, $\pi \circ \alpha$ is cohomologous in $L$ to a cocycle taking values in a compact subgroup $K' \subset L$. Since $L$ is locally compact totally disconnected, $K'$ is contained in some open compact subgroup $O \subset L$.

    Let $\phi: X \to L$ realise this cohomology, so that $\phi(\gamma \cdot x)^{-1} \, \pi(\alpha(\gamma, x)) \, \phi(x) \in O$ for $\nu$-a.e.\ $x$ and all $\gamma$. Since $\pi(\Delta)$ is dense in $L$ and $O$ is open, by measurable selection \cite[Thm. A.5]{ErgodicTheoryZimmer} there is a measurable $\psi: X \to \Delta$ with $\phi(x) \in \pi(\psi(x)) \, O$ for $\nu$-a.e.\ $x$. Then
    $$ \pi\bigl(\psi(\gamma \cdot x)^{-1} \, \alpha(\gamma, x) \, \psi(x)\bigr) \in O$$
    for $\nu$-a.e.\ $x$ and all $\gamma$, i.e.\ $\psi(\gamma \cdot x)^{-1} \alpha(\gamma, x) \psi(x) \in \Delta \cap \pi^{-1}(O)$. The subgroup $\Delta \cap \pi^{-1}(O)$ is commensurable with $\Delta \cap \pi^{-1}\bigl(\prod_i \SL_m(\mathbb{Z}_{p_i})\bigr) = \Delta \cap \SL_m(\mathbb{Z})$, since both $O$ and $\prod_i \SL_m(\mathbb{Z}_{p_i})$ are open compact subgroups of $\prod_i \SL_m(\mathbb{Q}_{p_i})$ and any two such are commensurable.
\end{proof}

Combining the previous results, we obtain a structural statement for cocycles into automorphism groups of nilpotent Lie groups commensurating a lattice.

\begin{prop}\label{Prop: cocycles into Aut(N) reduce to lattice stabiliser, baby case}
    Let $N$ be a nilpotent Lie group of the form $\mathbb{Z}^d \ltimes N_0$ with $d \geq 3$ and $N_0$ connected simply connected, and let $\Gamma \subset N$ be a lattice. Let $\Delta \subset \Aut(N)$ be a subgroup such that $\delta(\Gamma)$ is commensurable with $\Gamma$ for every $\delta \in \Delta$. Let $(X, \nu)$ be an ergodic $\SL_d(\mathbb{Z})$-system and $\alpha: \SL_d(\mathbb{Z}) \times X \to \Delta$ a cocycle.

    Then there exist a finite group extension $\pi: \tilde{X} \to X$ and a lattice $\Gamma' \subset N$ such that the lifted cocycle $\tilde{\alpha} := \alpha \circ (\id \times \pi)$ is cohomologous in $\Delta$ to a cocycle taking values in $\{\delta \in \Delta : \delta(\Gamma') = \Gamma'\}$.
\end{prop}

\begin{proof}
    Since $\Gamma$ is a lattice in the nilpotent Lie group $N$, it is finitely generated torsion-free nilpotent, with Malcev completion $\hat{\Gamma}$ a connected simply connected nilpotent Lie group of dimension $m := \dim N_0 + d$. Since each $\delta \in \Delta$ commensurates $\Gamma$, we can identify $\Delta$ with a subgroup of $\Aut(\hat\Gamma) \subset \mathrm{GL}_m(\mathbb{Q})$ by standard results of Malcev relating lattices and rational structure in nilpotent Lie groups \cite[Thm. 2.12]{RaghunathanDiscrete}. By Property (T) and \cite[Thm. 9.1.1]{ErgodicTheoryZimmer}, the composition of $\alpha$ with the natural projection from $\mathrm{GL}_m(\mathbb{Q})$ to the abelian group $\mathrm{GL}_m(\mathbb{Q})/\SL_m(\mathbb{Q})$ is cohomologous in the image of $\Delta$ to the trivial cocycle. Thus, we may assume $\Delta \subset \SL_m(\mathbb{Q})$.

    By Proposition \ref{Prop: SLd cocycles stabilize a lattice}, $\alpha$ is cohomologous in $\Delta$ to a cocycle taking values in a subgroup $\Delta' \subset \Delta$ commensurable with $\Delta \cap \SL_m(\mathbb{Z})$. Under Malcev's lattice--rational-structure correspondence \cite[Thm. 2.12]{RaghunathanDiscrete}, $\Delta'$ permutes finitely many lattices $\Gamma_1, \ldots, \Gamma_r$ commensurable with $\Gamma$. Thus we may assume $\alpha$ takes values in the setwise stabiliser of $\{\Gamma_1, \ldots, \Gamma_r\}$ in $\Delta$, fitting in a finite extension
    $$ 1 \to \mathrm{Stab}_\Delta(\Gamma_1) \cap \cdots \cap \mathrm{Stab}_\Delta(\Gamma_r) \to \mathrm{Stab}_\Delta(\{\Gamma_1, \ldots, \Gamma_r\}) \to S \to 1$$
    where $S$ is a finite group of permutations.

    Composing $\alpha$ with the quotient to $S$ gives a cocycle into a finite group; applying Lemma \ref{Lemma: From compact to trivial} to this finite quotient cocycle produces a finite extension $\tilde{X} \to X$ on which the lifted cocycle becomes a coboundary into $S$. After adjusting by the corresponding transfer function, the lifted cocycle takes values in the kernel $\bigcap_i \mathrm{Stab}_\Delta(\Gamma_i)$, which is contained in $\mathrm{Stab}_\Delta(\Gamma_1)$. Setting $\Gamma' := \Gamma_1$ concludes.
\end{proof}

In fact, we will need a slightly more involved result, which nevertheless holds thanks to the same proof:

\begin{prop}\label{Prop: cocycles into Aut(N) reduce to lattice stabiliser}
    Let $N$ be a nilpotent Lie group of the form $\mathbb{Z}^d \ltimes N_0$ with $d \geq 3$ and $N_0$ connected simply connected, and let $\Gamma \subset N$ be a lattice and $H \subset \Gamma$ be a subgroup. Let $\Delta \subset \Aut(N)$ be a subgroup such that $\delta(H) = H$ and $\delta(\Gamma)$ is commensurable with $\Gamma$ for every $\delta \in \Delta$. Let $(X, \nu)$ be an ergodic $\SL_d(\mathbb{Z})$-system and $\alpha: \SL_d(\mathbb{Z}) \times X \to \Delta$ be a map satisfying the cocycle equation modulo inner automorphisms of $H$ in $\Delta$.

    Then there exist a finite group extension $\pi: \tilde{X} \to X$, a lattice $\Gamma' \subset N$ containing $H$ and a measurable map $\phi: \tilde{X} \rightarrow \Delta$ such that 
    $$ \phi(\gamma\cdot\tilde{x}) \alpha(\gamma,  \pi(\tilde{x})) \phi(\tilde{x})$$ belongs to $\{\delta \in \Delta : \delta(\Gamma') = \Gamma'\}$ for all $\gamma \in \SL_d(\mathbb{Z})$  and almost all $\tilde{x} \in \tilde{X}$.
\end{prop}

\begin{proof}
    The proof of Proposition \ref{Prop: cocycles into Aut(N) reduce to lattice stabiliser} is the same as the proof of Proposition \ref{Prop: cocycles into Aut(N) reduce to lattice stabiliser, baby case} once we make the following observation. 

    Let $\hat{N}$ be the Malcev completion of $N$. Since $\hat{N}$ is connected simply connected the exponential and logarithm maps are well-defined and homeomorphisms. Let $\mathfrak{n}_{\mathbb{Q}}$ denote the $\mathbb{Q}$-Lie algebra spanned by the logarithms of elements of $\Gamma$ and let $\mathfrak{h}_{\mathbb{Q}}$ be the $\mathbb{Q}$-Lie algebra spanned by logarithms of elements of $H$. 

    Since elements of $\Delta$ stabilise $H$, there is an obvious group homomorphism $p: \Delta \rightarrow \mathrm{GL}(\mathfrak{n}_{\mathbb{Q}} / \mathfrak{h}_{\mathbb{Q}})$. By assumption, $p \circ \alpha$ is a cocycle. And by Proposition \ref{Prop: SLd cocycles stabilize a lattice}, we can assume that $p \circ \alpha$ takes values in a subgroup of the form $\bigoplus \mathbb{Z} e_i $ for some basis $(e_i)$ of $\mathfrak{n}_{\mathbb{Q}} / \mathfrak{h}_{\mathbb{Q}}$. 

    Unpacking this as in the proof of Proposition \ref{Prop: cocycles into Aut(N) reduce to lattice stabiliser, baby case}, we get the desired result. 
\end{proof}

\subsection{Nilpotent Lie groups}

A key step in this paper reduces the study of point processes to so-called \emph{systems of order $K$}; we recall their definition in \S \ref{Subsection: Reduction to Host--Kra}. Before turning to that, we collect the facts we will need on nilpotent Lie groups, which underlie the Host--Kra--Ziegler theory of such systems.

For a Lie group $G$, the \emph{lower central series} is the descending sequence of closed subgroups defined by
\begin{enumerate}
    \item $C_0(G) := G$;
    \item $C_{i+1}(G) := \overline{[G, C_i(G)]}$, the closure of the subgroup generated by all commutators $[g, c] := g c g^{-1} c^{-1}$ with $g \in G$ and $c \in C_i(G)$.
\end{enumerate}
We say $G$ is \emph{nilpotent} if $C_m(G) = \{e\}$ for some $m \geq 0$, and the smallest such $m$ is the \emph{nilpotency class} of $G$. Nilpotent Lie groups are denoted $N$ throughout.

Connected nilpotent Lie groups share many properties with abelian Lie groups. In particular, every connected nilpotent Lie group $N$ contains a unique maximal compact subgroup, which is central; $N$ is simply connected if and only if it contains no non-trivial connected compact subgroup. For connected simply connected $N$, the exponential map $\exp: \mathrm{Lie}(N) \to N$ is a diffeomorphism, and the multiplication is described in terms of the Lie bracket via the Baker--Campbell--Hausdorff formula; we refer to \cite[Ch. II]{RaghunathanDiscrete} for background.

A \emph{lattice} in a nilpotent Lie group $N$ is a discrete cocompact subgroup $\Gamma \subset N$. For connected simply connected $N$, lattices in $N$ correspond to certain rational structures on $\mathrm{Lie}(N)$; we refer to \cite[Thm. 2.12]{RaghunathanDiscrete} for the precise correspondence and to \cite[Ch. II]{RaghunathanDiscrete} for Malcev's theorems. In particular, every finitely generated torsion-free nilpotent group $\Gamma$ is isomorphic to a lattice in a unique connected simply connected nilpotent Lie group, the \emph{Malcev completion} of $\Gamma$. We have already encountered this correspondence in Proposition \ref{Prop: cocycles into Aut(N) reduce to lattice stabiliser} to relate cocycles into $\Aut(N)$ to cocycles into arithmetic subgroups of $\mathrm{GL}_m(\mathbb{Q})$.

\subsection{Host--Kra--Ziegler theory for non-ergodic actions}\label{Subsection: Host--Kra--Ziegler theory for non-ergodic actions}

Rather than nilpotent Lie groups themselves, it is their homogeneous spaces that will be instrumental in this work.

A \emph{nilmanifold} is the quotient $N/\Gamma$ of a nilpotent Lie group $N$ by a lattice $\Gamma \subset N$. Every nilmanifold carries a unique $N$-invariant probability measure $\mu_{N/\Gamma}$, the \emph{Haar measure}. If $\Delta \subset \Gamma$ is normal in $N$, we may identify $N/\Gamma$ with $(N/\Delta)/(\Gamma/\Delta)$. We say that the pair $(N, \Gamma)$ is \emph{reduced} if no nontrivial subgroup of $\Gamma$ is normal in $N$; equivalently, $\Gamma \cap Z(N) = \{e\}$, where $Z(N)$ denotes the centre of $N$.

For $d \geq 0$ and a fixed nilmanifold $N/\Gamma$, any group homomorphism $\rho: \mathbb{Z}^d \to N$ defines a $\mathbb{Z}^d$-action on $N/\Gamma$ preserving $\mu_{N/\Gamma}$. A \emph{nilsystem} is the data of a nilmanifold $N/\Gamma$ equipped with its Haar measure and a $\mathbb{Z}^d$-action of this form. We will always assume that $N$ is generated by its identity component $N^\circ$ together with $\rho(\mathbb{Z}^d)$; this is no loss of generality, since the orbit of $\overline{e}$ under $N^\circ \cdot \rho(\mathbb{Z}^d)$ is a clopen subnilmanifold of $N/\Gamma$ on which the $\mathbb{Z}^d$-action is supported when ergodic.

The space $\Hom(\mathbb{Z}^d, N)$ identifies with the closed subset
$$ \{(u_1, \ldots, u_d) \in N^d : u_1, \ldots, u_d \text{ commute}\} \subset N^d $$
and through this inherits a natural locally compact topology. We can thus define measurably varying actions on $N/\Gamma$ by considering any standard Borel space $\Omega$ together with any Borel map $\Omega \to \Hom(\mathbb{Z}^d, N)$.

Any automorphism $\phi$ of $N$ preserving $\Gamma$ descends to a measure-preserving automorphism of $N/\Gamma$, and conjugates translation actions to translation actions. We write $\Aut(N/\Gamma)$ for the group of automorphisms of $N/\Gamma$ obtained in this way. A second source of transformations is given by \emph{base change}: for fixed $m \in N$, the map $N/\Gamma \to N/\Gamma$ sending $n\Gamma$ to $mn\Gamma$. When $(N, \Gamma)$ is reduced this group is isomorphic to $N$. Together, these transformations form the \emph{affine automorphisms} of $N/\Gamma$, isomorphic to the quotient of $\Aut(N/\Gamma) \ltimes N$ by the normal subgroup $\{(\gamma^{-1} \cdot \gamma
, \gamma) \vert \gamma \in \Gamma\}$. In fact, every isomorphism between nilsystems is affine; see \cite[Ch. 10-11]{HostKra} for this and related results.

We can now define bundles of such systems.

\begin{defn}[\cite{JamneshanMachado}]
We call a \emph{bundle of nilsystems (of dimension $\leq k$) over $\Omega$} any $\mathbb{Z}^d$-system constructed from the following data:
\begin{enumerate}
    \item a standard Borel probability space $(\Omega, \nu)$;
    \item a countable Borel partition $(\Omega_n)_{n \geq 0}$ of $\Omega$ and a countable family of nilmanifolds $(N_n/\Gamma_n)_{n \geq 0}$ of dimension $\leq k$;
    \item Borel maps $\rho_n: \Omega_n \to \Hom(\mathbb{Z}^d, N_n)$ such that for $\nu$-a.e. $\omega \in \Omega_n$, the action of $\rho_n(\omega)$ on $N_n/\Gamma_n$ is ergodic.
\end{enumerate}
The space
$$ X := \bigsqcup_{n \geq 0} \Omega_n \times N_n/\Gamma_n,$$
equipped with the measure
$$\sum_{n \geq 0} \nu_{|\Omega_n} \otimes \mu_{N_n/\Gamma_n}$$
and the $\mathbb{Z}^d$-action defined for $(\omega, \overline{n}) \in \Omega_n \times (N_n/\Gamma_n)$ and $t \in \mathbb{Z}^d$ by
$$ t \cdot (\omega, \overline{n}) = (\omega, \rho_n(\omega)(t) \cdot \overline{n}),$$
is denoted
$$ \bigsqcup_{n \geq 0} \Omega_n \rtimes_{\rho_n} N_n/\Gamma_n.$$
The natural map $\bigsqcup_{n \geq 0} \Omega_n \rtimes_{\rho_n} N_n/\Gamma_n \to \Omega$ is called the \emph{bundle projection}.
\end{defn}

The Host--Kra theorem \cite{HostKra} is a cornerstone of the theory of structure factors: for an ergodic system $X$, each characteristic factor $Z^k(X)$ is an inverse limit of nilsystems of step at most $k$. In \cite{JamneshanMachado}, the author and Jamneshan extended this result to non-ergodic systems, using bundles of nilsystems and the notion of an extension over a fixed base space.

\begin{defn}[\cite{JamneshanMachado}]
    Let $(X, \nu)$ be a $\mathbb{Z}^d$-system and let $\pi_X: (X, \nu) \to (\Omega, \mu)$ be a $\mathbb{Z}^d$-invariant factor map. Let $(Y, \mu_Y)$ be a bundle of $\mathbb{Z}^d$-nilsystems over $\Omega$ with factor map $\pi_Y: (Y, \mu_Y) \to (\Omega, \mu)$. A measure-preserving Borel map $\phi: (X, \nu) \to (Y, \mu_Y)$ is a \emph{nilfactor over $\Omega$} if it is $\mathbb{Z}^d$-equivariant and satisfies $\pi_Y \circ \phi = \pi_X$.
\end{defn}

This setup extends the notion of nilsystems to the non-ergodic setting, in which one recovers an analogue of the Host--Kra theorem.

\begin{thm}[Jamneshan--M., \cite{JamneshanMachado}]\label{Thm: Host-Kra non-ergodic}
    Let $X$ be a (not necessarily ergodic) $\mathbb{Z}^d$-system, and let $Z^k(X)$ denote its $k$-th characteristic factor. Then $Z^k(X)$ is isomorphic to an inverse limit of nilfactors of step $k$ over the invariant factor $Z^0(X)$.
\end{thm}

\subsection{Automorphisms and joinings of nilbundles}

Leibman \cite{Leibman}, see also \cite[Ch. 11, Prop. 15]{HostKra} showed that any factor or joining of nilsystems is again a nilsystem. We extend this to nilbundles over a common base $\Omega$. Throughout this subsection we fix two $\mathbb{Z}^d$-bundles of nilsystems
$$ X = \bigsqcup_{k \geq 0} \Omega_{X,k} \rtimes_{\rho_{X,k}} N_{X,k}/\Gamma_{X,k}, \qquad Y = \bigsqcup_{l \geq 0} \Omega_{Y,l} \rtimes_{\rho_{Y,l}} N_{Y,l}/\Gamma_{Y,l},$$
both over the standard Borel base $(\Omega, \nu)$, with $\mu_X$ and $\mu_Y$ the corresponding measures.

By a \emph{subnilmanifold} of a nilmanifold $N/\Gamma$ we mean a closed subset of the form $u H \cdot \overline{e}$, where $u \in N$ and $H$ is a closed subgroup of $N$ for which $H \cap \Gamma $ is a lattice in $H$; we call such an $H$ a \emph{rational subgroup} of $N$. Each subnilmanifold carries a unique $uHu^{-1}$-invariant probability measure, which we refer to as its Haar (probability) measure.

\begin{lem}\label{Lemma: Joining nilbundles}
    Let $(Z, \eta)$ be a joining of $X$ and $Y$ over $\Omega$: that is, a $\mathbb{Z}^d$-system equipped with measure-preserving $\mathbb{Z}^d$-equivariant maps $\psi: Z \to X$ and $\phi: Z \to Y$ satisfying $\pi_X \circ \psi = \pi_Y \circ \phi$, so that the diagram
    \begin{center}
     \begin{tikzcd}
Z \arrow[r, "\phi"] \arrow[d, "\psi"] \arrow[dr, "\pi_Z"]
& Y  \arrow[d, "\pi_Y"] \\
X \arrow[r,  "\pi_X"]
&  \Omega
\end{tikzcd}
     \end{center}
     commutes, where $\pi_Z := \pi_X \circ \psi = \pi_Y \circ \phi$. Then $Z$ admits the structure of a bundle of nilsystems over $\Omega$ for which $\psi$ and $\phi$ are nilfactors over $\Omega$.
\end{lem}

The proof relies on the following measurable selection result.

\begin{lem}\label{Lemma: Measurable selection subnilmanifolds}
    Let $N/\Gamma$ be a nilmanifold and let $S$ be a standard Borel space. Let $\omega \mapsto \eta_\omega$ be a Borel map from $S$ into the space of Haar measures on subnilmanifolds of $N/\Gamma$, equipped with the weak-$*$ topology. Then there exist Borel maps $u: S \to N$ and $H: S \to \mathcal{H}$, where $\mathcal{H}$ denotes the (countable) set of rational subgroups of $N$, such that for every $\omega \in S$, $u(\omega)^* \eta_\omega$ is the Haar measure on $H(\omega) \cdot \overline{e}$.
\end{lem}

\begin{proof}
    Let $\mathcal{M}_{\mathrm{Haar}}$ denote the space of Haar measures on subnilmanifolds of $N/\Gamma$, equipped with the weak-$*$ topology; it is compact (see e.g. \cite{Leibman}). Enumerate $\mathcal{H} = (H_n)_{n \geq 0}$, and let $\mu_n$ be the Haar measure on $H_n \cdot \overline{e}$. Every $\eta \in \mathcal{M}_{\mathrm{Haar}}$ is of the form $u^* \mu_n$ for some $u \in N$ and some $n \geq 0$, so the continuous map
    $$ N \times \mathbb{N} \longrightarrow \mathcal{M}_{\mathrm{Haar}}, \qquad (u,n) \mapsto u^* \mu_n$$
    is surjective. By the Jankov--von Neumann measurable selection theorem \cite[Thm. 18.1]{KechrisDescriptive}, there exist Borel maps $\widetilde{u}: \mathcal{M}_{\mathrm{Haar}} \to N$ and $\widetilde{n}: \mathcal{M}_{\mathrm{Haar}} \to \mathbb{N}$ with $\widetilde{u}(\eta)^* \eta = \mu_{\widetilde{n}(\eta)}$. Setting $u(\omega) := \widetilde{u}(\eta_\omega)$ and $H(\omega) := H_{\widetilde{n}(\eta_\omega)}$ concludes.
\end{proof}

\begin{proof}[Proof of Lemma \ref{Lemma: Joining nilbundles}.]
    Disintegrate $\eta = \int_\Omega \eta_\omega \, d\nu(\omega)$ along $\pi_Z$. For $\omega \in \Omega_{X,k} \cap \Omega_{Y,l}$, the fibre $\pi_Z^{-1}(\omega)$ is naturally identified with a subset of $N_{X,k}/\Gamma_{X,k} \times N_{Y,l}/\Gamma_{Y,l}$, and under this identification $\eta_\omega$ is a self-joining of two ergodic nilsystems with Haar marginals. By \cite[Ch. 11, Prop. 15]{HostKra} , $\eta_\omega$ is the Haar measure on a subnilmanifold of $N_{X,k}/\Gamma_{X,k} \times N_{Y,l}/\Gamma_{Y,l}$, and $\omega \mapsto \eta_\omega$ is Borel by standard properties of disintegration.

    Applying Lemma \ref{Lemma: Measurable selection subnilmanifolds} on each piece $\Omega_{X,k} \cap \Omega_{Y,l}$ yields a countable Borel partition $(\Omega_{k,l,m})_{m \geq 0}$ of $\Omega_{X,k} \cap \Omega_{Y,l}$, together with rational subgroups $H_{k,l,m} \leq N_{X,k} \times N_{Y,l}$ and Borel maps $u_{k,l,m}: \Omega_{k,l,m} \to N_{X,k} \times N_{Y,l}$, such that for $\nu$-a.e. $\omega \in \Omega_{k,l,m}$, $\eta_\omega$ is the Haar measure on $u_{k,l,m}(\omega) \cdot H_{k,l,m} \cdot \overline{e}$. This exhibits $Z$ as a bundle of nilsystems over $\Omega$ with fibres modelled on $H_{k,l,m}/(H_{k,l,m} \cap (\Gamma_{X,k} \times \Gamma_{Y,l}))$, and the projections $\psi$, $\phi$ are nilfactors over $\Omega$ by construction.
\end{proof}

\begin{defn}
    Let $(Z, \mu_Z)$ be a joining of $X$ and $Y$ over $\Omega$, identified with a bundle of nilsystems over $\Omega$ via Lemma \ref{Lemma: Joining nilbundles}. We say that $Z$ is \emph{finite-to-finite} if both nilfactor maps $Z \to X$ and $Z \to Y$ are finite-to-one $\mu_Z$-almost everywhere.
\end{defn}

For factor maps we obtain a more precise description.

\begin{lem}\label{Lemma: Factors nilbundles}
    Suppose that $(N_{Y,l}, \Gamma_{Y,l})$ is reduced for every $l \geq 0$, and that $\pi: X \to Y$ is a nilfactor over $\Omega$. Then for all $k, l \geq 0$ there exist a Borel partition $(\Omega_{k,l,m})_{m \geq 0}$ of $\Omega_{X,k} \cap \Omega_{Y,l}$, surjective continuous group homomorphisms $\phi_{k,l,m}: N_{X,k} \to N_{Y,l}$ with $\phi_{k,l,m}(\Gamma_{X,k}) \subset \Gamma_{Y,l}$, and Borel maps $b_{k,l,m}: \Omega_{k,l,m} \to N_{Y,l}$, such that for $\nu$-a.e. $\omega \in \Omega_{k,l,m}$,
    $$ \pi(\omega, \overline{n}) = \bigl(\omega,\; b_{k,l,m}(\omega) \cdot \overline{\phi_{k,l,m}(n)}\bigr) \quad \text{for $\mu_{X,\omega}$-a.e. } \overline{n}.$$
\end{lem}

\begin{proof}
    Apply Lemma \ref{Lemma: Joining nilbundles} to the graph joining of $\pi$. On each piece $\Omega_{X,k} \cap \Omega_{Y,l}$ we obtain a countable Borel partition $(\Omega_{k,l,m})_{m \geq 0}$ together with a rational subgroup $H_{k,l,m} \leq N_{X,k} \times N_{Y,l}$ and a Borel map $\omega \mapsto u_{k,l,m}(\omega) \in N_{X,k} \times N_{Y,l}$ such that, for $\nu$-a.e. $\omega \in \Omega_{k,l,m}$, the graph of $\pi$ on the fibre over $\omega$ is the image in $N_{X,k}/\Gamma_{X,k} \times N_{Y,l}/\Gamma_{Y,l}$ of $u_{k,l,m}(\omega) \cdot H_{k,l,m}$.

    Since $\pi$ is a factor map over $\Omega$, the projection of this graph onto $N_{X,k}/\Gamma_{X,k}$ is surjective, hence the projection $p_X: H_{k,l,m} \to N_{X,k}$ has image acting transitively on $N_{X,k}/\Gamma_{X,k}$, which forces $p_X(H_{k,l,m}) = N_{X,k}$. Moreover $\pi$ is single-valued, so $H_{k,l,m} \cap (\{e\} \times N_{Y,l})$ projects to a closed normal subgroup of $N_{Y,l}$ contained in $\Gamma_{Y,l}$; reducedness of $(N_{Y,l}, \Gamma_{Y,l})$ implies that this subgroup is trivial. Hence $p_X$ is a continuous group isomorphism of $H_{k,l,m}$ onto $N_{X,k}$, and the formula
    $$ \phi_{k,l,m} := p_Y \circ p_X^{-1} : N_{X,k} \longrightarrow N_{Y,l}$$
    defines a continuous group homomorphism whose graph is $H_{k,l,m}$.

    Surjectivity of $\phi_{k,l,m}$ follows from the same argument applied on the $Y$-side: the projection of the graph of $\pi$ onto $N_{Y,l}/\Gamma_{Y,l}$ has full image (as the pushforward of $\mu_X$ by $\pi$ is $\mu_Y$), so $\phi_{k,l,m}(N_{X,k}) \cdot \Gamma_{Y,l} = N_{Y,l}$. Since $\phi_{k,l,m}(N_{X,k})$ is a closed subgroup of $N_{Y,l}$ and $\Gamma_{Y,l}$ is discrete, $\phi_{k,l,m}(N_{X,k})$ has finite index in $N_{Y,l}$, hence contains the connected component of the identity; this implies surjectivity since $N_{Y,l}$ is generated by its connected component of the identity and  $\rho_{Y,l}(\omega)(\mathbb{Z}^d) = \phi_{k,l,m} \circ \rho_{X,k}(\omega)(\mathbb{Z}^d)$.

    Finally, writing $u_{k,l,m}(\omega) = (a(\omega), b_{k,l,m}(\omega))$ and using that $H_{k,l,m}$ is the graph of $\phi_{k,l,m}$, the relation $\pi(\omega, \overline{n}) = (\omega, b_{k,l,m}(\omega) \cdot \overline{\phi_{k,l,m}(n)})$ holds, after absorbing $a(\omega)$ into $b_{k,l,m}(\omega)$ via the identity $\phi_{k,l,m}(a(\omega)^{-1} n) = \phi_{k,l,m}(a(\omega))^{-1} \phi_{k,l,m}(n)$. The inclusion $\phi_{k,l,m}(\Gamma_{X,k}) \subset \Gamma_{Y,l}$ is automatic from $\pi$ being well-defined on $N_{X,k}/\Gamma_{X,k}$.
\end{proof}

\section{Reduction to Host--Kra factors}\label{Subsection: Reduction to Host--Kra}

The goal of this section is to prove the following result.

\begin{thm}\label{Theorem: Reduction to Host--Kra factors}
    Let $(X, \mu)$ be an $\ASL_d(\mathbb{Z})$-system. Let $k \geq 1$, let $v_1, \ldots, v_k \in \mathbb{Z}^d \setminus \{0\}$ be pairwise distinct, and let $f_1, \ldots, f_k \in L^\infty(X)$ be $\SL_d(\mathbb{Z})$-invariant. Set $K := 2^{d-1}k$. Then
    $$ \mathbb{E}\left[(v_1 \cdot f_1) \cdots (v_k \cdot f_k)\right] = \mathbb{E}\left[\bigl(v_1 \cdot \mathbb{E}[f_1 \mid Z^{K}(X)]\bigr) \cdots \bigl(v_k \cdot \mathbb{E}[f_k \mid Z^{K}(X)]\bigr)\right],$$
    where $Z^{K}(X)$ denotes the $K$-th Host--Kra factor of $X$.
\end{thm}

The proof unfolds in two parts. First, a unipotent trick reduces the question to Host--Kra factors with respect to certain subgroups of $\mathbb{Z}^d$. Second, a \emph{concatenation} argument - relying on \cite{TaoZieglerConcat} - upgrades these subgroup factors to the full Host--Kra factor.

\subsection{Structure factors}

Let $(X, \mu)$ be a $\mathbb{Z}^d$-system. For $t \in \mathbb{Z}^d$ and $f \in L^\infty(X)$, define $\Delta_t f := (t \cdot f) \cdot \overline{f}$. Given a subgroup $H \subset \mathbb{Z}^d$ and a F\o lner sequence $(\mathcal{F}_n)_{n \geq 0}$ of $H$, it is proved in \cite[Thm 2.1]{TaoZieglerConcat} that the limit
$$ \|f\|_{U^k_H(X)} := \lim_{n_1 \to \infty} \cdots \lim_{n_k \to \infty}\left(\mathbb{E}_{g_1 \in \mathcal{F}_{n_1}} \cdots \mathbb{E}_{g_k \in \mathcal{F}_{n_k}} \int \Delta_{g_1} \cdots \Delta_{g_k} f \, d\mu\right)^{2^{-k}}$$
exists, defines a seminorm on $L^\infty(X)$, and is independent of the choice of F\o lner sequence. We refer to it as the $k$-th \emph{Gowers--Host--Kra seminorm} associated to $H$. It is related to structure factors via
$$ f \in L^\infty(Z^k_H(X))^\perp \iff \|f\|_{U^k_H(X)} = 0,$$
see \cite[Thm 2.4]{TaoZieglerConcat}. When $H = \mathbb{Z}^d$ we drop the subscript, writing $\|f\|_{U^k(X)}$ and $Z^k(X)$.

The independence on the choice of F\o lner sequence yields a first observation.

\begin{lem}\label{Lemma: ASLd invariance}
    Let $(X, \mu)$ be an $\ASL_d(\mathbb{Z}) = \SL_d(\mathbb{Z}) \ltimes \mathbb{Z}^d$-system, and let $Z^k(X)$ denote the $k$-th Host--Kra factor with respect to the translation action of $\mathbb{Z}^d$. Then $L^2(Z^k(X))$ is invariant under $\ASL_d(\mathbb{Z})$.
\end{lem}

\begin{proof}
    Since $\mu$ is $\ASL_d(\mathbb{Z})$-invariant, it suffices to show that $L^2(Z^k(X))^\perp$ is $\ASL_d(\mathbb{Z})$-invariant. Let $f \in L^2(Z^k(X))^\perp$, so that $\|f\|_{U^k(X)} = 0$. The image of any F\o lner sequence in $\mathbb{Z}^d$ under any $\sigma \in \ASL_d(\mathbb{Z})$ is again a F\o lner sequence in $\mathbb{Z}^d$, hence
    $$\|\sigma \cdot f\|_{U^k(X)} = \|f\|_{U^k(X)} = 0,$$
    which gives $\sigma \cdot f \in L^2(Z^k(X))^\perp$.
\end{proof}

Structure factors satisfy further natural properties.

\begin{lem}\label{Lemma: Successive factors}
    Let $\pi: (X, \mu) \to (Y, \nu)$ be a factor of $\mathbb{Z}^d$-systems, and let $H \subset \mathbb{Z}^d$ be an infinite subgroup. Identify $L^\infty(Y)$ with its image in $L^\infty(X)$ under $\pi$. Then
    $$ L^\infty(Z^k_H(Y)) = L^\infty(Z^k_H(X)) \cap L^\infty(Y)$$
    and, for every $f \in L^\infty(Y)$,
    $$ \mathbb{E}[f \mid Z^k_H(X)] = \mathbb{E}[f \mid Z^k_H(Y)].$$
\end{lem}

\begin{proof}
    The first identity is \cite[Cor. 2.5]{TaoZieglerConcat}. For the second, write $g := f - \mathbb{E}[f \mid Z^k_H(Y)]$. Since $g \in L^\infty(Y) \cap L^\infty(Z^k_H(Y))^\perp$, we have $\|g\|_{U^k_H(Y)} = 0$. The Gowers--Host--Kra seminorm depends only on the integrals of $\Delta_{g_1} \cdots \Delta_{g_k} g$, which are the same in $X$ and $Y$ (as $g$ comes from $Y$), so $\|g\|_{U^k_H(X)} = \|g\|_{U^k_H(Y)} = 0$ and therefore $g \in L^\infty(Z^k_H(X))^\perp$. Decomposing
    $$f = \mathbb{E}[f \mid Z^k_H(Y)] + g,$$
    with the first term in $L^\infty(Z^k_H(Y)) \subseteq L^\infty(Z^k_H(X))$ (by the first identity) and the second in $L^\infty(Z^k_H(X))^\perp$, identifies $\mathbb{E}[f \mid Z^k_H(Y)]$ with the orthogonal projection $\mathbb{E}[f \mid Z^k_H(X)]$.
\end{proof}

\subsection{Concatenation}

We use the \emph{concatenation theorem} of Tao--Ziegler \cite[Thm. 1.15]{TaoZieglerConcat}: for every $\mathbb{Z}^d$-system $(X, \mu)$, every $k, l \geq 1$, and every pair of subgroups $H_1, H_2 \subset \mathbb{Z}^d$,
\begin{equation}\label{Eq: Concatenation}
L^\infty\bigl(Z^k_{H_1}(X)\bigr) \cap L^\infty\bigl(Z^l_{H_2}(X)\bigr) \subset L^\infty\bigl(Z^{k+l-1}_{H_1 + H_2}(X)\bigr).
\end{equation}
This has a striking consequence for $\SL_d(\mathbb{Z})$-invariant functions on an $\ASL_d(\mathbb{Z})$-system.

\begin{prop}[$\ASL_d(\mathbb{Z})$-concatenation]\label{Proposition: Concatenation}
    Let $(X, \mu)$ be an $\ASL_d(\mathbb{Z})$-system and let $H \subset \mathbb{Z}^d$ be a subgroup of rank $r \geq 1$. For every $\SL_d(\mathbb{Z})$-invariant $f \in L^\infty(X)$ and every $k \geq 1$,
    $$ \mathbb{E}[f \mid Z^k_H(X)] \in L^\infty\bigl(Z^{\kappa_r(k)}(X)\bigr),$$
    where $\kappa_r$ is defined by the recursion $\kappa_d(k) = k$ and $\kappa_r(k) = \kappa_{r+1}(2k - 1)$ for $1 \leq r \leq d-1$.
\end{prop}

A direct computation gives $\kappa_r(k) = 2^{d-r}(k-1) + 1$, so in particular $\kappa_1(k) = 2^{d-1}(k-1) + 1 \leq 2^{d-1}k$ for every $k \geq 1$.

\begin{proof}
    We argue by descending induction on $r$.

    \emph{Base case} ($r = d$). Here $H$ has finite index in $\mathbb{Z}^d$, so $Z^k_H(X) = Z^k(X)$ \cite[Lem 3.1]{FrantzikinakisKucaJoint} and the conclusion is immediate.

    \emph{Inductive step.} Suppose the conclusion holds for ranks $\geq r+1$, and let $H \subset \mathbb{Z}^d$ have rank exactly $r$. Set
    $$ g := f - \mathbb{E}[f \mid Z^{\kappa_{r+1}(2k-1)}(X)].$$
    By Lemma \ref{Lemma: ASLd invariance}, $\mathbb{E}[f \mid Z^{\kappa_{r+1}(2k-1)}(X)]$ is $\SL_d(\mathbb{Z})$-invariant, hence so is $g$.

    For any $\gamma \in \SL_d(\mathbb{Z})$, the $\SL_d(\mathbb{Z})$-invariance of $g$ and of $\mu$ gives
    \begin{equation}\label{Eq: Norm equality}
    \|\mathbb{E}[g \mid Z^k_{\gamma(H)}(X)]\|_{L^2} = \|\mathbb{E}[g \mid Z^k_H(X)]\|_{L^2}.
    \end{equation}
    Moreover,
    $$\langle \mathbb{E}[g \mid Z^k_{\gamma(H)}(X)], \mathbb{E}[g \mid Z^k_H(X)] \rangle_{L^2} = \langle g, \mathbb{E}\bigl[\mathbb{E}[g \mid Z^k_H(X)] \,\big|\, Z^k_{\gamma(H)}(X)\bigr] \rangle_{L^2}.$$
    By two applications of Lemma \ref{Lemma: Successive factors}, the iterated conditional expectation collapses to the projection on the meet:
    \begin{align*}
        \mathbb{E}\bigl[\mathbb{E}[g \mid Z^k_H(X)] \,\big|\, Z^k_{\gamma(H)}(X)\bigr]
        &= \mathbb{E}\bigl[g \,\big|\, Z^k_{\gamma(H)}(Z^k_H(X))\bigr] \\
        &= \mathbb{E}[g \mid Z^k_{\gamma(H)}(X) \wedge Z^k_H(X)].
    \end{align*}
    By the concatenation theorem \eqref{Eq: Concatenation},
    $$L^\infty\bigl(Z^k_{\gamma(H)}(X) \wedge Z^k_H(X)\bigr) \subset L^\infty\bigl(Z^{2k-1}_{\gamma(H) + H}(X)\bigr).$$
    When $H + \gamma(H)$ has rank at least $r+1$, the inductive hypothesis applied to $g$ (which is $\SL_d(\mathbb{Z})$-invariant) gives
    $$\mathbb{E}[g \mid Z^{2k-1}_{\gamma(H) + H}(X)] \in L^\infty\bigl(Z^{\kappa_{r+1}(2k-1)}(X)\bigr),$$
    and the very definition of $g$ then forces $\mathbb{E}[g \mid Z^{2k-1}_{\gamma(H) + H}(X)] = 0$. A fortiori,
    \begin{equation}\label{Eq: Orthogonality}
    \langle \mathbb{E}[g \mid Z^k_{\gamma(H)}(X)], \mathbb{E}[g \mid Z^k_H(X)] \rangle_{L^2} = 0
    \end{equation}
    whenever $H + \gamma(H)$ has rank at least $r+1$.

    Now let
    $$ L := \{\gamma \in \SL_d(\mathbb{Z}) : \gamma(H \otimes \mathbb{Q}) = H \otimes \mathbb{Q}\}.$$
    This is the stabiliser in $\SL_d(\mathbb{Z})$ of the rational subspace $H \otimes \mathbb{Q} \subset \mathbb{Q}^d$ and is therefore a subgroup. For $\gamma \in \SL_d(\mathbb{Z}) \setminus L$, the rational subspace $\gamma(H) \otimes \mathbb{Q}$ differs from $H \otimes \mathbb{Q}$, so $H + \gamma(H)$ has rank strictly greater than $r$. Since $r \leq d-1$, the orbit of $H \otimes \mathbb{Q}$ under $\SL_d(\mathbb{Z})$ in the rational Grassmannian $\mathrm{Gr}_r(\mathbb{Q}^d)$ is infinite, hence $\SL_d(\mathbb{Z})/L$ is infinite. Choose representatives $(\gamma_n)_{n \geq 0}$ of pairwise distinct cosets, with $\gamma_0 = \id$.

    By \eqref{Eq: Orthogonality}, the family $\bigl(\mathbb{E}[g \mid Z^k_{\gamma_n H}(X)]\bigr)_{n \geq 0}$ is pairwise orthogonal in $L^2$. Bessel's inequality  and \eqref{Eq: Norm equality} give
    $$ \|g\|_{L^2}^2 \;\geq\; \sum_{n \geq 0} \|\mathbb{E}[g \mid Z^k_{\gamma_n H}(X)]\|_{L^2}^2 \;=\; \sum_{n \geq 0} \|\mathbb{E}[g \mid Z^k_H(X)]\|_{L^2}^2.$$
    The right-hand side is finite only if $\|\mathbb{E}[g \mid Z^k_H(X)]\|_{L^2} = 0$, i.e.\ $\mathbb{E}[g \mid Z^k_H(X)] = 0$. Therefore
    $$\mathbb{E}[f \mid Z^k_H(X)] = \mathbb{E}\bigl[\mathbb{E}[f \mid Z^{\kappa_{r+1}(2k-1)}(X)] \;\big|\; Z^k_H(X)\bigr] \in L^\infty\bigl(Z^{\kappa_{r+1}(2k-1)}(X)\bigr).$$
    This is the inductive step at rank $r$ with exponent $\kappa_r(k) = \kappa_{r+1}(2k-1)$.
\end{proof}

\subsection{Proof of Theorem \ref{Theorem: Reduction to Host--Kra factors}}

\begin{proof}[Proof of Theorem \ref{Theorem: Reduction to Host--Kra factors}]
    Set $K := 2^{d-1}k$ and $g_i := f_i - \mathbb{E}[f_i \mid Z^K(X)]$ for $i = 1, \ldots, k$. Each $g_i$ is $\SL_d(\mathbb{Z})$-invariant by Lemma \ref{Lemma: ASLd invariance}. We claim that for every infinite subgroup $H \subset \mathbb{Z}^d$ and every $i$,
    \begin{equation}\label{Eq: gi vanishing}
    \mathbb{E}[g_i \mid Z^k_H(X)] = 0.
    \end{equation}
    Indeed, by Proposition \ref{Proposition: Concatenation} we have $\mathbb{E}[f_i \mid Z^k_H(X)] \in L^\infty(Z^{\kappa_1(k)}(X)) \subseteq L^\infty(Z^K(X))$ (using $\kappa_1(k) \leq 2^{d-1}k = K$). Therefore,
    $$\mathbb{E}[g_i \mid Z^k_H(X)] = \mathbb{E}[f_i \mid Z^k_H(X)] - \mathbb{E}[\mathbb{E}[f_i \mid Z^K(X)] \mid Z^k_H(X)] = 0.$$

    Choose a group homomorphism $\phi: \mathbb{Z}^d \to \mathbb{Z}$ such that $\phi(v_i) \neq 0$ and $\phi(v_i - v_j) \neq 0$ for all $1 \leq i < j \leq k$ (possible since the $v_i$ are pairwise distinct and nonzero, and the conditions exclude only finitely many hyperplanes from the dual). Pick $u \in \ker\phi \setminus \{0\}$ and define $\psi: \mathbb{Z}^d \to \mathbb{Z}^d$ by $\psi(v) := v + \phi(v) u$. Since $\phi(u) = 0$, $\psi - \id$ squares to zero. Hence, $\psi \in \SL_d(\mathbb{Z})$ is unipotent and $\psi^l(v) = v + l \phi(v) u$ for every $l \geq 0$.

    Let $h_i \in \{f_i, g_i\}$. In either case, $h_i$ is $\SL_d(\mathbb{Z})$-invariant. By $\SL_d(\mathbb{Z})$-invariance of $\mu$ and of each $h_i$, for every $l \geq 0$,
    \begin{align*}
    \mathbb{E}[(v_1 \cdot h_1) \cdots (v_k \cdot h_k)]
    &= \mathbb{E}[\psi^l \cdot ((v_1 \cdot h_1) \cdots (v_k \cdot h_k))] \\
    &= \mathbb{E}[(\psi^l(v_1) \cdot \psi^l \cdot h_1) \cdots (\psi^l(v_k) \cdot \psi^l \cdot h_k)] \\
    &= \mathbb{E}[(\psi^l(v_1) \cdot h_1) \cdots (\psi^l(v_k) \cdot h_k)] \\
    &= \mathbb{E}\bigl[\bigl(l\phi(v_1) u \cdot (v_1 \cdot h_1)\bigr) \cdots \bigl(l\phi(v_k) u \cdot (v_k \cdot h_k)\bigr)\bigr].
    \end{align*}
    Averaging over $0 \leq l \leq N$,
    $$ \mathbb{E}[(v_1 \cdot h_1) \cdots (v_k \cdot h_k)] = \mathbb{E}\left[\frac{1}{N+1} \sum_{l=0}^{N} \prod_{i=1}^{k} \bigl(l\phi(v_i) u \cdot (v_i \cdot h_i)\bigr)\right].$$

    Suppose at least one $h_i$ equals $g_i$. The integers $\phi(v_i)$ are pairwise distinct (since $\phi(v_i - v_j) \neq 0$) and nonzero, so the average on the right is a non-conventional ergodic average in the sense of \cite{HostKra, ZieglerNonConventional}. By \cite[Ch. 21, Prop. 7]{HostKra}, this average converges in $L^2$-norm to $0$ provided at least one factor satisfies $\mathbb{E}[g_i \mid Z^k_{\langle u \rangle}(X)] = 0$, which holds by \eqref{Eq: gi vanishing} (applied with $H = \langle u \rangle$). Therefore
    \begin{equation}\label{Eq: Mixed term vanishing}
    \mathbb{E}[(v_1 \cdot h_1) \cdots (v_k \cdot h_k)] = 0 \quad \text{whenever at least one } h_i = g_i.
    \end{equation}

    Finally, expand each $f_i = \mathbb{E}[f_i \mid Z^K(X)] + g_i$ and write
    $$ \mathbb{E}[(v_1 \cdot f_1) \cdots (v_k \cdot f_k)] - \mathbb{E}\left[\bigl(v_1 \cdot \mathbb{E}[f_1 \mid Z^K(X)]\bigr) \cdots \bigl(v_k \cdot \mathbb{E}[f_k \mid Z^K(X)]\bigr)\right]$$
    as a telescoping sum over the $2^k - 1$ choices of $(h_1, \ldots, h_k) \in \{f_i, g_i\}^k \setminus \{(\mathbb{E}[f_i \mid Z^K(X)])_i\}$ in which at least one entry is a $g_i$ (see e.g. the proof of \cite[Ch. 21, Thm. 6]{HostKra}). By \eqref{Eq: Mixed term vanishing} every such term vanishes, completing the proof.
\end{proof}

\section{Finding $\ASL_d(\mathbb{Z})$-nilfactors}\label{Section: Bundles of nilsystems}

\subsection{Commensurated factors}

To prove our main result, we study $\SL_d(\mathbb{Z})$-invariant functions on $Z^k(X)$ by approximating them by their projections onto $\ASL_d(\mathbb{Z})$-invariant nilfactors. This reduces the proof to the study of measures on nilmanifolds invariant under an action by $\ASL_d(\mathbb{Z})$. As a first step towards establishing that such an approximation is possible, we establish the existence of \emph{commensurated} nilfactors on which we can approximate our functions efficiently.

\begin{defn}\label{Def: commensurated}
    Let $(X, \mu)$ be an $\ASL_d(\mathbb{Z})$-system and $\pi: X \to Y$ a nilfactor over $Z^0(X)$. We say $Y$ is \emph{commensurated} by $\SL_d(\mathbb{Z})$ if for every $\gamma \in \SL_d(\mathbb{Z})$ the joining of $Y$ with itself induced by $(\pi, \pi \circ \gamma): X \to Y \times Y$ is finite-to-finite.
\end{defn}

Note that this joining is indeed defined over $Z^0(X)$: the $\ASL_d(\mathbb{Z})$-action on $Z^0(X)$ factors through $\SL_d(\mathbb{Z})$, so $\gamma$ acts on $Z^0(X)$, and the diagram
$$\begin{tikzcd}
X \arrow[r, "\pi \circ \gamma"] \arrow[d, "\pi"] & Y \arrow[d] \\
Y \arrow[r] & Z^0(X)
\end{tikzcd}$$
commutes.

The following lemma describes how a nilfactor changes under precomposition by $\gamma \in \SL_d(\mathbb{Z})$. Throughout, for $\gamma \in \ASL_d(\mathbb{Z})$ we write $\gamma_0 \in \SL_d(\mathbb{Z})$ for its linear part, so that $\gamma$ acts on $\mathbb{Z}^d$ via $t \mapsto \gamma_0(t) + v_\gamma$ for some $v_\gamma \in \mathbb{Z}^d$.

\begin{lem}\label{Lemma: gamma-twist of a nilfactor}
    Let $(X, \mu)$ be an $\ASL_d(\mathbb{Z})$-system and $\pi: (X, \mu) \to (Y, \mu_Y)$ a nilfactor over $Z^0(X)$, with
    $$ Y = \bigsqcup_{n \geq 0} \Omega_n \rtimes_{\rho_n} N_n / \Gamma_n.$$
    Then for every $\gamma \in \SL_d(\mathbb{Z})$, the composition $\pi \circ \gamma$ is a nilfactor over $Z^0(X)$ with range
    $$ \bigsqcup_{n \geq 0} \gamma^{-1} \Omega_n \rtimes_{\rho_n^\gamma} N_n / \Gamma_n,$$
    where $\rho_n^\gamma(\omega)(t) := \rho_n(\gamma \cdot \omega)(\gamma(t))$ for $\omega \in \gamma^{-1} \Omega_n$ and $t \in \mathbb{Z}^d$.
\end{lem}

\begin{proof}
    Direct verification from the definitions, using that $\gamma$ acts measurably on $Z^0(X)$ and $\mathbb{Z}^d$.
\end{proof}

From now on, fix $k \geq 0$ and an $\ASL_d(\mathbb{Z})$-system $(X, \mu)$ which is ergodic and of order $k$ with respect to the $\mathbb{Z}^d$-action: $X = Z^k(X)$. We aim to understand $\SL_d(\mathbb{Z})$-invariant functions on $X$.

\begin{prop}\label{Prop: reduction to nilbundles}
    Let $(X, \mu)$ be as above, $f \in L^\infty(X)$ an $\SL_d(\mathbb{Z})$-invariant function, and $\epsilon > 0$. Then there exists a nilfactor $\pi: X \to Y$ commensurated by $\SL_d(\mathbb{Z})$ with
    $$ \|f - \mathbb{E}[f \mid Y]\|_2 \leq \epsilon.$$
\end{prop}

The proof relies on the following dimension-counting lemma, which controls the dimension of joinings under a non-degeneracy hypothesis.

\begin{lem}\label{Lemma: Dimension bound under proper factor condition}
    Let $X$ be an ergodic $\mathbb{Z}^{d'}$-system, $f \in L^\infty(X)$, and $\delta > 0$. Let $D \geq 0$ and let $(N_n)_{n \geq 0}$ be nilfactors of $X$, each of dimension at most $D$, satisfying:
    \begin{equation}\label{Eq: nondegeneracy}
        \|\mathbb{E}[f \mid N_n] - \mathbb{E}[f \mid N_n']\|_2 > \delta
    \end{equation}
    for every $n \geq 0$ and every strict nilfactor $N_n' \to N_n$ (i.e.\ of dimension strictly smaller than $\dim N_n$). Then for every finite $I \subset \mathbb{N}$, the joining of $\{N_n\}_{n \in I}$ inside $X$ is a nilsystem of dimension at most $D \|f\|_2^2 \delta^{-2}$.
\end{lem}

\begin{proof}
    For $i \geq 0$, let $N_{<i}$ denote the join of $N_0, \ldots, N_{i-1}$ inside $X$ i.e. the smallest factor of $X$ through which each $N_j$ ($j < i$) factors. By \cite[Ch. 13, Thm. 6]{HostKra}, $N_{<i}$ is a nilsystem, and $N_{<i+1}$ is a joining of $N_{<i}$ and $N_i$ relatively independent over a common factor $N_i \to N_i^*$.

    There are two cases. If $\dim N_i^* = \dim N_i$, then $\dim N_{<i+1} = \dim N_{<i}$. Otherwise $N_i^*$ is a strict factor of $N_i$, hypothesis \eqref{Eq: nondegeneracy} gives $\|\mathbb{E}[f \mid N_i] - \mathbb{E}[f \mid N_i^*]\|_2 > \delta$, and by relative independence of $N_{<i+1}$ over $N_i^*$ the function $\mathbb{E}[f \mid N_i] - \mathbb{E}[f \mid N_i^*]$ is orthogonal to $L^2(N_{<i})$. Hence
    \begin{align*}
        \|\mathbb{E}[f \mid N_{<i+1}]\|_2^2
        &\geq \|\bigl(\mathbb{E}[f \mid N_i] - \mathbb{E}[f \mid N_i^*]\bigr) + \mathbb{E}[f \mid N_{<i}]\|_2^2 \\
        &= \|\mathbb{E}[f \mid N_i] - \mathbb{E}[f \mid N_i^*]\|_2^2 + \|\mathbb{E}[f \mid N_{<i}]\|_2^2 \\
        &\geq \|\mathbb{E}[f \mid N_{<i}]\|_2^2 + \delta^2,
    \end{align*}
    and $\dim N_{<i+1} \leq \dim N_{<i} + D$.

    Since $\|\mathbb{E}[f \mid N_{<i}]\|_2 \leq \|f\|_2$ for all $i$, the second case occurs at most $\|f\|_2^2 \delta^{-2}$ times. Therefore $\dim N_{<i} \leq D \|f\|_2^2 \delta^{-2}$ for all $i$.
\end{proof}

\begin{proof}[Proof of Proposition \ref{Prop: reduction to nilbundles}]
    Throughout, write $\Gamma := \SL_d(\mathbb{Z})$ and $\Omega := Z^0(X)$.

    \textit{Step 1 (initial choice of nilfactor).} By Theorem \ref{Thm: Host-Kra non-ergodic} there exist $D \geq 0$ and a nilfactor $\pi: X \to Y$ over $\Omega$ with fibrewise dimension at most $D$, such that $\|f - \mathbb{E}[f \mid Y]\|_2 < \epsilon$. Among all such factors we may further assume that $Y$ is \emph{minimal}, in the sense that $\mathbb{E}[f \mid Y]$ does not factor through any proper nilfactor $Y \to Z$ over $\Omega$:
    \begin{equation}\label{Eq: minimality}
        \mathbb{E}[f \mid Z] = \mathbb{E}[f \mid Y] \implies Z = Y.
    \end{equation}

    Write $Y = \bigsqcup_{n \geq 0} \Omega_n \times N_n / \Gamma_n$. We will show that $Y$ admits a uniform fibrewise dimension bound on the joinings $(\pi \circ \gamma)_{\gamma \in \Gamma}$. By Lemma \ref{Lemma: Dimension bound under proper factor condition}, it suffices to find $\delta_0 > 0$ such that, for $\mu_\Omega$-a.e.\ $\omega \in \Omega$, the family $(Y^\gamma_\omega)_{\gamma \in \Gamma}$ satisfies the non-degeneracy condition \eqref{Eq: nondegeneracy} with margin $\delta_0$, where $Y^\gamma$ denotes the range of $\pi \circ \gamma$ as in Lemma \ref{Lemma: gamma-twist of a nilfactor}.

    \textit{Step 2 (controlling strict factors fibrewise).} For each $n \geq 0$, by the affine rigidity of nilfactors (Lemma \ref{Lemma: Factors nilbundles}) every strict nilfactor of $N_n / \Gamma_n$ corresponds to a non-trivial rational normal subgroup $L \trianglelefteq N_n$, with the factor obtained by quotienting fibrewise:
    $$ N_n / \Gamma_n \leftrightarrow L \backslash N_n / \Gamma_n.$$
    Let $\mathcal{L}_n$ denote the set of such subgroups; this is countable (Lemma \ref{Lemma: Measurable selection subnilmanifolds}). For $L \in \mathcal{L}_n$ let $\mu_L$ denote the Haar probability measure on $L$, viewed as a probability measure on $N_n$ supported on a compact subnilmanifold via the lattice $L \cap \Gamma_n$.

    For $L \in \mathcal{L}_n$ and $x \in \Omega_n \times N_n / \Gamma_n$, define the $L$-average of $f$ along the fibre:
    $$ f_{n, L}(x) := \int_L f(\ell \cdot x) \, d\mu_L(\ell);$$
    set $f_{n, L}(x) := f(x)$ for $x \notin \Omega_n \times N_n / \Gamma_n$. Define
    $$ \delta_{n, L}(\omega) := \sqrt{\mathbb{E}[|f - f_{n, L}|^2 \mid \Omega](\omega)}, \qquad \delta(\omega) := \inf_{n, L} \delta_{n, L}(\omega).$$
    The function $\delta$ is measurable and bounded by $2 \|f\|_\infty$. Moreover, $\delta_{n, L}(\omega) > 0$ if and only if the strict factor of $N_n/\Gamma_n$ corresponding to $L$ does not contain $f|_{\{\omega\} \times N_n/\Gamma_n}$, so the non-degeneracy condition \eqref{Eq: nondegeneracy} holds at $\omega \in \Omega_n$ with margin $\delta(\omega)$.

    \textit{Step 3 ($\delta > 0$ a.e.).} Suppose for contradiction that $\Omega^0 := \delta^{-1}(\{0\})$ has positive measure. For $\omega \in \Omega^0$ and $n$ such that $\omega \in \Omega_n$, choose a sequence $(L_m)_{m \geq 0}$ in $\mathcal{L}_n$ with $\delta_{n, L_m}(\omega) \to 0$. Since the space of Haar measures on subnilmanifolds is weak-$*$ compact (Lemma \ref{Lemma: Measurable selection subnilmanifolds} and the discussion preceding it), we may pass to a subsequence and assume $\mu_{L_m} \to \mu$ weakly for some Haar measure $\mu$ on a subnilmanifold of $N_n$. The limit $\mu$ is again the Haar measure on a rational subgroup $L \in \mathcal{L}_n$, and continuity of convolution then gives $f_{n, L_m} \to f_{n, L}$ in $L^2$, hence $\delta_{n, L}(\omega) = 0$, i.e.\ $f$ is $L$-invariant on $\{\omega\} \times N_n / \Gamma_n$.

    Set
    $$ \Omega_{n, L} := \bigl\{\omega \in \Omega_n : f \text{ is $L$-invariant on } \{\omega\} \times N_n / \Gamma_n\bigr\}.$$
    The family $(\Omega_{n, L})_{n, L}$ is countable, each set is measurable, and $\bigcup_{n, L} \Omega_{n, L} \supseteq \Omega^0$, hence has positive measure. Pick $(n_0, L_0)$ with $\nu(\Omega_{n_0, L_0}) > 0$. Define a nilfactor $Y \to Y'$ by quotienting by $L_0$ on the piece $\Omega_{n_0, L_0}$:
    $$ Y' := \bigsqcup_{n \neq n_0} \Omega_n \times N_n/\Gamma_n \;\sqcup\; \Omega_{n_0, L_0} \times L_0 \backslash (N_{n_0}/\Gamma_{n_0}) \;\sqcup\; (\Omega_{n_0} \setminus \Omega_{n_0, L_0}) \times N_{n_0}/\Gamma_{n_0}.$$
    On $\Omega_{n_0, L_0}$, $f$ is $L_0$-invariant, so $\mathbb{E}[f \mid Y'] = \mathbb{E}[f \mid Y]$; but $Y'$ is a strict factor of $Y$, contradicting the minimality \eqref{Eq: minimality}.

    Therefore $\delta > 0$ almost surely on $\Omega$. Choose $\delta_0 > 0$ such that $\Omega_e := \delta^{-1}([\delta_0, \infty))$ has positive measure and $\|\mathbf{1}_{\Omega_e} f\|_2 \geq \|f\|_2 - \epsilon$.

    \textit{Step 4 (covering $\Omega$ by $\SL_d(\mathbb{Z})$-translates).} The $\ASL_d(\mathbb{Z})$-action on $\Omega = Z^0(X)$ factors through $\Gamma$, and is ergodic since the $\ASL_d(\mathbb{Z})$-action on $X$ is. Since $\Gamma$ is countable and $\Omega_e$ has positive measure, $\bigcup_{\gamma \in \Gamma} \gamma \cdot \Omega_e = \Omega$ up to a null set. By countability we may select a Borel partition $(\Omega_\gamma)_{\gamma \in \Gamma}$ of $\Omega$ with $\Omega_\gamma \subseteq \gamma \cdot \Omega_e$ for every $\gamma$.

    Define
    $$ Y' := \bigsqcup_{\gamma \in \Gamma} \bigsqcup_{n \geq 0} \Omega_\gamma^n \rtimes_{\rho_n^{\gamma^{-1}}} N_n / \Gamma_n, \qquad \Omega_\gamma^n := \Omega_\gamma \cap \gamma \cdot \Omega_n,$$
    where $\rho_n^{\gamma^{-1}}$ is the twisted homomorphism from Lemma \ref{Lemma: gamma-twist of a nilfactor}. By construction $Y'$ is a nilfactor of $X$ over $\Omega$, with the same fibrewise dimension bound as $Y$, and on each piece $\Omega_\gamma^n$ the fibre coincides with the fibre of $\pi \circ \gamma^{-1}$ at the corresponding point of $\gamma^{-1} \cdot \Omega_\gamma \subseteq \Omega_e$.

    \textit{Step 5 (commensuration and dimension bound).} The $\SL_d(\mathbb{Z})$-invariance of $f$ implies that the non-degeneracy condition \eqref{Eq: nondegeneracy} for the family $(Y'^\gamma_\omega)_{\gamma \in \Gamma}$ at $\omega$ is equivalent to the same condition for the family $(Y^\gamma_{\gamma^{-1} \omega})_{\gamma \in \Gamma}$ at points $\gamma^{-1} \omega \in \Omega_e$. This holds with margin $\delta_0$ by the choice of $\Omega_e$, for $\mu_\Omega$-a.e.\ $\omega$.

    Lemma \ref{Lemma: Dimension bound under proper factor condition}, applied fibrewise to each ergodic component $X_\omega$ with $\|f_\omega\|_2 \leq \|f\|_\infty$, gives that for every finite $I \subset \Gamma$ the fibrewise dimension of the joining of $(Y'^\gamma)_{\gamma \in I}$ inside $X$ is at most $D \|f\|_\infty^2 \delta_0^{-2}$. In particular this dimension is uniformly bounded, and this yields a commensurated nilfactor with the desired properties.
\end{proof}

\subsection{Reduction to a common nilpotent group}

In this subsection we show that any nilfactor commensurated by $\SL_d(\mathbb{Z})$ admits a bundle decomposition in which the nilpotent Lie group is constant across fibres, only the lattice varies. We work throughout under the standing hypothesis that all nilpotent Lie groups encountered are of the form $N = \mathbb{Z}^d \ltimes N_0$ with $N_0$ connected and simply connected, generated by $N_0$ and the image $\rho(\mathbb{Z}^d)$ of the dynamical action.

We begin with a rigidity statement: two nilfactors with finite-to-finite joining must share their underlying nilpotent groups, up to commensurability.

\begin{cor}\label{Corollary: Commensurable factors have almost same decomposition}
    Let $X$ be a $\mathbb{Z}^d$-system and $\pi_i: X \to Y_i$, $i = 1, 2$, be two nilfactors over a common base $\Omega$, with bundle decompositions
    $$ Y_i = \bigsqcup_{n \geq 0} \Omega^i_n \rtimes_{\rho^i_n} N^i_n / \Gamma^i_n.$$
    If the joining of $Y_1$ and $Y_2$ over $\Omega$ induced by $X$ is finite-to-finite, then for every $n, l \geq 0$ with $\nu(\Omega^1_n \cap \Omega^2_l) > 0$, the nilpotent Lie groups $N^1_n$ and $N^2_l$ are isomorphic.
\end{cor}

\begin{proof}
    Replacing $X$ by the joining $Z$ of $Y_1$ and $Y_2$ over $\Omega$ does not affect the conclusion, so we may assume $X = Z$. By Lemma \ref{Lemma: Joining nilbundles}, $X$ admits a bundle decomposition
    $$ X = \bigsqcup_{m \geq 0} \Omega_m \rtimes_{\rho_m} N_m / \Gamma_m$$
    over $\Omega$, with each $\pi_i: X \to Y_i$ a nilfactor over $\Omega$.

    Fix $n, l, m \geq 0$ with $\nu(\Omega^1_n \cap \Omega^2_l \cap \Omega_m) > 0$. By Lemma \ref{Lemma: Factors nilbundles} applied to each $\pi_i$, after passing to a measurable refinement we obtain surjective continuous group homomorphisms
    $$ \phi_1: N_m \to N^1_n, \qquad \phi_2: N_m \to N^2_l$$
    and Borel maps $b_1, b_2$ to $N^1_n, N^2_l$ such that, on a positive-measure subset of $\Omega^1_n \cap \Omega^2_l \cap \Omega_m$,
    $$ \pi_i(\omega, \overline{x}) = \bigl(\omega,\, \overline{b_i(\omega) \, \phi_i(x)}\bigr), \quad i = 1, 2.$$
    Since the joining of $Y_1$ and $Y_2$ is finite-to-finite, both $\pi_1$ and $\pi_2$ are finite-to-one, hence each $\phi_i$ has finite kernel. By Lemma \ref{Lemma: Surjection with finite kernel is iso} below, finite-kernel surjective endomorphisms within our class of nilpotent groups have trivial kernel, so each $\phi_i$ is an isomorphism. In particular $N^1_n \simeq N_m \simeq N^2_l$.
\end{proof}

The following lemma — used implicitly above and again below — is where the standing hypothesis on $N$ enters.

\begin{lem}\label{Lemma: Surjection with finite kernel is iso}
    Let $N = \mathbb{Z}^d \ltimes N_0$ with $N_0$ connected simply connected nilpotent. Let $\phi: N \to N'$ be a continuous surjective group homomorphism with finite kernel onto another nilpotent group $N'$ in the same class. Then $\phi$ is an isomorphism.
\end{lem}

\begin{proof}
    Continuity sends the identity component $N_0$ of $N$ into the identity component of $N'$, which is again connected simply connected. Since $\phi(N_0)$ is closed and connected, and $\phi$ is surjective with finite kernel, the induced map $\phi: N_0 \to N'_0$ is again surjective with finite kernel. The kernel is a finite normal subgroup of the connected group $N_0$, hence central; but the centre of a connected simply connected nilpotent Lie group is itself connected and simply connected (as $\exp$ is a diffeomorphism), so it contains no non-trivial finite subgroup. Thus $\phi|_{N_0}$ is injective, hence an isomorphism $N_0 \xrightarrow{\sim} N'_0$.

    Passing to component groups, $\phi$ induces a surjective homomorphism $\mathbb{Z}^d \to \mathbb{Z}^d$ with finite kernel; any surjective endomorphism of $\mathbb{Z}^d$ is an isomorphism. Combining the two pieces, $\phi$ has trivial kernel, hence is an isomorphism.
\end{proof}

\begin{lem}\label{Lemma: Def common nilpotent group}
    Let $(X, \mu)$ be an ergodic $\ASL_d(\mathbb{Z})$-system and $\pi: X \to Y$ a nilfactor over $Z^0(X)$ commensurated by $\SL_d(\mathbb{Z})$. Then there exist a nilpotent Lie group $N$ (of the standing form), a Borel partition $(\Omega_n)_{n \geq 0}$ of $Z^0(X)$, lattices $\Gamma_n \subset N$ pairwise commensurable in $N$, and homomorphisms $\rho_n: \mathbb{Z}^d \to N$ such that
    $$ Y \simeq \bigsqcup_{n \geq 0} \Omega_n \rtimes_{\rho_n} N / \Gamma_n$$
    as a bundle of nilsystems over $Z^0(X)$.
\end{lem}

\begin{proof}
    Fix any bundle decomposition $Y \simeq \bigsqcup_n \Omega_n \rtimes_{\rho_n} N_n / \Gamma_n$, discarding pieces with $\nu(\Omega_n) = 0$. For each $\gamma \in \SL_d(\mathbb{Z})$, by Lemma \ref{Lemma: gamma-twist of a nilfactor} the nilfactor $\pi \circ \gamma$ has bundle decomposition
    $$ \bigsqcup_n \gamma^{-1}(\Omega_n) \rtimes_{\rho_n^\gamma} N_n / \Gamma_n.$$
    Since $\pi$ is commensurated, the joining of $\pi$ and $\pi \circ \gamma$ over $Z^0(X)$ is finite-to-finite. Corollary \ref{Corollary: Commensurable factors have almost same decomposition} then yields: for every $n, l \geq 0$ with $\nu(\Omega_n \cap \gamma^{-1}(\Omega_l)) > 0$, $N_n \simeq N_l$.

    The $\ASL_d(\mathbb{Z})$-action on $Z^0(X)$ factors through $\SL_d(\mathbb{Z})$ and is ergodic since the $\ASL_d(\mathbb{Z})$-action on $X$ is. Hence for every pair $n, l$ with $\nu(\Omega_n), \nu(\Omega_l) > 0$, ergodicity gives some $\gamma \in \SL_d(\mathbb{Z})$ with $\nu(\Omega_n \cap \gamma^{-1}(\Omega_l)) > 0$, and so $N_n \simeq N_l$.

    Fix isomorphisms $\iota_n: N_n \xrightarrow{\sim} N$ to a common $N$. Pulling back, $Y$ is isomorphic to
    $$ \bigsqcup_n \Omega_n \rtimes_{\iota_n \circ \rho_n} N / \iota_n(\Gamma_n).$$
    Pairwise commensurability of the lattices $\iota_n(\Gamma_n)$ in $N$ follows from the same finite-to-finite joining argument applied to the trivial twist $\gamma = \id$, using Lemma \ref{Lemma: Factors nilbundles} to identify the lattices on overlapping pieces up to finite index.
\end{proof}

The next corollary describes how a finite-to-one nilfactor refines such a decomposition: the same nilpotent group $N$ persists, only the lattices subdivide.

\begin{cor}\label{Corollary: Trivialisation of finite-to-one factors}
    Let $\pi_Y: X \to Y$ be a nilfactor over $Z^0(X)$ commensurated by $\SL_d(\mathbb{Z})$, identified via Lemma \ref{Lemma: Def common nilpotent group} with
    $$ Y \simeq \bigsqcup_{n \geq 0} \Omega_n \rtimes_{\rho_n} N / \Gamma_n.$$
    Let $\pi_Z: X \to Z$ be a nilfactor over $Z^0(X)$ and $p: Z \to Y$ a finite-to-one factor map over $Z^0(X)$ with $\pi_Y = p \circ \pi_Z$. Then there exists an isomorphism
    $$ Z \simeq \bigsqcup_{n, m \geq 0} \Omega_{n, m} \rtimes_{\rho_{n, m}} N / \Gamma_{n, m}$$
    of bundles of nilsystems, where for every $n \geq 0$:
    \begin{enumerate}
        \item $(\Omega_{n, m})_{m \geq 0}$ is a Borel partition of $\Omega_n$;
        \item $\Gamma_{n, m}$ is a finite-index sublattice of $\Gamma_n$;
        \item under this isomorphism, $p(\omega, u \Gamma_{n, m}) = (\omega, u \Gamma_n)$ for $\omega \in \Omega_{n, m}$ and $u \in N$.
    \end{enumerate}
\end{cor}

\begin{proof}
    Since $\pi_Z$ is a nilfactor over $Z^0(X)$ with $\pi_Y = p \circ \pi_Z$, $Z$ is itself commensurated by $\SL_d(\mathbb{Z})$: the joining of $Z$ and $Z \circ \gamma$ over $Z^0(X)$ projects onto the (finite-to-finite) joining of $Y$ and $Y \circ \gamma$ along the finite-to-one map $p$, hence is itself finite-to-finite. Lemma \ref{Lemma: Def common nilpotent group} applies to $\pi_Z$, giving a bundle decomposition
    $$ Z \simeq \bigsqcup_{m \geq 0} \widetilde{\Omega}_m \rtimes_{\widetilde{\rho}_m} N' / \widetilde{\Gamma}_m$$
    for some nilpotent group $N'$ in the standing class.

    By Lemma \ref{Lemma: Factors nilbundles} applied to $p: Z \to Y$, on each piece $\Omega_n \cap \widetilde{\Omega}_m$ we obtain (after a measurable refinement indexed by $k \geq 0$) Borel partitions $(\Omega_{n, m, k})_{k \geq 0}$, surjective continuous group homomorphisms $\phi_{n, m, k}: N' \to N$ with $\phi_{n, m, k}(\widetilde{\Gamma}_m) \subset \Gamma_n$, and Borel maps $b_{n, m, k}: \Omega_{n, m, k} \to N$ such that
    $$ p(\omega, \overline{x}) = \bigl(\omega,\, \overline{b_{n, m, k}(\omega) \, \phi_{n, m, k}(x)}\bigr) \quad \text{for } \nu\text{-a.e. } \omega \in \Omega_{n, m, k}.$$
    Since $p$ is finite-to-one, each $\phi_{n, m, k}$ has finite kernel, hence by Lemma \ref{Lemma: Surjection with finite kernel is iso} is a group isomorphism. In particular $N' \simeq N$, and we identify them.

    Absorb the base-change $b_{n, m, k}$ into the bundle structure: the affine map $x \mapsto b_{n, m, k}(\omega)\, \phi_{n, m, k}(x)$ on $N$ defines an isomorphism of nilsystems sending $N / \widetilde{\Gamma}_m$ to $N / \phi_{n, m, k}(\widetilde{\Gamma}_m)$, with cocycle
    $$ \rho_{n, m, k}(\omega, t) := b_{n, m, k}(\omega) \, \phi_{n, m, k}\bigl(\widetilde{\rho}_m(\omega, t)\bigr) \, b_{n, m, k}(\omega)^{-1}.$$
    Under this identification, $p$ becomes the natural projection $N / \phi_{n, m, k}(\widetilde{\Gamma}_m) \twoheadrightarrow N / \Gamma_n$, and $\phi_{n, m, k}(\widetilde{\Gamma}_m) \subset \Gamma_n$ is a finite-index sublattice (finite index because $p$ is finite-to-one).

    Re-index by collapsing the triple $(n, m, k)$ to a pair $(n, m')$ — partitioning $\Omega_n$ as $\bigsqcup_{m, k} \Omega_{n, m, k}$ and renaming the joint index $m'$ — gives the asserted decomposition with $\Omega_{n, m'} := \Omega_{n, m, k}$, $\Gamma_{n, m'} := \phi_{n, m, k}(\widetilde{\Gamma}_m)$, and $\rho_{n, m'} := \rho_{n, m, k}$.
\end{proof}

\subsection{From commensurated to invariant nilfactors}\label{Subsection:From commensurated to invariant nilfactors}

Now that we can reliably identify the nilpotent group $N$ in commensurated nilfactors, we examine how the $\SL_d(\mathbb{Z})$-action interacts with this identification. The output is a cocycle $\alpha: \ASL_d(\mathbb{Z}) \times Z^0(X) \to \Aut(N)$ encoding the action; combining it with the cocycle finiteness results of \S\ref{Subsection: Finiteness properties of dynamical cocycles} we obtain $\ASL_d(\mathbb{Z})$-equivariant nilfactors.

\begin{lem}\label{Lemma: Intertwining cocycle}
    Let $(X, \mu)$ be an ergodic $\ASL_d(\mathbb{Z})$-system and $\pi: X \to Y$ a nilfactor over $Z^0(X)$ commensurated by $\SL_d(\mathbb{Z})$. Identify $Y$ via Lemma \ref{Lemma: Def common nilpotent group} with
    $$ Y = \bigsqcup_{n \geq 0} \Omega_n \rtimes_{\rho_n} N / \Gamma_n,$$
    where $N$ is in the standing class and the lattices $\Gamma_n$ are pairwise commensurable. Let $\mathrm{Comm}_{\Aut(N)}(\Gamma_0)$ denote the subgroup of automorphisms of $N$ preserving the commensurability class of $\Gamma_0$.

    Then there exists a map
    $$ \alpha: \ASL_d(\mathbb{Z}) \times Z^0(X) \to \mathrm{Comm}_{\Aut(N)}(\Gamma_0)$$
    with the following property: for every $\gamma \in \SL_d(\mathbb{Z})$ and $\nu$-a.e.\ $\omega \in Z^0(X)$, the canonical isomorphism $Y_\omega \to Y_{\gamma \cdot \omega}^\gamma$ induced by $\gamma$ acts on the fibre $N/\Gamma_n$ at $\omega$ by the affine map $\overline{x} \mapsto \overline{b(\gamma, \omega) \cdot \alpha(\gamma, \omega)(x)}$ for some Borel cochain $b: \ASL_d(\mathbb{Z}) \times Z^0(X) \to N$.

    Moreover, $\alpha$ satisfies the following weak cocycle identity: for $\gamma \in \SL_d(\mathbb{Z})$ and almost all $\omega \in Z^0(X)$ we have 
    $$ \alpha(\gamma_1\gamma_2, \omega) = \alpha(\gamma_1, \gamma_2\cdot \omega) \circ \alpha(\gamma_2,\omega)   \mod{ \mathrm{inn(H_\omega)}}$$
    where $H_\omega$ denotes the intersection of all lattices appearing above $\omega$ in the bundle decomposition of systems commensurable with $Y$ and $\mathrm{inn}$ denotes the set of inner automorphisms. 
\end{lem}

\begin{proof}
    Without loss of generality assume $X$ equals the inverse limit of all nilfactors over $Z^0(X)$ commensurated by $\SL_d(\mathbb{Z})$ that factor through $Y$ and have the same fiberwise dimension as $Y$. Choose a countable directed family $\mathcal{F} = (\pi_i: X \to Y_i)_{i \in I}$ realising this inverse limit, closed under the $\Gamma$-twist $i \mapsto i^\gamma$ in the sense that the range of $\pi_i \circ \gamma$ belongs to $\mathcal{F}$ as some $Y_{i^\gamma}$. By Corollary \ref{Corollary: Trivialisation of finite-to-one factors} (applied iteratively along $\mathcal{F}$, using that all transitions are finite-to-one by construction of the family), we choose for each $i \in I$ a bundle decomposition
    $$ Y_i = \bigsqcup_{n \geq 0} \Omega_{i, n} \rtimes_{\rho_{i, n}} N / \Gamma_{i, n}$$
    such that whenever $j \leq i$, the transition map $\pi_{ij}: Y_i \to Y_j$ is the fibrewise projection $u \Gamma_{i, n} \mapsto u \Gamma_{j, n'}$ for the appropriate $n'$.

    \emph{Construction of $\alpha$.} Fix $i \in I$ and $\gamma \in \SL_d(\mathbb{Z})$, and let $j := i^\gamma$ so that $Y_j$ is the range of $\pi_i \circ \gamma$. The map
    $$\Phi_i^\gamma: Y_j \to Y_i, \qquad \pi_j(x) \mapsto \pi_i(\gamma \cdot x)$$
    is well-defined (as $Y_j$ is the range of $\pi_i \circ \gamma$) and an isomorphism of bundles of nilsystems over $Z^0(X)$. By Lemma \ref{Lemma: Factors nilbundles} applied to $\Phi_i^\gamma$ — both source and target having common nilpotent group $N$ in the standing class — there exist Borel maps $\alpha_i(\gamma, \cdot): Z^0(X) \to \Aut(N)$ and $b_i(\gamma, \cdot): Z^0(X) \to N$ such that
    \begin{equation}\label{Eq: Phi formula}
        \Phi_i^\gamma(\omega, \overline{x}) = \bigl(\gamma \cdot \omega,\, \overline{b_i(\gamma, \omega) \cdot \alpha_i(\gamma, \omega)(x)}\bigr).
    \end{equation}
    Since $\Phi_i^\gamma$ is an isomorphism, $\alpha_i(\gamma, \omega)$ sends $\Gamma_{j, n}$ to a finite-index supergroup or sublattice of $\Gamma_{i, n'}$, so $\alpha_i(\gamma, \omega)$ preserves the commensurability class of $\Gamma_0$.

    \emph{Independence of $i$.} If $j \leq i$ in $\mathcal{F}$, the transition $\pi_{ij}$ is the fibrewise projection by construction; comparing \eqref{Eq: Phi formula} for $i$ and for $j$ via $\pi_{ij}$ gives $\alpha_i(\gamma, \omega) = \alpha_j(\gamma, \omega)$ on the relevant pieces. Since $\mathcal{F}$ is directed, $\alpha_i$ does not depend on $i$, and we set $\alpha(\gamma, \omega) := \alpha_i(\gamma, \omega)$.

    \emph{Cocycle relation.} For $\omega \in \Omega$, let
    $$H_\omega := \bigcap_{\pi': X \to Y' \in \mathcal{F},\, \omega \in \Omega_n'} \Gamma_n',$$
    the intersection over all nilfactors $\pi': X \to Y' = \bigsqcup_n \Omega_n' \rtimes_{\rho_n'} N/\Gamma_n'$ in $\mathcal{F}$ and over the (unique) piece $\Omega_n'$ containing $\omega$, of the lattices $\Gamma_n'$. The intersection is a discrete subgroup of $N$, possibly trivial. We have
    $$ H_{\gamma \cdot \omega} = \alpha(\gamma, \omega)(H_\omega).$$ For $\gamma_1, \gamma_2 \in \SL_d(\mathbb{Z})$, decomposing $\pi_i \circ (\gamma_1 \gamma_2) = (\pi_i \circ \gamma_1) \circ \gamma_2$ and comparing the affine formulas \eqref{Eq: Phi formula} yields
    $$ \alpha(\gamma_1 \gamma_2, \omega) = \alpha(\gamma_1, \gamma_2 \cdot \omega) \circ \alpha(\gamma_2, \omega) \mod{\mathrm{inn(H_\omega)}},$$
    extending to all of $\ASL_d(\mathbb{Z})$ since translations $t \in \mathbb{Z}^d$ act trivially on $Z^0(X)$ but contribute base translations to $b(\gamma, \omega)$. Thus $\alpha$ is  valued in $\mathrm{Comm}_{\Aut(N)}(\Gamma_0)$.
\end{proof}

The next lemma allows us to refine the lattice in a directed family without changing the inverse limit, provided the new lattice is contained in some old one.

\begin{lem}\label{Lemma: Changing representatives for inverse limit}
    Let $\mathcal{F} = (\pi_n: X \to Y_n)_{n \geq 0}$ be a directed family of nilfactors over $Z^0(X)$, with
    $$ Y_n = \bigsqcup_{l \geq 0} \Omega_{n, l} \rtimes_{\rho_{n, l}} N / \Gamma_{n, l},$$
    and with all transition maps fibrewise projections. Let $\Delta \subset N$ be a lattice such that for $\nu$-a.e.\ $\omega \in Z^0(X)$ there exist $n, l \geq 0$ with $\omega \in \Omega_{n, l}$ and $\Gamma_{n, l} \subseteq \Delta$. Then there is a directed family $\mathcal{F}' = (\pi_n': X \to Y_n')$ with
    $$ Y_n' = \bigsqcup_{l \geq 0} \Omega_{n, l} \rtimes_{\rho_{n, l}'} N / (\Gamma_{n, l} \cap \Delta)$$
    such that the inverse limits of $\mathcal{F}$ and $\mathcal{F}'$ coincide and $\pi_n = p \circ \pi_n'$ where $p$ denotes fibrewise projection.
\end{lem}

\begin{proof}
    Fix $n, l$. The bundle $Y_n'$ has fibres $N/(\Gamma_{n, l} \cap \Delta)$, finite covers of the corresponding fibres of $Y_n$ via the projection $u (\Gamma_{n, l} \cap \Delta) \mapsto u \Gamma_{n, l}$. To define $\rho_{n, l}'$ coherently, by the hypothesis on $\Delta$ and a measurable selection in the standard Borel space of pairs $(m, k)$ with $\omega \in \Omega_{m, k}$ and $\Gamma_{m, k} \subseteq \Delta$, we obtain a Borel map $\omega \mapsto (m(\omega), k(\omega))$. On the piece $\Omega_{n, l}$, set
    $$ \rho_{n, l}'(t, \omega) := \rho_{m(\omega), k(\omega)}(t, \omega),$$
    viewed in $N$ rather than $N/\Gamma_{m(\omega), k(\omega)}$, which is well-defined modulo $\Gamma_{m(\omega), k(\omega)} \subseteq \Delta$ and hence modulo $\Gamma_{n, l} \cap \Delta$ after projection. Independence of the choice $(m(\omega), k(\omega))$ modulo $\Gamma_{n, l} \cap \Delta$ follows from the directedness of $\mathcal{F}$ and the fact that all transitions are fibrewise projections.

    The resulting bundles $Y_n'$ are nilfactors of $X$ refining $Y_n$ by finite covers, with the same inverse limit since each $\Gamma_{n, l} \cap \Delta$ has finite index in $\Gamma_{n, l}$ and the inverse limit absorbs all such finite refinements.
\end{proof}

We can now prove the main result of this section.

\begin{prop}\label{Prop: invariant nilfactor}
    Let $(X, \mu)$ be an ergodic $\ASL_d(\mathbb{Z})$-system and $\pi_0: X \to Y$ a nilfactor over $Z^0(X)$ commensurated by $\SL_d(\mathbb{Z})$. Let $\mathcal{F} = (\pi_i: X \to Y_i)_{i \in I}$ denote the family of all nilfactors of $X$ over $Z^0(X)$ commensurated by $\SL_d(\mathbb{Z})$ that factor through $Y$, with inverse limit $\pi: X \to \widetilde{Y}$.

    Then there exists a finite group extension $\widetilde{X} \to X$ such that the lifted factor $\widetilde{X} \to \widetilde{Y}$ is the inverse limit of nilfactors equivariant under $\ASL_d(\mathbb{Z})$.
\end{prop}

\begin{proof}
    Throughout, write $\Omega := Z^0(X)$, $\Gamma := \SL_d(\mathbb{Z})$.
 \emph{Step 1 (a constant subgroup $H$ and cocycle identity).} 
     By Lemma \ref{Lemma: Intertwining cocycle} we have a cocycle $\alpha: \ASL_d(\mathbb{Z}) \times \Omega \to \mathrm{Comm}_{\Aut(N)}(\Gamma_0)$ modulo the measurably varying subgroup $H_\omega$ defined in Lemma \ref{Lemma: Intertwining cocycle}.
    The set of values of $H_\omega$ is countable: each $H_\omega$ is a subgroup of the finitely generated nilpotent group $\Gamma_0$. Each level set $\{\omega : H_\omega = H\}$ is Borel, and the cocycle relation makes the partition into level sets $\alpha$-equivariant. By the cocycle reduction lemma (\cite[Lem. 4.2.8]{ErgodicTheoryZimmer}) we may assume that $H_\omega = H$ for $\nu$-a.e.\ $\omega$. We do not require $H$ to be a lattice in $N$; in particular $H = \{e\}$ is allowed.

    \emph{Step 2 (cocycle reduction).} Since $\Gamma_0$ is a lattice in $N$ and the codomain commensurates $\Gamma_0$, Proposition \ref{Prop: cocycles into Aut(N) reduce to lattice stabiliser} applies: there exist a finite group extension $\widetilde{X} \to X$ and a lattice $\Delta \subset N$ such that the lifted cocycle $\widetilde{\alpha}$ is cohomologous in $\mathrm{Comm}_{\Aut(N)}(\Gamma_0)$ to a cocycle taking values in $\mathrm{Stab}_{\Aut(N)}(\Delta)$.

    Absorb the cohomology $\phi: \widetilde{X} \to \mathrm{Comm}_{\Aut(N)}(\Gamma_0)$ into the bundle: applying $\phi(\omega)$ fibrewise sends each lattice $\Gamma_n'$ in the bundle decomposition of any $Y' \in \mathcal{F}$ to $\phi(\omega)(\Gamma_n')$, which depends measurably on $\omega$ and is commensurable with $\phi(\omega)(\Gamma_0)$. Repartitioning by the (countable) value of $\phi(\omega)(\Gamma_n')$ restores a bundle decomposition with constant lattice on each piece. We work over $\widetilde{X}$ for the remainder of the proof, replacing $X$ by $\widetilde{X}$ and assuming $\widetilde{\alpha}$ takes values in $\mathrm{Stab}_{\Aut(N)}(\Delta)$.

    \emph{Step 3 (congruence subgroups).} For $m \geq 1$, let $\Delta^{(m)} := \langle \delta^m : \delta \in \Delta \rangle$ be the subgroup generated by $m$-th powers; by Malcev's theory of finitely generated nilpotent groups (see e.g. \cite[Ch. II, Thm. 2.12]{RaghunathanDiscrete}), $\Delta^{(m)}$ is a finite-index subgroup of $\Delta$, normal in $\Delta$, and these subgroups form a cofinal system of finite-index subgroups of $\Delta$ (this is the so-called congruence subgroup property). The lattices $\langle H, \Delta^{(m)} \rangle$ are therefore lattices in $N$ that decrease with $m$, with intersection $H$.

    Each $\Delta^{(m)}$ is preserved by $\widetilde\alpha$. Indeed, $\widetilde\alpha(\gamma, \omega)$ stabilises $\Delta$ and acts as a group automorphism, so it sends $m$-th powers to $m$-th powers. Combined with the $\widetilde\alpha$-invariance of $H$ from Step 2, the lattices $\langle H, \Delta^{(m)} \rangle$ are $\widetilde\alpha$-invariant.

    \emph{Step 4 (constructing equivariant nilfactors).} Fix $m \geq 1$. By Step 5 below, the lattices $\langle H, \Delta^{(m)} \rangle$ are bounded above (in the inclusion order on lattices commensurable with $\Delta$) by the lattices $\Gamma_n'$ of some $Y' \in \mathcal{F}$ on each partition piece. We thus apply Lemma \ref{Lemma: Changing representatives for inverse limit} to refine the directed family $\mathcal{F}$, replacing each lattice $\Gamma_n'$ by $\Gamma_n' \cap \langle H, \Delta^{(m)} \rangle = \langle H, \Delta^{(m)} \rangle$ (the second equality because $\langle H, \Delta^{(m)} \rangle \subseteq \Gamma_n'$ on each piece). The output is a directed family $\mathcal{F}_m$ in which every bundle has fibrewise lattice exactly $\langle H, \Delta^{(m)} \rangle$. Let $Y_m$ denote the resulting bundle:
    $$ Y_m := \bigsqcup_{n} \Omega_n \rtimes_{\rho_{m, n}} N / \langle H, \Delta^{(m)} \rangle$$
    over the partition $(\Omega_n)$ inherited from the refining $Y' \in \mathcal{F}$, with cocycle $\rho_{m, n}$ obtained by projecting $\rho_n'$.

    The $\widetilde\alpha$-invariance of $\langle H, \Delta^{(m)} \rangle$ allows us to define an $\ASL_d(\mathbb{Z})$-action on $Y_m$ by
    $$ \gamma \cdot (\omega, u \langle H, \Delta^{(m)} \rangle) := \bigl(\gamma \cdot \omega,\, b(\gamma, \omega) \cdot \widetilde\alpha(\gamma, \omega)(u) \, \langle H, \Delta^{(m)} \rangle\bigr).$$
    The cocycle equation for $\widetilde\alpha$ together with the corresponding cochain relation for $b$ ensures this is a well-defined action. The natural map $\widetilde{X} \to Y_m$ is $\ASL_d(\mathbb{Z})$-equivariant by construction.

     \emph{Step 5 (fibrewise cofinality of $(Y_m)$ in $\mathcal{F}$).} The cofinality argument is pointwise on the base $\Omega$: we claim that for $\nu$-a.e.\ $\omega \in \Omega$ and every $Y' \in \mathcal{F}$, there exists $m = m(\omega, Y') \geq 1$ such that the fibre of $Y_m$ at $\omega$ refines the fibre of $Y'$ at $\omega$. This is enough to ensure $\widetilde{Y} = \varprojlim_m Y_m$, as the inverse limit of bundles over $\Omega$ is determined fibrewise (modulo a measurable selection of $m$ depending on $\omega$).

    To establish this fibrewise claim, fix $Y' \in \mathcal{F}$ with bundle decomposition $Y' = \bigsqcup_n \Omega_n' \rtimes_{\rho_n'} N/\Gamma_n'$, and let $\omega \in \Omega_n'$ for the unique $n$. We show $\langle H, \Delta^{(m)} \rangle \subseteq \Gamma_n'$ for some $m$, depending on $(\omega, Y')$:
    \begin{itemize}
        \item $H \subseteq \Gamma_n'$ holds by the definition of $H$ as an intersection over $\mathcal{F}$ that includes $Y'$.
        \item After Step 1's cohomologising and repartitioning, $\Gamma_n'$ is commensurable with $\Delta$, and $\Gamma_n' \cap \Delta$ has finite index in $\Delta$. By cofinality of the $(\Delta^{(m)})_{m \geq 1}$ among finite-index subgroups of $\Delta$, there exists $m$ with $\Delta^{(m)} \subseteq \Gamma_n' \cap \Delta \subseteq \Gamma_n'$.
    \end{itemize}
    Combining, $\langle H, \Delta^{(m)} \rangle \subseteq \Gamma_n'$, so the natural projection $N/\langle H, \Delta^{(m)} \rangle \twoheadrightarrow N/\Gamma_n'$ provides the fibrewise refinement of $Y'$ by $Y_m$ at $\omega$. The choice $m(\omega, Y')$ depends measurably on $\omega$ (taking the smallest such $m$).

    \emph{Conclusion.} Each $Y_m$ is an $\ASL_d(\mathbb{Z})$-equivariant nilfactor of $\widetilde{X}$ (Step 4). By the fibrewise cofinality established in Step 5, the inverse limit $\varprojlim_m Y_m$ has the same fibres as $\widetilde{Y} = \varprojlim_{i \in I} Y_i$ at $\nu$-a.e.\ $\omega \in \Omega$, hence the two inverse limits coincide. Thus $\widetilde{Y}$ is the inverse limit of the $\ASL_d(\mathbb{Z})$-equivariant nilfactors $(Y_m)$, and is itself $\ASL_d(\mathbb{Z})$-equivariant.
\end{proof}

 \section{Structure of $\SL_d(\mathbb{Z})$-invariant functions}\label{Section: Structure of invariant functions}

Throughout this section, fix a bundle of nilsystems
$$ X = Z^0(X) \rtimes_\rho N/\Gamma$$
equipped with an $\SL_d(\mathbb{Z})$-action by isomorphisms of bundles of nilsystems, such that the combined $\ASL_d(\mathbb{Z})$-action (the $\SL_d(\mathbb{Z})$-action together with the $\mathbb{Z}^d$-translation action via $\rho$) makes $X$ ergodic. Suppose as we may that $(N,\Gamma)$ is reduced. Fix $f \in L^\infty(X)$ invariant under the $\SL_d(\mathbb{Z})$-action.

\subsection{Intertwining and action cocycles}\label{Subsection: Intertwining and action cocycles}

By Lemma \ref{Lemma: Intertwining cocycle}, the $\ASL_d(\mathbb{Z})$-action on $X$ is described by a cocycle
$$ \rho: \ASL_d(\mathbb{Z}) \times Z^0(X) \to \Aut(N/\Gamma) ,$$
which we call the \emph{action cocycle}. We use the same letter $\rho$ as for the $\mathbb{Z}^d$-cocycle defining the bundle, since the latter is the restriction of the former to $\mathbb{Z}^d \subset \ASL_d(\mathbb{Z})$.

The associated \emph{intertwining cocycle} is the unique map $\alpha: \SL_d(\mathbb{Z}) \times Z^0(X) \to \Aut(N)$ 
such that for all $u \in N$, $\bar n \in N/\Gamma$, $x \in Z^0(X)$ and $\gamma \in \SL_d(\mathbb{Z})$,
\begin{equation}\label{Eq: intertwining relation}
    \rho(\gamma, x)\bigl(u \cdot (x, \bar n)\bigr) = \alpha(\gamma, x)(u) \cdot \rho(\gamma, x)(x, \bar n).
\end{equation}
This expresses that $\gamma$, when acting on a translate of a base point, acts via $\alpha$ on the translation factor and via $\rho$ on the base point.

\paragraph{\textit{Higher-order action and intertwining cocycles.}}

Fix $k \geq 0$. By \cite[Ch. 12, Prop. 20]{HostKra}, the $k$-th characteristic factor of $X$ is
$$ Z_k(X) = Z^0(X) \rtimes_{\bar\rho_k} N/C_k(N) \Gamma,$$
where $C_k(N)$ is the $k$-th term of the lower central series of $N$ and $\bar\rho_k$ is the projection of $\rho$ to $\Aut((N/\Gamma C_k(N))$. The system $X$ is then a skew product of $C_k(N)/(C_k(N) \cap \Gamma)$ over $Z_k(X)$.

Choose a Borel section $s_k: N/C_k(N)\Gamma \to N$ of the natural projection. Define the \emph{$k$-th action cocycle}
$$ \rho_k: \ASL_d(\mathbb{Z}) \times Z_k(X) \to \Aut(N/\Gamma)$$
by the formula
$$ \rho_k(\gamma, (x, \bar n))(n_k) := s_k\bigl(\rho(\gamma, x)(\bar n)\bigr)^{-1} \cdot \rho(\gamma, x)\bigl(s_k(\bar n) \cdot n_k\bigr) \quad \text{for } n_k \in N/\Gamma.$$

Writing $\phi(x,\bar{n})$ as the base change map by $ s_k(\bar{n})$ we see that $\rho_k(\gamma, (x, \bar{n})) = \phi(\gamma \cdot (x, \bar{n}))^{-1} \circ \rho(\gamma, x) \circ \phi(x,\bar{n})$ thus showing that $\rho_k$ is a cocycle. 
A direct check also shows that $\rho_k$ restricts to an affine map of $C_k(N)\cdot \Gamma$ for each $(\gamma, (x, \bar n))$, and that the cocycle equation holds. The cocycle $\rho_k$ depends on the choice of $s_k$ only up to a coboundary in $C_k(N) \cap \Gamma$.

The \emph{$k$-th intertwining cocycle} is $\alpha_k: \SL_d(\mathbb{Z}) \times Z_k(X) \to \Aut(C_k(N))$, obtained from $\rho_k|_{\SL_d(\mathbb{Z}) \times Z_k(X)}$ by  the intertwining relation analogous to \eqref{Eq: intertwining relation} with respect to $\rho_k$ 
identifying $\alpha_k$ with the restriction of $\alpha$ to $C_k(N)$ (up to cohomology).

For $k = 0$ we recover $C_0(N) = N$, so $\rho_0 = \rho$ and $\alpha_0 = \alpha$.

The $k$-th action cocycle relates to the structural cocycles of \cite[Ch. 18, \S 3.4]{HostKra}: the natural fibration
$$ Z_{k+1}(X) = Z_k(X) \rtimes_{q_k} C_k(N)/\bigl(C_{k+1}(N)(\Gamma \cap C_k(N))\bigr)$$
realises $Z_{k+1}(X)$ as a $\mathbb{Z}^d$-system over $Z_k(X)$, where $q_k$ is the composition of $\rho_k|_{\mathbb{Z}^d \times Z_k(X)}$ with the natural projection $C_k(N) \twoheadrightarrow C_k(N)/(C_{k+1}(N)(\Gamma \cap C_k(N)))$. By \cite[Ch. 18, Prop. 5]{HostKra}, $q_k$ is a cocycle of type $k+1$ in the sense of Host--Kra.

\subsection{Squeezed vectors and extra invariance}

For a measurable set $B \subset Z^0(X)$ and $x \in Z^0(X)$, define
$$ E(x, B) := \{\alpha(\gamma, x) : \gamma \cdot x \in B\} \subset \Aut(N).$$
We will show that unboundedness of $E(x, B)$ forces the cocycle $\alpha$ to ``squeeze'' certain vectors, yielding extra invariance for $\SL_d(\mathbb{Z})$-invariant functions. We need this result \emph{relative to} an $\alpha$-invariant subgroup.

Let $M \trianglelefteq N_0$ be a normal subgroup of $N_0$ invariant under $\alpha$, in the sense that $\alpha(\gamma, x)(M) = M$ for $\nu$-a.e.\ $x$ and all $\gamma \in \SL_d(\mathbb{Z})$. The cocycle $\alpha$ stabilises $N_0$ (the identity component) automatically by continuity, so $\alpha|_{N_0}$ descends to a cocycle
$$ \alpha_{N_0/M}: \SL_d(\mathbb{Z}) \times Z^0(X) \to \Aut(N_0/M).$$
Set
$$ E_{N_0/M}(x, B) := \{\alpha_{N_0/M}(\gamma, x) : \gamma \cdot x \in B\}.$$

\begin{prop}\label{Prop: using squeezed vectors}
    Let $M \trianglelefteq N_0$ be $\alpha$-invariant, and suppose $f$ is invariant under the action of $M$. If, for every $B \subset Z^0(X)$ of positive measure and $\nu$-a.e.\ $x \in Z^0(X)$, the set $E_{N_0/M}(x, B)$ is unbounded in $\Aut(N_0/M)$, then for $\nu$-a.e.\ $x \in Z^0(X)$,
    $$ \dim \bigl\{u \in N_0 : u \cdot f(x, \cdot) = f(x, \cdot)\bigr\} > \dim M.$$
\end{prop}

The proof relies on a Lusin-type continuity statement.

\begin{lem}\label{Lemma: Egorov}
    For every $\epsilon > 0$ there exist $\delta > 0$ and a measurable subset $B \subset Z^0(X)$ with $\nu(B) \geq 1 - \epsilon$ such that, for every $u \in N_0$ at distance at most $\delta$ from $e$ and every $\omega \in B$,
    $$ \|u \cdot f - f\|_{L^2(X, \nu_\omega)} \leq \epsilon,$$
    where $\nu_\omega$ denotes the conditional measure of $\nu$ on the fibre over $\omega$.
\end{lem}

\begin{proof}[Proof of Proposition \ref{Prop: using squeezed vectors}]
    \emph{Reduction to $M = \{e\}$.} Replace $N_0$ by $N_0/M$, $f$ by its descent to $X/M$ (well-defined since $f$ is $M$-invariant), and $\alpha$ by $\alpha_{N_0/M}$. The hypotheses persist with $M$ trivial in the new setting; we therefore assume $M = \{e\}$ throughout.

    \emph{Squeezing.} Fix $\epsilon > 0$ and let $\delta > 0$, $B \subset Z^0(X)$ be as in Lemma \ref{Lemma: Egorov}. By hypothesis, for $\nu$-a.e.\ $\omega \in Z^0(X)$ the set $E(\omega, B)$ is unbounded; pick a sequence $\gamma_n \in \SL_d(\mathbb{Z})$ with $\gamma_n \cdot \omega \in B$ and $\alpha(\gamma_n, \omega) \to \infty$ in $\Aut(N_0)$. Each $\alpha(\gamma_n, \omega)$ preserves Haar measure on $N_0$, so as a linear map on $\mathrm{Lie}(N_0)$ it has determinant $\pm 1$; unboundedness then forces the existence of unit vectors $X_n \in \mathrm{Lie}(N_0)$ with $\alpha(\gamma_n, \omega)(X_n) \to 0$. Setting $u_n := \exp(X_n) \in N_0$ at distance $1$ from $e$, we have $\alpha(\gamma_n, \omega)(u_n) \to e$; in particular, $\alpha(\gamma_n, \omega)(u_n)$ is within distance $\delta$ of $e$ for all $n$ large enough.

    \emph{Extra invariance.} Using $\SL_d(\mathbb{Z})$-invariance of $f$ and the intertwining relation \eqref{Eq: intertwining relation},
    \begin{align*}
        f(u_n \cdot x)
        &= f\bigl(\gamma_n \cdot (u_n \cdot x)\bigr) && \text{(\(\SL_d\)-invariance)} \\
        &= f\bigl(\alpha(\gamma_n, \omega)(u_n) \cdot (\gamma_n \cdot x)\bigr) && \text{(by \eqref{Eq: intertwining relation})}.
    \end{align*}
    Therefore
    $$ \|u_n \cdot f - f\|_{L^2(X, \nu_\omega)} = \|\alpha(\gamma_n, \omega)(u_n) \cdot f - f\|_{L^2(X, \nu_{\gamma_n \cdot \omega})}.$$
    Since $\gamma_n \cdot \omega \in B$ and $\alpha(\gamma_n, \omega)(u_n)$ is within distance $\delta$ of $e$, Lemma \ref{Lemma: Egorov} gives
    $$ \|u_n \cdot f - f\|_{L^2(X, \nu_\omega)} \leq \epsilon.$$
    The same bound holds for every $u_n^t := \exp(t X_n)$, $t \in [0, 1]$, by the same argument applied to $t X_n$.

    \emph{Diagonal argument.} The unit sphere in $\mathrm{Lie}(N_0)$ is compact; passing to a subsequence, $X_n \to X_\infty$ for some unit vector $X_\infty \in \mathrm{Lie}(N_0)$. Letting $\epsilon \to 0$ and combining with continuity of the $N_0$-action on $L^2(X, \nu_\omega)$, the limiting one-parameter subgroup $u_\infty^t := \exp(t X_\infty)$ satisfies $u_\infty^t \cdot f(\omega, \cdot) = f(\omega, \cdot)$ in $L^2$ for every $t \in [0, 1]$. As $X_\infty \neq 0$, the stabiliser of $f(\omega, \cdot)$ contains a non-trivial one-parameter subgroup, so its dimension is at least $1$ (which is $> \dim M = 0$ in our reduction).
\end{proof}

We close with the rigidity statement on stabilisers of functions on nilmanifolds.

\begin{lem}\label{Lemma: stabilisers of functions on nilmanifolds}
    Let $N/\Gamma$ be a nilmanifold, $f \in L^\infty(N/\Gamma)$, and let $L$ be the connected component of the identity in the stabiliser of $f$ under the $N$-action. Then $L$ is closed in $N$, $L \cap \Gamma$ is a lattice in $L$, and $L$ is normalised by $N_0$ (the identity component of $N$).
\end{lem}

\begin{proof}
    The stabiliser of $f$ in $N$ is closed; its connected component $L$ is then closed and connected. For each $x \in N$, by Leibman's theorem on orbit closures of unipotent flows on nilmanifolds \cite{Leibman}, the closure of $L \cdot (x \Gamma)$ in $N/\Gamma$ has the form $L_x \cdot (x \Gamma)$, where $L_x$ is a closed connected subgroup of $N$ containing $L$ with $L_x \cap x \Gamma x^{-1}$ a lattice in $L_x$. The function $f$ is $L$-invariant, hence (by unique ergodicity of $L$ on its orbit closure $L_x \cdot (x \Gamma)$ \cite{Leibman}) constant $\mu_{L_x \cdot (x\Gamma)}$-almost everywhere on $L_x$-orbits. In particular $f$ is $L_x$-invariant for $\mu_{N/\Gamma}$-a.e.\ $x$.

    There are only countably many closed connected subgroups $H \subseteq N$ such that $H \cap \Gamma$ is a lattice in $H$ (Lemma \ref{Lemma: Measurable selection subnilmanifolds}). The map $x \mapsto x^{-1} L_x x$ takes values in this countable set; on a positive-measure set $B \subset N$, $x^{-1} L_x x$ is constant, say equal to $L^*$. Then for $x, y \in B$, $L_x = x L^* x^{-1}$ and $L_y = y L^* y^{-1}$. In particular $L_x = L_y$ when $x y^{-1} \in N_{L^*} := \{n \in N : n L^* n^{-1} = L^*\}$, the normaliser of $L^*$ in $N$.

    The set $B \cdot B^{-1}$ has non-empty interior (since $B$ has positive measure in a locally compact group); the smallest closed subgroup containing $B \cdot B^{-1}$ contains $N_0$. The relation $L_x \supseteq L$ together with the equivariance $L_{nx} = n L_x n^{-1}$ for $n$ in the normaliser of $L^*$ extends to $N_0$ by continuity, giving $n L n^{-1} \subseteq L_{nx}$ for all $n \in N_0$. Since $L$ is the connected component of the stabiliser of $f$ and $f$ is $L_{nx}$-invariant, $n L n^{-1} \subseteq L$, i.e., $L$ is normalised by $N_0$.
\end{proof}

\begin{lem}\label{Lemma: Dimension reduction}
    With $f$ as fixed at the start of this section, there exist (up to an isomorphism of bundles of nilsystems) a subgroup $M \subseteq N_0$ such that:
    \begin{enumerate}
        \item $M$ stabilises $f$;
        \item $M \cap \Gamma$ is a lattice in $M$;
        \item $M$ is normal in $N_0$ and invariant under elements in the image of $\alpha$;
        \item the cocycle $\alpha_{N_0/M}$ is cohomologous to a cocycle taking values in a compact subgroup of $\Aut(N_0/M)$.
    \end{enumerate}
\end{lem}

\begin{proof}
    For $x \in Z^0(X)$, let $M_x \subseteq N_0$ denote the connected component of the identity in $\{u \in N_0 : u \cdot f(x, \cdot) = f(x, \cdot)\}$. By Lemma \ref{Lemma: stabilisers of functions on nilmanifolds}, $M_x$ is closed, $M_x \cap \Gamma$ is a lattice in $M_x$, and $M_x$ is normal in $N_0$. The $\SL_d(\mathbb{Z})$-invariance of $f$ together with the intertwining relation \eqref{Eq: intertwining relation} gives
    \begin{equation}\label{Eq: Invariance stabiliser}
        M_{\gamma \cdot x} = \alpha(\gamma, x)(M_x).
    \end{equation}

    Connected closed subgroups $H \subseteq N_0$ with $H \cap \Gamma$ a lattice form a countable set (Lemma \ref{Lemma: Measurable selection subnilmanifolds}); $x \mapsto M_x$ takes values in this set. By a measurable-selection argument, there exists a positive-measure $A \subseteq Z^0(X)$ on which $M_x = M$ is constant. Ergodicity of $\SL_d(\mathbb{Z})$ on $Z^0(X)$ together with \eqref{Eq: Invariance stabiliser} implies that $M_x$ takes values in the $\alpha$-orbit of $M$ for $\nu$-a.e.\ $x$; by countability of this orbit and measurable selection \cite[Thm. A.5]{ErgodicTheoryZimmer}, there exists a Borel map $\phi: Z^0(X) \to \langle \alpha \rangle \subset \Aut(N)$ (with $\langle \alpha \rangle$ the subgroup of $\Aut(N)$ generated by the image of $\alpha$) such that $\phi(x)(M) = M_x$ for $\nu$-a.e.\ $x$.

    The map $\Phi: X \to X$ defined by $\Phi(x, \bar n) := (x, \phi(x)(\bar n))$ is an isomorphism of bundles of nilsystems (since $\phi(x)$ commensurates $\Gamma$, having values in $\langle \alpha \rangle \subseteq \mathrm{Comm}_{\Aut(N)}(\Gamma)$). After replacing $f$ by $f \circ \Phi$ and $\alpha$ by its conjugate by $\phi$, the new fibrewise stabiliser is the constant group $M$, and items (1)--(3) hold.

    For (4): suppose for contradiction that $\alpha_{N_0/M}$ were not cohomologous to a cocycle taking values in a compact subgroup. By the compactness criterion (Lemma \ref{Lemma: Compactness criterion}), there exists $B \subseteq Z^0(X)$ of positive measure such that $E_{N_0/M}(x, B)$ is unbounded for $\nu$-a.e.\ $x$. Proposition \ref{Prop: using squeezed vectors} then gives, for $\nu$-a.e.\ $x$,
    $$ \dim \{u \in N_0 : u \cdot f(x, \cdot) = f(x, \cdot)\} > \dim M.$$
    But the connected component of this stabiliser is $M_x = M$, of dimension $\dim M$ — contradiction. Therefore $\alpha_{N_0/M}$ is cohomologous to a cocycle taking values in a compact subgroup of $\Aut(N_0/M)$.
\end{proof}

\subsection{Space of bounded directions}\label{Subsec: Space of bounded directions}

Under the standing assumption $d \geq 3$, we find a measurably-varying complement to $M$ in $N_0$. The key step is the following.

\begin{lem}\label{Lem: L is abelian}
    With the notation and hypotheses of Lemma \ref{Lemma: Dimension reduction}, $N_0/M$ is abelian.
\end{lem}

\begin{proof}
    Let $\alpha$ denote the intertwining cocycle (\S\ref{Subsection: Intertwining and action cocycles}), and let $\beta := \alpha|_{N_0}: \SL_d(\mathbb{Z}) \times Z^0(X) \to \Aut(N_0)$ be its restriction to the connected component $N_0$ (well-defined since $\alpha$ stabilises $N_0$ by continuity). By Zimmer's cocycle superrigidity theorem (e.g. \cite[Thm. 1.5]{FisherMargulis}), $\beta$ is cohomologous, in the Zariski-closed subgroup of $\Aut(N_0)$ stabilising $M$, to a cocycle of the form
    $$ (\gamma, x) \mapsto \sigma(\gamma) \, z(\gamma, x),$$
    where $\sigma: \SL_d(\mathbb{Z}) \to \Aut(N_0)$ is a group homomorphism with semi-simple Zariski closure, and $z$ takes values in a compact subgroup $K \subset \Aut(N_0)$ centralised by $\sigma(\SL_d(\mathbb{Z}))$.

    \emph{The subgroup $V$ of bounded directions.} Define
    $$ V := \{u \in N_0 : \sigma(\gamma)(u) = u \text{ for all } \gamma \in \SL_d(\mathbb{Z})\},$$
    the subgroup of $\sigma(\SL_d(\mathbb{Z}))$-fixed points in $N_0$. Since $\sigma$ has semi-simple Zariski closure, $V$ is a closed connected subgroup of $N_0$. Moreover $V$ is invariant under $K$, since $K$ centralises $\sigma(\SL_d(\mathbb{Z}))$.

    \emph{Claim: $V \cdot M = N_0$.} Since $\sigma$ has semi-simple Zariski closure, $\mathrm{Lie}(N_0)$ decomposes as a $\sigma$-module as $\mathrm{Lie}(V) \oplus W$, where $W$ has no non-zero $\sigma$-invariant vectors. By Lemma \ref{Lemma: Dimension reduction}(4), the projection $\beta_{N_0/M}$ is cohomologous to a compact-valued cocycle. As constant cocycles cohomologous to compact-valued cocycles must themselves have compact image, the projection of $\sigma$ to $\Aut(N_0/M)$ has compact image. Since $\sigma|_W$ has no fixed vectors, the projection of $W$ to $\mathrm{Lie}(N_0)/\mathrm{Lie}(M)$ must vanish — otherwise the image of $\sigma$ on this quotient would be non-compact. Hence $W \subseteq \mathrm{Lie}(M)$, so $\mathrm{Lie}(V) + \mathrm{Lie}(M) = \mathrm{Lie}(N_0)$, i.e.\ $V \cdot M = N_0$.

    \emph{Reduction to a non-abelian subgroup of $V$.} Suppose for contradiction that $N_0/M$ is non-abelian. Then $V$ is non-abelian (since $V$ surjects onto $N_0/M$). Decompose $V$ as follows. Let $Z := Z(V)$ be the centre of $V$, $p: V \twoheadrightarrow V^{\mathrm{ab}} := V/[V, V]$ the abelianisation, and choose a closed subgroup $L_0 \subseteq V^{\mathrm{ab}}$ complementing $p(Z)$, so that $V^{\mathrm{ab}} = p(Z) \oplus L_0$. Set $L := p^{-1}(L_0)$. Then $V = LZ$, $[L, L] = [V, V]$, and the centre of $L$ satisfies $Z(L) \subseteq [L, L]$ (a non-central element of $L$ projects non-trivially to $L_0$, while a central element of $L$ commutes with $LZ = V$, hence lies in $L \cap Z \subseteq L \cap p^{-1}(0) = [L, L]$).

    \emph{Choosing the level $k$.} Let $k \geq 0$ be such that $L \subseteq C_k(N)$ but $L \not\subseteq C_{k+1}(N)$, where $C_\bullet(N)$ denotes the lower central series of $N$. Since $V$ is non-abelian and connected, such $k$ exists.

    Recall $\alpha_k$ and $\rho_k$ from \S\ref{Subsection: Intertwining and action cocycles}. Let $K_k$ denote the maximal compact subgroup of $C_k(N)$. The subgroups $K_k$, $C_{k+1}(N)$ and $M \cap C_k(N)$ are normal in $C_k(N)$ and invariant under the range of $\alpha_k$, hence so is the product
    $$ Q := (M \cap C_k(N)) \cdot K_k \cdot C_{k+1}(N) \trianglelefteq C_k(N).$$
    Let $\pi: C_k(N) \twoheadrightarrow C_k(N)/Q$ be the natural projection.

    The image $\pi(L)$ is non-trivial: $L \subseteq C_k(N)$ but $L \not\subseteq C_{k+1}(N) \subseteq Q$. Moreover, $\pi(L)$ has trivial centre (centres of $L$ are killed by $\pi$ since $Z(L) \subseteq [L, L] \subseteq C_{k+1}(N)$ when $L \subseteq C_k(N)$) and contains no non-trivial compact subgroups (since $K_k \subseteq Q$). As a connected nilpotent Lie group with these properties, $\pi(L) \cong \mathbb{R}^n$ for some $n > 0$.

    \emph{Reaching a contradiction.} Set $\widetilde L := \pi(L)$. The cocycle $\pi \circ \rho_k$ takes values in $\Aut(\widetilde L) \ltimes \widetilde L$ (since the kernel $Q$ is $\rho_k$-invariant). The restriction $\alpha_k|_L$, which by construction of $L \subseteq V$ is cohomologous to a cocycle taking values in a compact subgroup of $\Aut(\widetilde L)$ (descending from the compactness on $V$ via the superrigidity decomposition above). Hence $\pi \circ \rho_k$ takes values, up to cohomology, in (compact) $\ltimes \widetilde L$, an amenable group.

    Since $\ASL_d(\mathbb{Z})$ has property (T) (using $d \geq 3$), Zimmer's theorem on  cocycles with amenable target \cite[Thm. 9.1.1]{ErgodicTheoryZimmer} reduces $\pi \circ \rho_k$ to a cocycle taking values in a compact subgroup. Restricting to $\mathbb{Z}^d \times Z_k(X)$, the restriction lands in $\widetilde L \cong \mathbb{R}^n$, which contains no non-trivial compact subgroups; by Corollary \ref{Cor: coho to compact indep of ambient}, the restriction is cohomologous to the trivial cocycle in $\widetilde L$.

    But the projection of $\rho_k|_{\mathbb{Z}^d \times Z_k(X)}$ to $C_k(N)/C_{k+1}(N)$ is precisely the Host--Kra structural cocycle $q_k$ of type $k+1$ (see \S\ref{Subsection: Intertwining and action cocycles}), and \cite[Ch. 13, Prop. 8]{HostKra} shows $q_k$ is non-trivial as a cocycle modulo coboundaries: more precisely, its further projection along $\pi$ to $\widetilde L$ is non-trivial whenever $\widetilde L \neq 0$ and $k \geq 1$. This contradicts triviality of $\pi \circ \rho_k|_{\mathbb{Z}^d \times Z_k(X)}$. Therefore $N_0/M$ is abelian.
\end{proof}

\subsection{Abelian connected component}

\begin{prop}\label{Prop: abelian connected component}
    Up to passing to a nilfactor of $X$ and a finite extension, we may assume that $N \cong \mathbb{Z}^d \ltimes A$ for a connected abelian Lie group $A$, that the subgroup of bounded directions $V$ (Lemma \ref{Lem: L is abelian}) coincides with the centre $Z(N)$, and that $\alpha$ acts trivially on $Z(N)$.
\end{prop}

\begin{proof}
    \emph{Reduction to abelian $N_0$.} By Lemma \ref{Lem: L is abelian}, $N_0/M$ is abelian, hence $[N_0, N_0] \subseteq M$. The subgroup $[N_0, N_0]$ is characteristic in $N_0$, hence normal in $N$. Since $f$ is $M$-invariant and $[N_0, N_0] \subseteq M$, $f$ factors through the quotient
    $$ X = Z^0(X) \rtimes_\rho N/\Gamma \;\twoheadrightarrow\; Z^0(X) \rtimes_{\rho'} N'/\Gamma',$$
    where $N' := N/[N_0, N_0]$, $\Gamma' := [N_0, N_0]\Gamma / [N_0, N_0]$, and $\rho'$ is the projected cocycle. Renaming, we work with the new bundle and write $N = \mathbb{Z}^d \ltimes A$ with $A$ a connected abelian Lie group.

    \emph{$V$ becomes central modulo $M$.} Let $V$ be the subgroup of bounded directions from Lemma \ref{Lem: L is abelian}. The Fisher--Margulis decomposition $\beta \sim \sigma \cdot z$ from the proof of Lemma \ref{Lem: L is abelian} is a cohomology between $\beta$ and $\sigma \cdot z$, realised by some Borel cochain $\phi: Z^0(X) \to \Aut(N_0) = \Aut(A)$. The twisted subgroup $\phi_x(V)$ is then $\beta$-equivariant: $\alpha(\gamma, x)(\phi_x(V)) = \phi_{\gamma x}(V)$ for $\nu$-a.e.\ $x$ and all $\gamma \in \SL_d(\mathbb{Z})$.

    Let $\pi: A \twoheadrightarrow A/M$ and let $\widetilde\alpha: \SL_d(\mathbb{Z}) \times Z^0(X) \to \Aut(A/M)$ be the unique cocycle with $\widetilde\alpha \circ \pi = \pi \circ \alpha|_A$. Define
    $$ b: \mathbb{Z}^d \times A \to A/M, \qquad b(t, a) := \pi([t, a]),$$
    where $[t, a] = (t \cdot a) \cdot a^{-1} \in A$ uses the $\mathbb{Z}^d$-action on $A$ in the semidirect product. Since $A$ is abelian and $\mathbb{Z}^d$ acts unipotently, $b(t, a)$ is polynomial in $t$ (of degree at most the nilpotency class minus one) and linear in $a$.

    For $v \in V$, applying $\widetilde\alpha(\gamma, x)$ to $b(t, \phi_x(v))$:
    \begin{align*}
        \widetilde\alpha(\gamma, x)\bigl(b(t, \phi_x(v))\bigr)
        &= \pi\bigl(\alpha(\gamma, x)([t, \phi_x(v)])\bigr) \\
        &= \pi\bigl([\gamma(t),\, \alpha(\gamma, x)(\phi_x(v))]\bigr) \\
        &= \pi\bigl([\gamma(t),\, \phi_{\gamma \cdot x}(v)]\bigr) \\
        &= b(\gamma(t), \phi_{\gamma \cdot x}(v)),
    \end{align*}
    using respectively the definition of $\widetilde\alpha$, that $\alpha$ is an automorphism (with $\alpha$ acting on $\mathbb{Z}^d \subset N$ via $\gamma$'s linear part), and the $\beta$-equivariance of $\phi_x(V)$.

    By Lemma \ref{Lemma: Dimension reduction}(4), $\widetilde\alpha$ is cohomologous to a cocycle taking values in a compact subgroup $K \subset \Aut(A/M)$, via some Borel cochain $\psi$. Hence $\widetilde\alpha(\gamma, x) \in \psi_{\gamma x}^{-1} K \psi_x$, and the orbit of $b(\cdot, \phi_x(v))$ under the $\SL_d(\mathbb{Z})$-action through $\widetilde\alpha$ is bounded in $\mathbb{R}_{\leq n}[X_1, \ldots, X_d]$. By Lemma \ref{Lemma: Invariant polynomials} below, the polynomial $t \mapsto b(t, \phi_x(v))$ is constant in $t$ for $\nu$-a.e.\ $x$. Since $b(0, a) = 0$, this gives $b(t, \phi_x(v)) = 0$ for all $t$, i.e.\ $\pi([t, \phi_x(v)]) = 0$ in $A/M$.

    \emph{Reducing to $V \subseteq Z(N)$.} The closed subgroup
    $$ H := \overline{\langle [t, \phi_x(v)] - [t', \phi_x(v)] : t, t' \in \mathbb{Z}^d,\, v \in V,\, x \in Z^0(X) \rangle} \subseteq A$$
    lies in $M$ (by the previous paragraph) and is normalised by $\mathbb{Z}^d$, hence normal in $N$ (since $A$ is abelian). Quotienting $X$ by $H$ preserves $f$-invariance and produces a new bundle in which $[t, \phi_x(v)] = 0$ for all $t \in \mathbb{Z}^d$ and $\phi_x(v) \in V$, i.e.\ $\phi_x(v) \in Z(N)$. Quotienting further by $M \cap Z(N)$, which is normal in $N$, we obtain $V = Z(N)$.

    \emph{$Z(N)$ is compact.} Under the standing reducedness assumption on $(N, \Gamma)$ (see the first paragraph of \S\ref{Section: Structure of invariant functions}), $Z(N) \cap \Gamma$ is trivial, so $Z(N)$ embeds into $N/\Gamma$ as a compact subgroup; hence $Z(N)$ is compact.

    \emph{Triviality of $\alpha|_{Z(N)}$.} The cocycle $\alpha' := \alpha(\gamma, x)|_{Z(N)}: \SL_d(\mathbb{Z}) \times Z^0(X) \to \Aut(Z(N))$ is cohomologous to a compact-valued cocycle by Lemma \ref{Lemma: Dimension reduction}(4) restricted to $Z(N) \subseteq N_0$. Since $Z(N)$ is a compact abelian Lie group, $\Aut(Z(N))$ is discrete; the closure of the image of $\alpha'$ equals the image. By Corollary \ref{Cor: coho to compact indep of ambient}, $\alpha'$ is cohomologous in its image to a cocycle taking values in a compact (hence finite) subgroup. Lemma \ref{Lemma: From compact to trivial} then provides a finite extension $\widetilde X \to X$ on which the lifted $\alpha'$ becomes a coboundary. Replacing $X$ by $\widetilde X$, we may assume $\alpha$ acts trivially on $Z(N)$.
\end{proof}

\begin{lem}\label{Lemma: Invariant polynomials}
    Let $n \geq 0$ and let $P: Z^0(X) \to \mathbb{R}_{\leq n}[X_1, \ldots, X_d]$ be a measurable map such that, for every $\gamma \in \SL_d(\mathbb{Z})$, every $t \in \mathbb{Z}^d$, and $\nu$-a.e.\ $x \in Z^0(X)$,
    $$ P(\gamma \cdot x)(t) = P(x)(\gamma(t)) \quad \text{(up to a uniformly bounded error)}.$$
    Then $P(x)$ is the constant polynomial for $\nu$-a.e.\ $x$.
\end{lem}

\begin{proof}
    Let $u \in \SL_d(\mathbb{Z})$ be the unipotent with $u(e_1) = e_1 + e_2$ and $u(e_i) = e_i$ for $i \geq 2$. Then for every $t = (t_1, \ldots, t_d)$ and every $l \in \mathbb{Z}$,
    $$ P(u^l \cdot x)(t_1, t_2, \ldots, t_d) = P(x)(t_1 + l t_2, t_2, \ldots, t_d) + O(1).$$
    Let $B \subseteq Z^0(X)$ be a measurable subset on which $P(x)$ is uniformly bounded. By Poincaré recurrence, for $\nu$-a.e.\ $x \in B$ there is an increasing sequence $(l_n)_{n \geq 0}$ with $u^{l_n} \cdot x \in B$. Hence $P(u^{l_n} \cdot x)(t)$ is bounded in $n$, and equals $P(x)(t_1 + l_n t_2, t_2, \ldots, t_d) + O(1)$. The polynomial $s \mapsto P(x)(s, t_2, \ldots, t_d)$ is bounded along an infinite arithmetic progression $\{t_1 + l_n t_2\}$ whenever $t_2 \neq 0$, hence constant in $s$.

    We deduce that $P(x)$ does not depend on $X_1$, for $\nu$-a.e.\ $x \in B$. Letting $B$ exhaust $Z^0(X)$, $P(x)$ is independent of $X_1$ for $\nu$-a.e.\ $x$. The same argument with the unipotent permuting other pairs of basis vectors shows $P(x)$ is independent of $X_k$ for each $k = 1, \ldots, d$. Hence $P(x)$ is the constant polynomial.
\end{proof}

\section{From nilsystems to polynomials}\label{Section: From nilsystems to polynomials}

We can now piece together the elements of the proof. The first subsection relates relative independence of marginals to a Bernoulli-thinning description of point processes, and connects inverse limits of factors to successive approximations of these processes. The second subsection then translates the structural results of the previous sections into the language of polynomials and exhibits the families of invariant point processes they produce.

\subsection{Approximation by return times}\label{Section: Approximation by return times}

Our first lemma turns an equality of marginals into a generative description of the point process.

\begin{lem}\label{Lemma: Relative mixing implies Bernoulli thinning}
    Let $\underline{X}$ be a $\mathbb{Z}^d$-invariant point process with associated dynamical system $X$, and let $\pi: X \to Y$ be a $\mathbb{Z}^d$-factor. Suppose that for every $r \geq 1$ and every distinct $t_1, \ldots, t_r \in \mathbb{Z}^d$,
    \begin{equation}\label{Eq: Equality marginals}
        \mathbb{E}[\mathbf{1}_{t_1 \in \underline{X}} \cdots \mathbf{1}_{t_r \in \underline{X}}] = \mathbb{E}\bigl[\mathbb{E}[\mathbf{1}_{t_1 \in \underline{X}} \mid Y] \cdots \mathbb{E}[\mathbf{1}_{t_r \in \underline{X}} \mid Y]\bigr].
    \end{equation}
    Then the law of $\underline{X}$ coincides with that of the point process obtained by drawing $y \in Y$ at random and keeping each $t \in \mathbb{Z}^d$ independently with probability $f((-t) \cdot y)$, where $f := \mathbb{E}[\mathbf{1}_{0 \in \underline{X}} \mid Y]$.
\end{lem}

\begin{proof}
    The marginal of $\underline{X}$ at distinct points $t_1, \ldots, t_r$ is the left-hand side of \eqref{Eq: Equality marginals}. Writing $\underline{Y}$ for the candidate point process described in the conclusion, its marginal at the same points is
    $$ \mathbb{P}(t_1, \ldots, t_r \in \underline{Y}) = \mathbb{E}\bigl[(t_1 \cdot f) \cdots (t_r \cdot f)\bigr] = \mathbb{E}\bigl[\mathbb{E}[\mathbf{1}_{t_1 \in \underline{X}} \mid Y] \cdots \mathbb{E}[\mathbf{1}_{t_r \in \underline{X}} \mid Y]\bigr],$$
    using the $\mathbb{Z}^d$-equivariance of conditional expectation. Hypothesis \eqref{Eq: Equality marginals} gives equality of marginals at every $r$-tuple, hence equality of laws.
\end{proof}

The second lemma describes how an inverse limit of factor systems is approximated, at the level of point processes, by the corresponding chain of approximations.

\begin{lem}\label{Lemma: Inverse limit approximation}
    Let $Y$ be an ergodic $\mathbb{Z}^d$-system with $f_1,f_2 \in L^\infty(Y)$ taking values in $[0, 1]$, and let $\underline{Y_1}, \underline{Y_2}$ be the point processes obtained by drawing $y \in Y$ at random and keeping each $t \in \mathbb{Z}^d$ independently with probability $f_1((-t) \cdot y)$ and $f_2((-t) \cdot y)$ respectively. 
    Then there is a coupling of $\underline{Y_1}$ and $\underline{Y_2}$ such that
    $$d(\underline{Y_1} \Delta \underline{Y_2}) \leq \|f_1-f_2 \|_2$$
    almost surely - where for a subset $A \subset \mathbb{Z}^d$, $d(A)$ denotes the \emph{density}
    $$ d(A) = \limsup_{N \rightarrow \infty} \frac{\vert A \cap [-N,N]^d\vert }{(2N+1)^d}.$$
\end{lem}

\begin{proof}
    Write $\delta = f_2 - f_1$, and for brevity write $f_i(t,y) := f_i((-t)\cdot y)$ and $\delta(t,y) := \delta((-t)\cdot y)$.
    To construct the coupling, draw $y \in Y$ at random. For each $t \in \mathbb{Z}^d$, independently keep $t$ with probability $f_1(t,y)$; this provides our realisation of $\underline{Y_1}$.
    
    Now, using the same $y \in Y$, define $\underline{Y_2}$ by independently redrawing each $t \in \mathbb{Z}^d$ as follows:
    \begin{itemize}
        \item conditional on $t \in \underline{Y_1}$, remove $t$ with probability 
        $\dfrac{-\min(\delta(t,y), 0)}{f_1(t,y)}$ (and $0$ if $f_1(t,y) = 0$);
        \item conditional on $t \notin \underline{Y_1}$, add $t$ with probability 
        $\dfrac{\max(\delta(t,y), 0)}{1-f_1(t,y)}$ (and $0$ if $f_1(t,y) = 1$).
    \end{itemize}
    Since $0 \leq f_1+\delta = f_2 \leq 1$, both fractions lie in $[0,1]$, so this is well-defined. A direct computation shows that, conditional on $y$,
    $$ \mathbb{P}(t \in \underline{Y_2} \mid y) = f_1(t,y) + \delta(t,y) = f_2(t,y), $$
    and that the events $\{t \in \underline{Y_2}\}_{t \in \mathbb{Z}^d}$ are conditionally independent given $y$. Hence $\underline{Y_2}$ has the desired law.
    
    We now compute the conditional probability of the symmetric difference. Conditional on $y$, for every $t \in \mathbb{Z}^d$,
    \begin{align*}
        \mathbb{P}(t \in \underline{Y_1}, t \notin \underline{Y_2} \mid y) 
        &= f_1(t,y) \cdot \frac{-\min(\delta(t,y), 0)}{f_1(t,y)} = -\min(\delta(t,y), 0), \\
        \mathbb{P}(t \notin \underline{Y_1}, t \in \underline{Y_2} \mid y) 
        &= (1-f_1(t,y)) \cdot \frac{\max(\delta(t,y), 0)}{1-f_1(t,y)} = \max(\delta(t,y), 0),
    \end{align*}
    (with the obvious interpretation when $f_1(t,y) \in \{0,1\}$). Summing,
    $$ \mathbb{P}(t \in \underline{Y_1} \Delta \underline{Y_2} \mid y) = |\delta(t,y)|. $$
    
    Conditional on $y$, the indicators $\big(\mathbf{1}_{t \in \underline{Y_1} \Delta \underline{Y_2}}\big)_{t \in \mathbb{Z}^d}$ are independent and bounded by $1$. By the strong law of large numbers (applied conditionally on $y$), almost surely,
    $$ \frac{1}{(2N+1)^d} \left| (\underline{Y_1} \Delta \underline{Y_2}) \cap [-N,N]^d \right| 
    - \frac{1}{(2N+1)^d} \sum_{t \in [-N,N]^d} |\delta(t,y)| \xrightarrow[N\to\infty]{} 0. $$
    By the pointwise ergodic theorem for $\mathbb{Z}^d$-actions along the Følner sequence $([-N,N]^d)_N$, and since $Y$ is ergodic,
    $$ \frac{1}{(2N+1)^d} \sum_{t \in [-N,N]^d} |\delta(t,y)| \xrightarrow[N\to\infty]{} \|\delta\|_1 \quad \text{for a.e. } y \in Y. $$
    Combining the two  yields $d(\underline{Y_1} \Delta \underline{Y_2}) = \|\delta\|_1$ almost surely, and in particular
    $$ d(\underline{Y_1} \Delta \underline{Y_2}) = \|\delta\|_1 \leq \|\delta\|_2 = \|f_1 - f_2\|_2, $$
    by Cauchy--Schwarz (since $Y$ is a probability space).
\end{proof}
\subsection{Invariant polynomial point processes}

We now describe a family of invariant point processes built from the algebraic data we have isolated in the previous sections. Throughout, we work in $\mathbb{T}^m$, the $m$-dimensional torus. For $k \geq 0$, write
$$ \mathbb{T}^m_{\leq k}[X_1, \ldots, X_d] = \bigl\{(t_1, \ldots, t_d) \mapsto \!\!\sum_{i_1 + \cdots + i_d \leq k} t_1^{i_1} \cdots t_d^{i_d} \, \tau_{(i_1, \ldots, i_d)} : \tau_{(i_1, \ldots, i_d)} \in \mathbb{T}^m \bigr\}$$
for the space of polynomial maps $\mathbb{Z}^d \to \mathbb{T}^m$ of total degree at most $k$. This space is itself a finite-dimensional torus, and admits two natural actions of $\ASL_d(\mathbb{Z})$: by precomposition $P \mapsto P \circ \gamma$, and by linear action on the coefficient torus.

One checks from the $1$-cocycle equation that the space of cocycles for a unipotent $\mathbb{Z}^d$-action on $\mathbb{T}^m$ embeds in $\mathbb{T}^m_{\leq k}[X_1, \ldots, X_d]$ for $k$ the nilpotency class of $\mathbb{Z}^d \ltimes \mathbb{T}^m$. 

The simplest invariant point processes arising from this picture are the following.

\begin{defn}\label{Def: simple polynomial process}
    Fix integers $k, m \geq 0$, let $\mu$ denote the Haar probability measure on the subspace of $\mathbb{T}^m_{\leq k+1}[X_1, \ldots, X_d]$ consisting of polynomials of degree exactly $k+1$, and let $f: \mathbb{T}^m \to [0, 1]$ be measurable. The associated \emph{simple polynomial point process} on $\mathbb{Z}^d$ is obtained as follows:
    \begin{itemize}
        \item draw a random polynomial $P$ according to $\mu$;
        \item independently across $t \in \mathbb{Z}^d$, keep $t$ with probability $f(P(t))$.
    \end{itemize}
\end{defn}

The measure $\mu$ is invariant under the precomposition action of $\ASL_d(\mathbb{Z})$, since this action preserves total degree, and the resulting point process is therefore $\ASL_d(\mathbb{Z})$-invariant.

The structural theorem we are aiming at says, conversely, that every $\ASL_d(\mathbb{Z})$-invariant point process arises from such a polynomial construction, after suitable enlargement. For degrees $k = 1$ and $k = 2$ the simple polynomial point processes (with $\mathbb{T}$ replaced by general compact abelian groups) already exhaust all examples. In higher degree, an additional intertwining action enters.

To describe the general construction, fix $k, m \geq 0$ and a closed subgroup $T \subseteq \mathbb{T}^m$, and consider the space $T_{\leq k}[X_1, \ldots, X_d]$ of polynomial maps of degree at most $k$ taking values in $T$. Choose a representation $\pi: \SL_d(\mathbb{Z}) \to \SL_m(\mathbb{R})$ together with a Borel section $s: \SL_m(\mathbb{R})/\SL_m(\mathbb{Z}) \to \SL_m(\mathbb{R})$. Define the Mackey cocycle
\begin{align*}
    \alpha: \SL_d(\mathbb{Z}) \times \SL_m(\mathbb{R})/\SL_m(\mathbb{Z}) &\longrightarrow \SL_m(\mathbb{Z}) \\
    (\gamma, \bar g) &\longmapsto s(\pi(\gamma) \bar g)^{-1} \, \pi(\gamma) \, s(\bar g),
\end{align*}
which lands in $\SL_m(\mathbb{Z})$ since the projection of the right-hand side to $\SL_m(\mathbb{R})/\SL_m(\mathbb{Z})$ is the identity. We assume that the range of $\alpha$ stabilises $T$, so that $\alpha$ acts on $T_{\leq k}[X_1, \ldots, X_d]$.

\begin{defn}\label{Def: general polynomial process}
    With the data $(k, m, T, \pi, s)$ as above, let $\nu$ be a the Haar probability measure on a closed subgroup of  $T_{\leq k}[X_1, \ldots, X_d]$ stable under  $P \mapsto \alpha(\gamma, x)(P) \circ \gamma^{-1}$  for every $\gamma \in \SL_d(\mathbb{Z})$, and let $f: T \to [0, 1]$ be measurable and invariant under the range of $\alpha$. The associated \emph{polynomial point process} on $\mathbb{Z}^d$ is obtained as follows:
    \begin{itemize}
        \item draw a random polynomial $\underline{P}$ according to $\nu$;
        \item independently across $t \in \mathbb{Z}^d$, keep $t$ with probability $f(\underline{P}(t))$.
    \end{itemize}
\end{defn}

The equivariance of $\underline{P}$ and the $\alpha$-invariance of $f$ together ensure that the process is $\ASL_d(\mathbb{Z})$-invariant.

\begin{thm}\label{Thm: All processes are polynomial}
    Every $\ASL_d(\mathbb{Z})$-invariant point process on $\mathbb{Z}^d$ is approximated by point processes arising in this way, for some choice of $k, m, T, \pi, \underline{P}$, and $f$ as above.
\end{thm}

The remainder of the section is devoted to the proof of Theorem \ref{Thm: All processes are polynomial}.

\subsection{Reducing the base space and an independence property}

The polynomial description of the previous subsection is completed by two further structural results: a measurability statement that locates the cocycle $\alpha$ over a concrete homogeneous space, and an independence statement between $\alpha$ and $b$. Together with the previous reductions, these produce the description of point processes from Definition \ref{Def: general polynomial process} in full generality.

\begin{lem}\label{Lemma: Measurability over homogeneous factor}
    Let $(X, \nu)$ be an $\SL_d(\mathbb{Z})$-ergodic system, $N$ a connected simply connected nilpotent Lie group, $\Gamma \subset N$ a lattice, and $\alpha: \SL_d(\mathbb{Z}) \times X \to \Aut(N, \Gamma)$ a cocycle valued in the stabiliser of $\Gamma$ in $\Aut(N)$. Up to a compact extension of $X$, there exist a homomorphism $\pi: \SL_d(\mathbb{Z}) \to \Aut(N)$ with image contained in a closed subgroup $H \subseteq \Aut(N)$, and an $\SL_d(\mathbb{Z})$-equivariant measurable map
    $$ \Phi: X \longrightarrow H/\Gamma_H, \qquad \Gamma_H := H \cap \Aut(N, \Gamma),$$
    such that $\alpha$ factors through $\Phi$ and is given by the constant cocycle $\alpha(\gamma, x) = \pi(\gamma)$.
\end{lem}

\begin{proof}
    By Zimmer's superrigidity theorem \cite[Thm. 1.5]{FisherMargulis}, $\alpha$ is cohomologous to a cocycle of the form $(\gamma, x) \mapsto \pi(\gamma) z(\gamma, x)$, where $\pi: \SL_d(\mathbb{Z}) \to \Aut(N)$ is a group homomorphism and $z$ takes values in a compact subgroup centralising $\pi(\SL_d(\mathbb{Z}))$. Lemma \ref{Lemma: From compact to trivial} provides a compact extension on which $z$ becomes trivial, so we may assume $\alpha(\gamma, x) = \phi(\gamma \cdot x)^{-1} \pi(\gamma) \phi(x)$ for some measurable $\phi: X \to \Aut(N)$.

    Since $\alpha(\gamma, x) \in \Aut(N, \Gamma)$, we have
    $$ \phi(\gamma \cdot x) \, \Aut(N, \Gamma) = \pi(\gamma) \, \phi(x) \, \Aut(N, \Gamma)$$
    in $\Aut(N)/\Aut(N, \Gamma)$. Hence the map $x \mapsto \phi(x) \, \Aut(N, \Gamma)$ is $\SL_d(\mathbb{Z})$-equivariant from $X$ to $\Aut(N)/\Aut(N, \Gamma)$, where $\SL_d(\mathbb{Z})$ acts on the target through $\pi$.

    By Benoist--Quint's theorem \cite{BenoistQuint}, the pushforward measure is supported on a closed orbit $H \cdot \Aut(N, \Gamma)$ and is the Haar probability measure on this orbit, for some closed subgroup $H \subseteq \Aut(N)$ containing $\pi(\SL_d(\mathbb{Z}))$. Identifying this orbit with $H / \Gamma_H$ for $\Gamma_H := H \cap \Aut(N, \Gamma)$ gives the desired map $\Phi$. The cohomologised cocycle $\alpha(\gamma, x) = \pi(\gamma)$ no longer depends on $x$ and is the announced constant cocycle.
\end{proof}

We now turn to the independence statement.

\begin{lem}\label{Lemma: Independence}
    With $\alpha$ and $b$ as in \S\ref{Subsection: From bundle of nilsystems to bundle of polynomials}, the random variables $\gamma \mapsto \alpha(\gamma, \cdot)$ and $t \mapsto b(t, \cdot)$ on $(X, \nu)$ are independent.
\end{lem}

\begin{proof}
    By Lemma \ref{Lemma: Measurability over homogeneous factor}, $\alpha$ factors through the $\SL_d(\mathbb{Z})$-equivariant map $\Phi: X \to H/\Gamma_H$ as a constant cocycle. The map $x \mapsto b(\cdot, x)$ takes values in the space of polynomial maps $\mathbb{Z}^d \to T$ satisfying \eqref{Eq: cocycle eq polynomial} from the previous subsection, which is a compact subgroup $K \subseteq \mathbb{T}^m_{\leq l}[X_1, \ldots, X_d]$ for some $l$ depending only on the nilpotency class of $N$.

    The joint law of $\alpha$ and $b$ is therefore supported on the homogeneous space $H/\Gamma_H \times K$, which we realise as $G/\Lambda$ with
    $$ G := H \times \bigl(\SL_d(\mathbb{Z}) \ltimes \mathbb{R}^m_{\leq l}[X_1, \ldots, X_d]\bigr), \qquad \Lambda := \Gamma_H \times \bigl(\SL_d(\mathbb{Z}) \ltimes \mathbb{Z}^m_{\leq l}[X_1, \ldots, X_d]\bigr).$$
    Here $\SL_d(\mathbb{Z})$ acts on $\mathbb{R}^m_{\leq l}[X_1, \ldots, X_d]$ by precomposition, and $H$ acts (commutingly) by its natural action on the coefficient torus $\mathbb{T}^m$. The action of $\SL_d(\mathbb{Z})$ on $G/\Lambda$ in question is by translation through the embedding
    $$ \pi \times \id: \SL_d(\mathbb{Z}) \hookrightarrow H \times \SL_d(\mathbb{Z}) \subset G,$$
    realising both the cocycle action on $\alpha$ (through $\pi$) and the precomposition action on $b$ (through $\id$).

    Since $d \geq 3$, the Zariski closure of the image of $\pi \times \id$ in $G$ is semisimple without compact factors. By Benoist--Quint \cite{BenoistQuint}, every ergodic $\SL_d(\mathbb{Z})$-invariant measure on $G/\Lambda$ is the Haar probability measure on a closed orbit $L \cdot y$ for some closed subgroup $L \subseteq G$ and some $y \in G/\Lambda$. Let $\mu$ be the pushforward to $G/\Lambda$ of $\nu$; then $\mu$ has this affine form for some $L$.

    The projection of $\mu$ to $H/\Gamma_H$ is the Haar measure (this is the marginal law of $\Phi_* \nu$), which forces the projection of $L$ to $H$ to be surjective. Likewise, the marginal in the $\SL_d(\mathbb{Z}) \ltimes \mathbb{R}^m_{\leq l}[X_1, \ldots, X_d]$ factor projects surjectively onto the discrete $\SL_d(\mathbb{Z})$-quotient. Combined, the projection of $L$ to $H \times \SL_d(\mathbb{Z})$ is surjective.

    Set $V := L \cap \mathbb{R}^m_{\leq l}[X_1, \ldots, X_d]$. The closed-orbit hypothesis forces $V$ to intersect $\mathbb{Z}^m_{\leq l}[X_1, \ldots, X_d]$ as a lattice, hence the image $T$ of $V$ in $\mathbb{T}^m_{\leq l}[X_1, \ldots, X_d]$ is a closed subgroup. The structure of $L$ as an extension of $H \times \SL_d(\mathbb{Z})$ by $V$ identifies $\mu$ as the product
    $$ \mu = \mu_{H/\Gamma_H} \otimes \mu_T,$$
    where $\mu_T$ is Haar on $T$. The product structure means precisely that $\alpha$ (whose law is $\mu_{H/\Gamma_H}$) and $b$ (whose law is $\mu_T$) are independent.
\end{proof}
\subsection{From bundle of nilsystems to bundle of polynomials}\label{Subsection: From bundle of nilsystems to bundle of polynomials}

Combining the results of \S\ref{Section: Approximation by return times} with the structural theorems of the previous sections, the classification of $\ASL_d(\mathbb{Z})$-invariant point processes reduces to understanding $\ASL_d(\mathbb{Z})$-ergodic systems of a specific shape. Let us make this reduction explicit.

Starting from an arbitrary $\ASL_d(\mathbb{Z})$-invariant point process $\underline X$ with associated dynamical system $X$, Theorem \ref{Theorem: Reduction to Host--Kra factors} approximates the relevant marginals by their projections onto the characteristic factors $Z^k(X)$, which are inverse limits of bundles of nilsystems over $Z^0(X)$. Proposition \ref{Prop: reduction to nilbundles} then upgrades the approximation to a commensurated nilfactor, and the results of \S\ref{Subsection:From commensurated to invariant nilfactors} extract from this an inverse limit of $\ASL_d(\mathbb{Z})$-equivariant nilfactors after passing to a finite extension. Finally, the structural results of \S\ref{Subsec: Space of bounded directions} reduce the connected component of the underlying nilpotent group to an abelian connected Lie group $A$, with a decomposition of the bounded directions $V = Z_N(A)$ as a direct factor; the reducedness assumption then forces $A$ to be a torus.

We may therefore assume that the system at play has the following form. Fix an ergodic base $X$, a torus $T = \mathbb{T}^m$, a nilpotent group $N = \mathbb{Z}^d \ltimes T$ in which $\mathbb{Z}^d$ acts on $T$ by automorphisms, a lattice $\Gamma \subset N$ projecting to a finite-index subgroup of $\mathbb{Z}^d$, and a $\mathbb{Z}^d$-cocycle $\rho: \mathbb{Z}^d \times X \to N$ giving rise to the bundle $X \rtimes_\rho N/\Gamma$. The structural results give in addition a connected $\SL_d(\mathbb{Z})$-invariant subgroup $M \subseteq T$ together with a decomposition $T = M \oplus Z_N(T)$, where $Z_N(T)$ denotes the subgroup of elements of $T$ central in $N$. The function describing the point process is a measurable $f: M \backslash N/\Gamma \to [0, 1]$.

Standard results from homogeneous dynamics let us upgrade the regularity of this picture step by step. We begin by replacing $\Gamma$ with a more convenient lattice.

\begin{lem}\label{Lemma: convenient lattice}
    With notation as above, we may assume $\Gamma \cap T = \{0\}$, and there exists a lattice $\Gamma_0 \subset N$ containing $\Gamma$ whose projection to $\mathbb{Z}^d = N/T$ is surjective. Moreover, $\Gamma_0$ is invariant under all elements of $\Aut(N)$ stabilising $\Gamma$.
\end{lem}

\begin{proof}[Proof sketch]
    The derived subgroup of $\Gamma$ lies in $T$ and is discrete, hence finite. A standard argument shows $\Gamma$ contains a finite-index abelian subgroup isomorphic to $\mathbb{Z}^d$ that intersects $T$ trivially and is characteristic in $\Gamma$; replacing $\Gamma$ by this subgroup yields the first claim. For the second, take $\Gamma_0$ to be generated by all $u \in N$ with $u^m \in \Gamma$ for some sufficiently large $m$; this group is discrete, projects surjectively onto $\mathbb{Z}^d$, and is invariant under any automorphism of $N$ stabilising $\Gamma$.
\end{proof}

Fix $\Gamma$ and $\Gamma_0$ as in Lemma \ref{Lemma: convenient lattice}. Although $\Gamma$ itself need not project surjectively onto $\mathbb{Z}^d$, we can construct measurable sections of the projection $\Gamma_0 \to \mathbb{Z}^d$ that intertwine well with the $\ASL_d(\mathbb{Z})$-action.

\begin{lem}\label{Lemma: intertwining sections}
    Up to passing to a finite extension of $X$, there is a measurable map $\sigma: \mathbb{Z}^d \times X \to \Gamma_0$ satisfying:
    \begin{enumerate}
        \item for every $x \in X$, the map $t \mapsto \sigma(t, x)$ is a section of the projection $p: \Gamma_0 \to \mathbb{Z}^d$;
        \item for every $x \in X$, $t \in \mathbb{Z}^d$, and $t' \in \Gamma$, $\sigma(t + p(t'), x) = \sigma(t, x) \cdot t'$;
        \item for every $\gamma \in \SL_d(\mathbb{Z})$, $t \in \mathbb{Z}^d$, and $\nu$-a.e.\ $x \in X$, $\alpha(\gamma, x)(\sigma(t, x)) = \sigma(\gamma(t), \gamma \cdot x)$.
    \end{enumerate}
\end{lem}

\begin{proof}
    Let $\Sigma$ denote the set of sections of $p: \Gamma_0 \to \mathbb{Z}^d$ satisfying (2). Then $\Sigma$ is finite (it is a torsor over the finite group $\Hom(\mathbb{Z}^d/p(\Gamma), \Gamma_0 \cap T)$). For every $x \in X$, $\gamma \in \SL_d(\mathbb{Z})$, and $\sigma_1 \in \Sigma$, the composition $\alpha(\gamma, x) \circ \sigma_1 \circ \gamma^{-1}$ is again in $\Sigma$, defining an action of $\SL_d(\mathbb{Z})$ on $X \times \Sigma$ by $\gamma \cdot (x, \sigma) := (\gamma \cdot x, \alpha(\gamma, x) \circ \sigma \circ \gamma^{-1})$. The corresponding cocycle takes values in the finite symmetric group of $\Sigma$, and Lemma \ref{Lemma: From compact to trivial} produces a finite extension on which it becomes a coboundary. The trivialised section is the desired $\sigma$.
\end{proof}

Fix such a measurably-varying section $\sigma$. We can now write the cocycle $\rho$ explicitly. For $\nu$-a.e.\ $x \in X$ and every $t \in \mathbb{Z}^d$, write
$$ \rho(t, x) = \sigma(t, x) \cdot b(t, x)$$
for the unique $b(t, x) \in T$ that makes this identity hold. The cocycle equation $\rho(t_1 + t_2, x) = \rho(t_1, x) \rho(t_2, x)$ together with the definition of $\sigma$ then yields, after a short computation,
\begin{equation}\label{Eq: cocycle eq polynomial}
    c(t_1, t_2, x) = b(t_1 + t_2, x) - b(t_1, x) - t_1 \cdot b(t_2, x),
\end{equation}
where $c(t_1, t_2, x) := \sigma(t_1 + t_2, x)^{-1} \sigma(t_1, x) \sigma(t_2, x) \in \Gamma_0 \cap T$. The map $c(\cdot, \cdot, x)$ is a $2$-cocycle on $\mathbb{Z}^d$ valued in the finite group $\Gamma_0 \cap T$, and depends on $t_1, t_2$ only through their classes modulo $p(\Gamma)$. Hence $c$ factors through $(\mathbb{Z}/m\mathbb{Z})^d \times (\mathbb{Z}/m\mathbb{Z})^d$ for some $m$ depending only on $\Gamma_0 \cap T$ and the $\mathbb{Z}^d$-action, and is polynomial in its arguments of degree bounded in terms of the nilpotency class of $N$ alone.

This polynomial structure of $c$ propagates to $b$.

\begin{lem}\label{Lemma: b is polynomial}
    For $\nu$-a.e.\ $x \in X$, the map $t \mapsto b(t, x)$ is a polynomial map $\mathbb{Z}^d \to T$ of degree bounded in terms of the nilpotency class of $N$ alone.
\end{lem}

The proof is a direct manipulation of \eqref{Eq: cocycle eq polynomial}, exploiting that $c(\cdot, \cdot, x)$ is polynomial and that the $\mathbb{Z}^d$-action on $T$ is unipotent, so finite differences of $b$ reduce its degree.

We now turn to the $\SL_d(\mathbb{Z})$-equivariance. Since $\rho$ takes values in the affine group $\Aut(N) \ltimes N$ and $d \geq 3$, the restriction $\rho|_{\SL_d(\mathbb{Z}) \times X}$ has reductive algebraic hull \cite[Thm. 1.5]{FisherMargulis}. Standard cocycle reduction provides a measurable $\phi: X \to N$ such that
$$ \phi(\gamma \cdot x)^{-1} \, \rho(\gamma, x) \, \phi(x) \in \Aut(N)$$
for every $\gamma \in \SL_d(\mathbb{Z})$ and $\nu$-a.e.\ $x \in X$. Replacing $\rho$ by this conjugate, we may assume $\rho|_{\SL_d(\mathbb{Z}) \times X}$ takes values in $\Aut(N)$ from the outset.

For $\gamma \in \SL_d(\mathbb{Z})$, $t \in \mathbb{Z}^d$, and $\nu$-a.e.\ $x \in X$, applying the cocycle equation in two ways and using that $\mathbb{Z}^d$ acts trivially on $X$,
$$ \rho(\gamma, x) \, \rho(t, x) = \rho(\gamma t, x) = \rho(\gamma(t) \gamma, x) = \rho(\gamma(t), \gamma \cdot x) \, \rho(\gamma, x).$$
Substituting $\rho(t, x) = \sigma(t, x) b(t, x)$ and using the equivariance of $\sigma$ (Lemma \ref{Lemma: intertwining sections}(3)) gives
\begin{equation}\label{Eq: Invariance eq for cocycle end}
    \alpha(\gamma, x)(b(t, x)) = b(\gamma(t), \gamma \cdot x).
\end{equation}

We can now describe the action of $\rho(t, x)$ on a Haar-random point of $N/\Gamma_0$. The choice of section $\sigma(\cdot, x)$ provides an identification $T \times \mathbb{Z}^d/p(\Gamma_0) \simeq N/\Gamma_0$ via $(\tau, \bar t) \mapsto \tau \, \sigma(t, x) \, \Gamma_0$ for any representative $t$ of $\bar t$. Drawing $\bar n \in N/\Gamma_0$ at Haar random is then equivalent to drawing $t_{\mathrm{rand}} \in \mathbb{Z}^d/p(\Gamma_0)$ uniformly and $\tau \in T$ at Haar random independently. For such $\bar n$,
$$ \rho(t, x) \cdot \bigl(\tau \, \sigma(t_{\mathrm{rand}}, x) \, \Gamma_0\bigr) = \bigl(b(t, x) + t \cdot \tau + c(t_{\mathrm{rand}}, t, x)\bigr) \, \sigma(t + t_{\mathrm{rand}}, x) \, \Gamma_0,$$
using the polynomial cocycle equation \eqref{Eq: cocycle eq polynomial} for the rearrangement.

The right-hand side displays $\rho(t, x) \cdot \bar n$ as a polynomial expression in $t$ with coefficients in $T \times \mathbb{Z}^d/p(\Gamma_0) \simeq N/\Gamma_0$, of degree bounded in terms of the nilpotency class of $N$ alone. The $\SL_d(\mathbb{Z})$-equivariance properties \eqref{Eq: Invariance eq for cocycle end} of $b$ and the analogous equivariance of $c$, together with the identity $\rho(\gamma, x)(t \cdot \tau) = \gamma(t) \cdot \alpha(\gamma, x)(\tau)$, give that this random polynomial $P(t, x)$ satisfies
$$ \alpha(\gamma, x)(P(t, x)) = P(\gamma(t), \gamma \cdot x).$$
Furthermore, the random variables $\alpha(\gamma, \cdot)$ and $P(t, \cdot)$ are independent by the independence property established in the next subsection (Lemma \ref{Lemma: Independence}).

This is precisely the description of $\rho(t, x) \cdot \bar n$ as the value at $t$ of an $\alpha$-equivariant random polynomial, matching the framework of Definition \ref{Def: general polynomial process}. We can now assemble the proof of the main theorem.

\begin{proof}[Proof of Theorem \ref{Thm: All processes are polynomial}]
    Let $\underline X$ be an $\ASL_d(\mathbb{Z})$-invariant point process on $\mathbb{Z}^d$, with associated dynamical system $X$ and indicator function $\mathbf{1}_{0 \in \underline X} \in L^\infty(X)$. The reduction described at the beginning of this subsection yields, after passing to a finite extension and to an $\ASL_d(\mathbb{Z})$-equivariant nilfactor $X \to Y$, a system of the standing form $Y = X' \rtimes_\rho N/\Gamma$ with $N = \mathbb{Z}^d \ltimes T$, $T = \mathbb{T}^m$, and a measurable function $f: M \backslash N/\Gamma \to [0, 1]$ such that $\mathbb{E}[\mathbf{1}_{0 \in \underline X} \mid Y] = f$.

    By Theorem \ref{Theorem: Reduction to Host--Kra factors}, the marginals of $\underline X$ at any finite collection of distinct points satisfy
    $$ \mathbb{E}[\mathbf{1}_{t_1 \in \underline X} \cdots \mathbf{1}_{t_r \in \underline X}] = \mathbb{E}\bigl[\mathbb{E}[\mathbf{1}_{t_1 \in \underline X} \mid Y] \cdots \mathbb{E}[\mathbf{1}_{t_r \in \underline X} \mid Y]\bigr]$$
    in the limit along the inverse system. Lemma \ref{Lemma: Relative mixing implies Bernoulli thinning} applied at each level of the inverse limit, combined with Lemma \ref{Lemma: Inverse limit approximation}, identifies the law of $\underline X$ with the limit of point processes obtained by drawing $y \in Y$ at random and keeping each $t \in \mathbb{Z}^d$ independently with probability $f((-t) \cdot y)$.

    Lemma \ref{Lemma: b is polynomial} and the surrounding discussion express the action of $\rho(t, x)$ on a Haar-random $\bar n \in N/\Gamma_0$ as the value at $t$ of a random polynomial $P(t, x)$ taking values in $T_{\leq l}[X_1, \ldots, X_d]$, where $l$ depends only on the nilpotency class of $N$. The polynomial satisfies the equivariance property $\alpha(\gamma, x)(P(t, x)) = P(\gamma(t), \gamma \cdot x)$ from \eqref{Eq: Invariance eq for cocycle end}.

    By Lemma \ref{Lemma: Measurability over homogeneous factor}, the cocycle $\alpha$ factors through an $\SL_d(\mathbb{Z})$-equivariant map to a homogeneous space $H/\Gamma_H$, with $\alpha$ becoming the constant cocycle $\pi$. By Lemma \ref{Lemma: Independence}, the random variables $x \mapsto \alpha(\gamma, x)$ and $x \mapsto P(t, x)$ are independent. The data $(l, m, T, \pi, P, f)$ now matches the requirements of Definition \ref{Def: general polynomial process}: $P$ is the  random polynomial $\underline P$ whose law is a Haar measure and $\alpha$-invariant, and $f$ is the measurable function on $T$ (after identifying $M \backslash T$ with the relevant quotient and triviality of the action of $\alpha$ on $M \backslash T$).

    The point process associated with this data, by Definition \ref{Def: general polynomial process}, has the same marginals as $\underline X$ at every finite collection of points, hence the same law. This concludes the proof.
\end{proof}

\bibliographystyle{alpha}
\bibliography{ref}

@inproceedings{de1937prevision,
  title={La pr{\'e}vision: ses lois logiques, ses sources subjectives},
  author={De Finetti, Bruno},
  booktitle={Annales de l'institut Henri Poincar{\'e}},
  volume={7},
  pages={1--68},
  year={1937}
}

@article{diaconis1977finite,
  title={Finite forms of de Finetti's theorem on exchangeability},
  author={Diaconis, Persi},
  journal={Synthese},
  volume={36},
  number={2},
  pages={271--281},
  year={1977},
  publisher={Springer}
}

@article{diaconis1980finite,
  title={Finite exchangeable sequences},
  author={Diaconis, Persi and Freedman, David},
  journal={The Annals of Probability},
  pages={745--764},
  year={1980},
  publisher={JSTOR}
}

@article{gavalakis2021information,
  title={An information-theoretic proof of a finite de Finetti theorem},
  author={Gavalakis, Lampros and Kontoyiannis, Ioannis},
  journal={Electronic Communications in Probability},
  volume={26},
  pages={1--5},
  year={2021},
  publisher={The Institute of Mathematical Statistics and the Bernoulli Society}
}

@article{kirsch2019elementary,
  title={An elementary proof of de Finetti’s theorem},
  author={Kirsch, Werner},
  journal={Statistics \& Probability Letters},
  volume={151},
  pages={84--88},
  year={2019},
  publisher={Elsevier}
}

@book{HostKra,
 author = {Host, Bernard and Kra, Bryna},
 title = {Nilpotent structures in ergodic theory},
 fseries = {Mathematical Surveys and Monographs},
 series = {Math. Surv. Monogr.},
 issn = {0076-5376},
 volume = {236},
 isbn = {978-1-4704-4780-9; 978-1-4704-5061-8},
 year = {2018},
 publisher = {Providence, RI: American Mathematical Society (AMS)},
 language = {English},
 doi = {10.1090/surv/236},
 zbMATH = {7007704},
 Zbl = {1433.37001}
}

@book{ErgodicTheoryZimmer,
 author = {Zimmer, Robert J.},
 title = {Ergodic theory and semisimple groups},
 fseries = {Monographs in Mathematics},
 series = {Monogr. Math., Basel},
 issn = {1017-0480},
 volume = {81},
 year = {1984},
 publisher = {Birkh{\"a}user, Cham},
 language = {English},
 zbMATH = {3911340},
 Zbl = {0571.58015}
}

@article{ZieglerNonConventional,
 author = {Ziegler, Tamar},
 title = {Universal characteristic factors and {Furstenberg} averages},
 fjournal = {Journal of the American Mathematical Society},
 journal = {J. Am. Math. Soc.},
 issn = {0894-0347},
 volume = {20},
 number = {1},
 pages = {53--97},
 year = {2007},
 language = {English},
 doi = {10.1090/S0894-0347-06-00532-7},
 zbMATH = {5120455},
 Zbl = {1198.37014}
}

@article{Austin,
 author = {Austin, Tim},
 title = {Extensions of probability-preserving systems by measurably-varying homogeneous spaces and applications},
 fjournal = {Fundamenta Mathematicae},
 journal = {Fundam. Math.},
 issn = {0016-2736},
 volume = {210},
 number = {2},
 pages = {133--206},
 year = {2010},
 language = {English},
 doi = {10.4064/fm210-2-3},
 zbMATH = {5834943},
 Zbl = {1206.28023}
}

@incollection{FisherMargulis,
 author = {Fisher, David and Margulis, G. A.},
 title = {Local rigidity for cocycles},
 booktitle = {Surveys in differential geometry. Vol. VIII: Lectures on geometry and topology held in honor of Calabi, Lawson, Siu, and Uhlenbeck at Harvard University, Cambridge, MA, USA, May 3--5, 2002},
 isbn = {1-57146-114-0},
 pages = {191--234},
 year = {2003},
 publisher = {Somerville, MA: International Press},
 language = {English},
 zbMATH = {2070183},
 Zbl = {1062.22044}
}

@article{RaghunathanDiscrete,
 author = {Raghunathan, M. S.},
 title = {Discrete subgroups of {Lie} groups},
 fjournal = {The Mathematics Student},
 journal = {Math. Stud.},
 issn = {0025-5742},
 volume = {2007},
 pages = {59--70},
 year = {2007},
 language = {English},
 zbMATH = {5575404},
 Zbl = {1193.22011}
}

@misc{JamneshanMachado,
 author = {Asgar Jamneshan and Simon Machado},
 title = {The non-ergodic {Host}-{Kra}-{Ziegler} structure theorem for {$\mathbb{Z}^d$}-actions via measurable selections},
 year = {2026},
 howpublished = {Preprint, {arXiv}:2601.09553 [math.{DS}] (2026)},
 url = {https://arxiv.org/abs/2601.09553},
 arXiv = {arXiv:2601.09553}
}

@article{Leibman,
 author = {Leibman, A.},
 title = {Pointwise convergence of ergodic averages for polynomial actions of {{\(\mathbb{Z}^d\)}} by translations on a nilmanifold},
 fjournal = {Ergodic Theory and Dynamical Systems},
 journal = {Ergodic Theory Dyn. Syst.},
 issn = {0143-3857},
 volume = {25},
 number = {1},
 pages = {215--225},
 year = {2005},
 language = {English},
 doi = {10.1017/S0143385704000227},
 zbMATH = {2174669},
 Zbl = {1080.37004}
}

@article{TaoZieglerConcat,
 author = {Tao, Terence and Ziegler, Tamar},
 title = {Concatenation theorems for anti-{Gowers}-uniform functions and {Host}-{Kra} characteristic factors},
 fjournal = {Discrete Analysis},
 journal = {Discrete Anal.},
 issn = {2397-3129},
 volume = {2016},
 pages = {60},
 note = {Id/No 13},
 year = {2016},
 language = {English},
 doi = {10.19086/da.850},
 zbMATH = {6637026},
 Zbl = {1400.11028}
}

@article{FrantzikinakisKucaJoint,
 author = {Frantzikinakis, Nikos and Kuca, Borys},
 title = {Joint ergodicity for commuting transformations and applications to polynomial sequences},
 fjournal = {Inventiones Mathematicae},
 journal = {Invent. Math.},
 issn = {0020-9910},
 volume = {239},
 number = {2},
 pages = {621--706},
 year = {2025},
 language = {English},
 doi = {10.1007/s00222-024-01313-w},
 zbMATH = {7968253},
 Zbl = {1566.37007}
}

@article{MackeyPointRealization,
 author = {Mackey, G. W.},
 title = {Point realizations of transformation groups},
 fjournal = {Illinois Journal of Mathematics},
 journal = {Ill. J. Math.},
 issn = {0019-2082},
 volume = {6},
 pages = {327--335},
 year = {1962},
 language = {English},
 zbMATH = {3284736},
 Zbl = {0178.17203}
}

@book{PropertyT,
 author = {Bekka, Bachir and de la Harpe, Pierre and Valette, Alain},
 title = {Kazhdan's property},
 fseries = {New Mathematical Monographs},
 series = {New Math. Monogr.},
 volume = {11},
 isbn = {978-0-521-88720-5},
 year = {2008},
 publisher = {Cambridge: Cambridge University Press},
 language = {English},
 zbMATH = {5284649},
 Zbl = {1146.22009}
}

@book{KechrisDescriptive,
 author = {Kechris, Alexander S.},
 title = {Classical descriptive set theory},
 fseries = {Graduate Texts in Mathematics},
 series = {Grad. Texts Math.},
 issn = {0072-5285},
 volume = {156},
 isbn = {3-540-94374-9},
 year = {1995},
 publisher = {Berlin: Springer-Verlag},
 language = {English},
 zbMATH = {722611},
 Zbl = {0819.04002}
}

@article{BenoistQuint,
 author = {Benoist, Yves and Quint, Jean-Fran{\c{c}}ois},
 title = {Stationary measures and invariant subsets of homogeneous spaces. {II}},
 fjournal = {Journal of the American Mathematical Society},
 journal = {J. Am. Math. Soc.},
 issn = {0894-0347},
 volume = {26},
 number = {3},
 pages = {659--734},
 year = {2013},
 language = {English},
 doi = {10.1090/S0894-0347-2013-00760-2},
 zbMATH = {6168604},
 Zbl = {1268.22011}
}

@article{guiheneuf2017minkowski,
  title={A Minkowski theorem for quasicrystals},
  author={Guih{\'e}neuf, Pierre-Antoine and Joly, {\'E}milien},
  journal={Discrete \& Computational Geometry},
  volume={58},
  number={3},
  pages={596--613},
  year={2017},
  publisher={Springer}
}

@article{ruhr2023classification,
  title={Classification and statistics of cut-and-project sets},
  author={R{\"u}hr, Ren{\'e} and Smilansky, Yotam and Weiss, Barak},
  journal={Journal of the European Mathematical Society},
  volume={26},
  number={9},
  pages={3575--3638},
  year={2023}
}

@article{howe1979asymptotic,
  title={Asymptotic properties of unitary representations},
  author={Howe, Roger E and Moore, Calvin C},
  journal={Journal of Functional Analysis},
  volume={32},
  number={1},
  pages={72--96},
  year={1979},
  publisher={Elsevier}
}

@book{daley2003introduction,
  title={An introduction to the theory of point processes: volume I: elementary theory and methods},
  author={Daley, Daryl J and Vere-Jones, David},
  year={2003},
  publisher={Springer}
}

@book{baake2013aperiodic,
  title={Aperiodic Order. Volume 1: A Mathematical Invitation},
  author={Baake, Michael and Grimm, Uwe},
  year={2013},
  publisher={Cambridge University Press}
}

@book{senechal1995quasicrystals,
  title={Quasicrystals and Geometry},
  author={Senechal, Marjorie},
  year={1995},
  publisher={Cambridge University Press}
}

@article{foreman2011conjugacy,
  title={The conjugacy equivalence relation of ergodic measure-preserving transformations is not Borel},
  author={Foreman, Matthew and Rudolph, Daniel J and Weiss, Benjamin},
  journal={Annals of Mathematics},
  pages={1529--1586},
  year={2011},
  publisher={JSTOR}
}

@article{hjorth2001invariants,
  title={On invariants for measure preserving transformations},
  author={Hjorth, Greg},
  journal={Fundamenta Mathematicae},
  volume={169},
  number={1},
  pages={51--84},
  year={2001}
}
\end{document}